\documentclass{amsart}
\usepackage{amssymb, amsmath, amsthm, graphics, comment, xspace, enumerate}
\usepackage{color}
\newtheorem{thm}{Theorem}[section]
\newtheorem{cor}[thm]{Corollary}
\newtheorem{lem}[thm]{Lemma}
\newtheorem{prop}[thm]{Proposition}
\theoremstyle{definition}
\newtheorem{defn}[thm]{Definition}
\newtheorem{example}[thm]{Example}
\theoremstyle{remark}
\newtheorem{rem}[thm]{Remark}
\numberwithin{equation}{section}

\begin{document}
\title[Multi-dimensional almost periodic type functions and applications]{Multi-dimensional almost periodic type functions and applications}

\author{A. Ch\'avez}
\address{Departamento de
Matem\'aticas, Facultad de Ciencias F\' isicas Y Matem\'aticas, Universidad Nacional de Trujillo, Trujillo, Chile}
\email{ajchavez@unitru.edu.pe}

\author{K. Khalil}
\address{Department of Mathematics, Faculty of Sciences Semlalia, Cadi Ayyad University,
 B.P. 2390, 40000 Marrakesh, Morocco}
\email{kamal.khalil.00@gmail.com}

\author{M. Kosti\' c}
\address{Faculty of Technical Sciences,
University of Novi Sad,
Trg D. Obradovi\' ca 6, 21125 Novi Sad, Serbia}
\email{marco.s@verat.net}

\author{M. Pinto}
\address{Departamento de
Matem\'aticas, Facultad de Ciencias, Universidad de Chile, Santiago de Chile, Chile}
\email{pintoj.uchile@gmail.com}

{\renewcommand{\thefootnote}{} \footnote{2010 {\it Mathematics
Subject Classification.} 42A75, 43A60, 47D99.
\\ \text{  }  \ \    {\it Key words and phrases.} $({\mathrm R},{\mathcal B})$-Multi-almost periodic type functions, $({\mathrm R}_{X},{\mathcal B})$-multi-almost periodic type functions, Bohr ${\mathcal B}$-almost periodic type functions, composition principles,
abstract Volterra integro-differential equations.
\\  \text{  }  
Marko Kosti\' c is partially supported by grant 451-03-68/2020/14/200156 of Ministry
of Science and Technological Development, Republic of Serbia.
Manuel Pinto is partially supported by Fondecyt 1170466.}}

\begin{abstract}
In this paper, we analyze multi-dimensional $({\mathrm R}_{X},{\mathcal B})$-almost periodic type functions and multi-dimensional Bohr ${\mathcal B}$-almost periodic type functions.
The main structural characterizations and composition principles for the introduced classes of almost periodic functions are established. 
Several applications of our abstract theoretical results to
the abstract Volterra integro-differential equations in Banach spaces are provided, as well.
\end{abstract}
\maketitle

\begin{center}
{\sc Contents}
\end{center}

\vspace{0.75cm}

\noindent 1. Introduction and preliminaries\\
\text{\ \ \ \ \ } 1.1. Almost periodic type functions on ${\mathbb R}^{n}$\\
\text{\ \ \ \ \ } 1.2. Applications\\

\noindent 2. $({\mathrm R}_{X},{\mathcal B})$-Multi-almost periodic type functions and Bohr ${\mathcal B}$-almost periodic type functions\\
\text{\ \ \ \ \ } 2.1. Strongly ${\mathcal B}$-almost periodic functions\\
\text{\ \ \ \ \ }  2.2. ${\mathbb D}$-asymptotically $({\mathrm R}_{X},{\mathcal B})$-multi-almost periodic type functions\\
\text{\ \ \ \ \ }  2.3. Differentiation and integration of $({\mathrm R}_{X},{\mathcal B})$-multi-almost periodic functions\\
\text{\ \ \ \ \ }  2.4. Composition theorems for $({\mathrm R},{\mathcal B})$-multi-almost periodic type functions\\
\text{\ \ \ \ \ }  2.5. Invariance of $({\mathrm R},{\mathcal B})$-multi-almost periodicity under the actions of convolution products\\

\noindent 3. Examples and applications to the abstract Volterra integro-differential equations\\
\text{\ \ \ \ \ } 3.1. Application to nonautonomous retarded functional evolution equations\\

\noindent 4. Appendix\\
\text{\ \ \ \ \ } 4.1. $n$-Parameter strongly continuous semigroups\\
\text{\ \ \ \ \ } 4.2. Multivariate trigonometric polynomials and approximations of periodic functions of several real variables\\

\noindent References

\newpage

\section{Introduction and preliminaries}

The notion of almost periodicity was introduced by the Danish mathematician H. Bohr around 1924-1926 and later generalized by many other authors (for further information regarding almost periodic functions, we refer the reader to the research monographs \cite{besik}, \cite{diagana}, \cite{fink}, \cite{gaston}, \cite{nova-mono}-\cite{nova-man}, \cite{188}-\cite{18} and \cite{30}). Let $I$ be either ${\mathbb R}$ or $[0,\infty),$ and let $f : I \rightarrow X$ be a given continuous function, where $X$ is a complex Banach space equipped with the norm $\| \cdot \|$. Given $\varepsilon>0,$ we call $\tau>0$ a $\varepsilon$-period for $f(\cdot)$ if and only if\index{$\varepsilon$-period}
$
\| f(t+\tau)-f(t) \| \leq \varepsilon,$ $ t\in I.
$
The set consisting of all $\varepsilon$-periods for $f(\cdot)$ is denoted by $\vartheta(f,\varepsilon).$ The function $f(\cdot)$ is said to be almost periodic if and only if for each $\varepsilon>0$ the set $\vartheta(f,\varepsilon)$ is relatively dense in $[0,\infty),$ which means that
there exists $l>0$ such that any subinterval of $[0,\infty)$ of length $l$ intersects $\vartheta(f,\varepsilon)$. The collection of all almost periodic functions will be denoted by $AP(I : X).$ 

The notion of
an almost periodic function $ f: G \rightarrow E, $ where $G$ is a topological group and $E$ is a complete locally convex space, was introduced in the landmark paper \cite{boh-noe} by S. Bochner and
J. von Neumann (1935); see also the paper \cite{neumann} by J. von Neumann  (1934) for the scalar-valued case $E={\mathbb C}$.
Almost periodic functions on (semi-)topological (semi-)groups have been also analyzed in the research monographs \cite{berglund1} by J. F. Berglund, K. H. Hofmann, \cite{berglund12} by J. F. Berglund, H. D. Junghenn, P. Milnes, \cite{burckel} by R. B. Burckel, \cite{188} by B. M. Levitan and \cite{pankov} by A. A. Pankov, the doctoral dissertation of X. Zhu \cite{windzor}, the survey article \cite{shtern} by A. I. Shtern as well as the articles \cite{alfsen}, \cite{deLeeuw0}-\cite{deLeeuw}, \cite{goldberg}, \cite{kayano} and \cite{Milnes1}-\cite{mijake1};
for more details about almost automorphic functions on (semi-)topological groups, the reader may consult \cite{Milnes},  \cite{Reich} and the list of references given in our recent research paper \cite{genralized-multiaa}, where we analyze
$({\mathrm R}_{X},{\mathcal B})$-multi-almost automorphic type functions and related applications.
Working with general almost periodic functions on topological groups is rather non-trivial and, clearly, it is very difficult to provide certain applications to the abstract PDEs following this general approach. Because of that, we have decided to concretize many things here by considering various notions of almost periodicity for the vector-valued functions defined on the domain of form $I\times X,$ where $\emptyset \neq I \subseteq {\mathbb R}^{n}$ generally does not satisfy the semigroup property $I+I\subseteq I$ or contain the zero vector.
Actually,
the main aim of this paper is to introduce and systematically analyze various classes of (asymptotically) $({\mathrm R}_{X},{\mathcal B})$-multi-almost periodic type functions and (asymptotically) Bohr ${\mathcal B}$-almost periodic type functions as well as to provide several important applications to the abstract PDEs in Banach spaces. With the exception of the usual notion of Bohr almost periodicity, the introduced notion  is new even in the one-dimensional setting.

We feel it is our duty to say that the theory of almost periodic functions of several real variables has not attracted so much attention of the authors compared with the theory of almost periodic functions of one real variable by now. In support of our investigation of the multi-dimensional almost periodicity, we would like to present the following illustrative example:

\begin{example}\label{nijedanije}
Suppose that a closed linear operator $A$ generates a strongly continuous semigroup $(T(t))_{t\geq 0}$ on a Banach space $X$ whose elements are certain complex-valued functions defined on ${\mathbb R}^{n}.$ Then it is well known that, under some assumptions, 
the function
\begin{align}\label{fujmije}
u(t,x)=\bigl(T(t)u_{0}\bigr)(x)+\int^{t}_{0}[T(t-s)f(s)](x)\, ds,\quad t\geq 0,\ x\in {\mathbb R}^{n}
\end{align}
is a unique classical solution of the abstract Cauchy problem
\begin{align*}
u_{t}(t,x)=Au(t,x)+F(t,x),\ t\geq 0,\ x\in {\mathbb R}^{n}; \ u(0,x)=u_{0}(x),
\end{align*} 
where $F(t,x):=[f(t)](x),$ $t\geq 0,$ $x\in {\mathbb R}^{n}.$ In many concrete situations (for example, this holds for the Gaussian semigroup on ${\mathbb R}^{n};$ see \cite[Example 3.7.6]{a43} and the point [1.] of applications given in Section \ref{some12345}), there exists a kernel $(t,y)\mapsto E(t,y),$ $t> 0,$ $y\in {\mathbb R}^{n}$ which is integrable on any set $[0,T]\times {\mathbb R}^{n}$ ($T>0$) and satisfies that
$$
[T(t)f(s)](x)=\int_{{\mathbb R}^{n}}F(s,x-y)E(t,y)\, dy,\quad t>0,\ s\geq 0,\ x\in {\mathbb R}^{n}.
$$
If this is the case, let us fix a positive real number
$t_{0}>0.$ Concerning the inhomogeneous part in the equation \eqref{fujmije}, we would like to note that
the almost periodic behaviour of function $x\mapsto u_{t_{0}}(x)\equiv \int^{t_{0}}_{0}[T(t_{0}-s)f(s)](x)\, ds,$ $x\in {\mathbb R}^{n}$ depends on the almost periodic behaviour of function $F(t,x)$ in the space variable $x.$ The most intriguing case is that in which the function $F(t,x)$
is Bohr almost periodic with respect to the variable $x\in {\mathbb R}^{n},$ uniformly in the variable $t$ on compact subsets of $[0,\infty);$ see Subsection \ref{stavisub} below for the notion. Then the function 
$u_{t_{0}}(\cdot)$ is likewise Bohr almost periodic, which follows from the estimate
\begin{align*}
\Bigl| & u_{t_{0}}(x+\tau)- u_{t_{0}}(x) \Bigr|\leq  \int^{t_{0}}_{0} \int_{{\mathbb R}^{n}}
 | F(s,x+\tau-y)-F(s,x-y)| \cdot \bigl|E\bigl(t_{0},y\bigr)\bigr|\, dy\, ds
\\& \leq \epsilon \int^{t_{0}}_{0} \int_{{\mathbb R}^{n}}\bigl|E\bigl(t_{0},y\bigr)\bigr|\, dy\, ds
\end{align*}
and corresponding definitions.
\end{example}

\noindent {\bf Notation and terminology.}
We assume henceforth that $(X,\| \cdot \|)$, $(Y, \|\cdot\|_Y)$ and $(Z, \|\cdot\|_Z)$ are complex Banach spaces, $n\in {\mathbb N},$ $\emptyset  \neq I \subseteq {\mathbb R}^{n},$
${\mathcal B}$ is a non-empty collection of non-empty subsets of $X,$ ${\mathrm R}$ is a non-empty collection of sequences in ${\mathbb R}^{n}$
and ${\mathrm R}_{\mathrm X}$ is a non-empty collection of sequences in ${\mathbb R}^{n} \times X$;  usually, ${\mathcal B}$ denotes the collection of all bounded subsets of $X$ or all compact subsets
of $X.$ Set ${\mathcal B}_{X}:=\{y\in X : (\exists B\in {\mathcal B})\, y\in B\}.$ Although it may seem a slightly redundant (see e.g. Theorem \ref{Bochner123456}, which holds without imposing this condition), 
we will always assume henceforth that ${\mathcal B}_{X}
=X,$ i.e., 
that
for each $x\in X$ there exists $B\in {\mathcal B}$ such that $x\in B.$ By
$L(X,Y)$ we denote the Banach algebra of all bounded linear operators from $X$ into
$Y;$ $L(X,X)\equiv L(X)$. If $A: D(A) \subseteq X \mapsto X$ is a closed linear operator,
then its range and spectrum will be denoted respectively by
$R(A)$ and $\sigma(A).$ By $B^{\circ}$ and $\partial B$ we denote the interior and the boundary of a subset $B$ of a topological space $X$, respectively.

The symbol $C(I: X)$ stands for the space of all $X$-valued
continuous functions defined  on the domain $I$. By $C_{b}(I: X)$ (respectively, $BUC(I: X)$) we denote the subspace of $C(I: X)$ consisting of all bounded (respectively, all bounded uniformly continuous functions). Both $C_{b}(I: X)$ and $BUC(I: X)$ are Banach spaces with the sup-norm. This also holds for the space $C_{0}(I : X)$ consisting of all continuous functions $f : I \rightarrow X$ such that $\lim_{|t|\rightarrow +\infty}f(t)=0.$ Let su recall that a continuous function $f : I \rightarrow X,$ where $I$ is either ${\mathbb R}$ or $[0,\infty),$ is said to be asymptotically almost periodic if and only if there exist functions $g \in AP(I :X)$ and $\phi \in  C_{0}(I: X)$
such that $f = g+\phi$ (the notion of asymptotical almost periodicity for a continuous function $f : {\mathbb R} \rightarrow X$ can be introduced in some other, not equivalent, ways; see \cite{nova-mono}-\cite{nova-man} for more details). 

The Euler Gamma function is denoted by
$\Gamma(\cdot)$. 
 If ${\bf t_{0}}\in {\mathbb R}^{n}$ and $\epsilon>0$, then we set $B({\bf t}_{0},\epsilon):=\{{\bf t } \in {\mathbb R}^{n} : |{\bf t}-{\bf t_{0}}| \leq \epsilon\},$
where $|\cdot|$ denotes the Euclidean norm in ${\mathbb R}^{n}.$ Set 
${\mathbb N}_{n}:=\{1,\cdot \cdot \cdot, n\}$ and $\Delta_{n}:=\{ (t,t,\cdot \cdot \cdot,t)  \in {\mathbb R}^{n} : t\in {\mathbb R} \}$ ($n\in {\mathbb N}$).\vspace{0.2cm}

Now we are ready to briefly explain the organization and main ideas of this paper.
In Subsection \ref{stavisub}, we recall the basic facts and definitions about vector-valued almost periodic functions of several real variables; in Subsection \ref{stavisub1}, we recall some applications of vector-valued almost periodic functions of several real variables made so far.
Definition \ref{eovakoap} and Definition \ref{eovakoap1} introduce the notion of $({\mathrm R},{\mathcal B})$-multi-almost periodicity and the notion of $({\mathrm R}_{\mathrm X}, {\mathcal B})$-multi-almost periodicity for a continuous function $F : I \times X \rightarrow Y,$ respectively. The convolution invariance of space consisting of all $({\mathrm R}_{\mathrm X},{\mathcal B})$-multi-almost periodic functions is stated in Proposition \ref{convdiaggas}, while the supremum formula for the class of $({\mathrm R},{\mathcal B})$-multi-almost periodic functions is stated in Proposition \ref{netokaureap}. 

The notion of Bohr ${\mathcal B}$-almost periodicity and the notion of ${\mathcal B}$-uniform recurrence for a continuous function $F : I \times X \rightarrow Y$ are introduced in Definition \ref{nafaks1234567890}, provided that the region $I$ satisfies the semigroup property $I +I \subseteq I$. 
Numerous illustrative examples of Bohr ${\mathcal B}$-almost periodic functions and ${\mathcal B}$-uniformly recurrent functions are presented in Example \ref{pwerqwerzeka}
and Example \ref{pwerqwer}. In Definition
\ref{nafaks123456789012345}, we introduce the notion of  Bohr $({\mathcal B},I')$-almost periodicity and $({\mathcal B},I')$-uniform recurrence, provided that 
$\emptyset  \neq I'\subseteq I \subseteq {\mathbb R}^{n},$ $F : I \times X \rightarrow Y$ is a continuous function and $I +I' \subseteq I.$ 
After that, we provide several examples of  Bohr $({\mathcal B},I')$-almost periodic functions and $({\mathcal B},I')$-uniformly recurrent functions in Example \ref{rajkomilice}. 
The relative compactness of range $F(I\times B)$ for a Bohr ${\mathcal B}$-almost periodic function $F : I \times X \rightarrow Y,$ where $B\in {\mathcal B},$ is analyzed in Proposition \ref{bounded-pazi}. The Bochner criterion for Bohr ${\mathcal B}$-almost periodic functions is stated in Theorem \ref{Bochner123456}. Proposition \ref{dekartovproizvod} is crucial for clarifying the composition principles of Bohr ${\mathcal B}$-almost periodic functions; in this proposition, we investigate the common $\epsilon$-periods (see Definition \ref{nafaks1234567890}(i)) for the finite families of Bohr ${\mathcal B}$-almost periodic functions defined on ${\mathbb R}^{n} \times X$ (see also Proposition \ref{pointwise-prod-rn} and Proposition \ref{pointwise-prod-rn12345}, where we analyze the pointwise products of Bohr ${\mathcal B}$-almost periodic functions and $({\mathrm R},{\mathcal B})$-multi-almost periodic functions with scalar-valued functions). 
The uniform continuity of Bohr ${\mathcal B}$-almost periodic functions is investigated in Proposition \ref{nijenaivno}, while the class of
strongly ${\mathcal B}$-almost periodic functions is investigated in Subsection \ref{lakolakop}.
The analysis carried out in Proposition \ref{nijenaivno} indicates that there exist many concrete situations in which it is very unpleasant to work with a general region $I\neq {\mathbb R}^{n}.$

Definition \ref{kompleks12345}
introduces the notion of space $C_{0,{\mathbb D}}(I \times X :Y)$, which is crucial for introducing the notions of various types of ${\mathbb D}$-asymptotically $({\mathrm R},{\mathcal B})$-multi-almost periodicity and ${\mathbb D}$-asymptotically Bohr ${\mathcal B}$-almost periodicity; see Definition \ref{braindamage12345} and Proposition \ref{mismodranini}-Proposition \ref{2.1.10obazacap}. 
Definition \ref{nafaks123456789012345123}
introduces the notion of
${\mathbb D}$-asymptotical Bohr $({\mathcal B},I')$-almost periodicity of type $1$ 
and ${\mathbb D}$-asymptotical $({\mathcal B},I')$-uniform recurrence of type $1,$ which are further analyzed in several results preceding Subsection \ref{london-calling}, where we investigate the differentiation and integration of $({\mathrm R}_{X},{\mathcal B})$-multi-almost periodic type functions.
Our main results in this part are Theorem \ref{bounded-paziem} and Theorem \ref{izgubljeni-em}.

Concerning the proof of Theorem \ref{bounded-paziem}, we would like to emphasize that  H. Bart and S. Goldberg have proved in \cite{bart} that,
for every function $f\in AP([0,\infty) : X),$ there exists a unique almost periodic function $Ef : {\mathbb R} \rightarrow X$
such that $Ef(t)=f(t)$ for all $t\geq 0$ (see also the paper \cite{favarov-tube} by S. Favarov and O. Udodova, where the authors have investigated the extensions of almost periodic functions defined on ${\mathbb R}^{n}$ to the tube domains in ${\mathbb C}^{n}$, and the paper \cite{berg-lund} by J. F. Berglund, where the author has investigated the extensions of almost periodic functions in topological groups and semigroups). We will investigate 
the extensions of multi-dimensional almost periodic functions and multi-dimensional uniformly recurrent functions
in Remark \ref{remarka}, Theorem \ref{lenny-jasson} and Corollary \ref{lenny-jasson1} following the method proposed in the proof of Theorem \ref{bounded-paziem}, which is essentially based on the use of argumentation contained in the proof of the important theoretical result \cite[Theorem 3.4]{RUESS-1}, established by W. M. Ruess and W.  H. Summers.
Subsection \ref{marekzero} investigates the composition theorems for multi-almost periodic type functions. 

The third section of paper is reserved for the applications of our abstract theoretical results to
the various classes of abstract Volterra integro-differential equations in Banach spaces (Subsection \ref{kamlebravo} investigates some applications to the nonautonomous retarded functional evolution equations), while the final section of paper is reserved for some appendicies and notes about multiparameter strongly continuous semigroups, multivariate trigonometric polynomials and their applications (the appendix section as well as Subsection \ref{stavisub} and Subsection \ref{stavisub1} do not contain any original result of ours; their aim is to motivate the reader for further investigations of multi-dimensional almost periodicity and related applications).

\subsection{Almost periodic type functions on ${\mathbb R}^{n}$}\label{stavisub}

The notion of almost periodicity can be simply generalized to the case in which $I={\mathbb R}^{n}.$ Suppose that $F : {\mathbb R}^{n} \rightarrow X$ is a continuous function. Then we say that $F(\cdot)$ is almost periodic if and only if for each $\epsilon>0$
there exists $l>0$ such that for each ${\bf t}_{0} \in {\mathbb R}^{n}$ there exists ${\bf \tau} \in B({\bf t}_{0},l)$ such that
\begin{align}\label{emojmarko0}
\bigl\|F({\bf t}+{\bf \tau})-F({\bf t})\bigr\| \leq \epsilon,\quad {\bf t}\in {\mathbb R}^{n}.
\end{align}
This is equivalent to saying that for any sequence $({\bf b}_n)$ in ${\mathbb R}^{n}$ there exists a subsequence $({\bf a}_{n})$ of $({\bf b}_n)$
such that $(F(\cdot+{\bf a}_{n}))$ converges in $C_{b}({\mathbb R}^{n}: X).$ Any trigonometric polynomial in ${\mathbb R}^{n}$ is almost periodic and it is also well known that $F(\cdot)$ is almost periodic if and only if there exists a sequence of trigonometric polynomials in ${\mathbb R}^{n}$ which converges uniformly to $F(\cdot);$
let us recall that a trigonometric polynomial in ${\mathbb R}^{n}$ is any linear combination of functions like ${\bf t}\mapsto e^{i\langle
 {\bf \lambda}, {\bf t} \rangle},$ ${\bf t}\in{\mathbb R}^{n},$ where ${\bf \lambda}\in {\mathbb R}^{n}$ and $\langle \cdot, \cdot \rangle$ denotes the inner product in ${\mathbb R}^{n}.$ 
Using the above clarifications, we can simply prove that a continuous function $F: {\mathbb R}^{n} \rightarrow X$ is almost periodic if and only if any of the following equivalent conditions hold:
\begin{itemize}
\item[(i)] For every $j\in {\mathbb N}_{n}$ and $\epsilon>0,$ there exists a finite real number $l>0$ such that every interval $I\subseteq {\mathbb R}$ of length 
$l$ contains a point $\tau_{j} \in I$ such that 
\begin{align}\label{poziv}
\Bigl\| F\bigl( t_{1},t_{2},\cdot \cdot \cdot, t_{j}+\tau_{j},\cdot \cdot \cdot,t_{n} \bigr) -F\bigl( t_{1},t_{2},\cdot \cdot \cdot, t_{j},\cdot \cdot \cdot, t_{n} \bigr)\Bigr\|\leq \epsilon,\ {\bf t}=\bigl( t_{1},\cdot \cdot \cdot, t_{n} \bigr)\in {\mathbb R}^{n};
\end{align} 
\item[(ii)] For every $\epsilon>0,$ there exists a finite real number $l>0$ such that, for every $j\in {\mathbb N}_{n}$ and every interval $I\subseteq {\mathbb R}$ of length 
$l,$ there exists a point $\tau_{j} \in I$ such that \eqref{poziv} holds;
\item[(iii)] For every $\epsilon>0,$ there exists a finite real number $l>0$ such that every interval $I\subseteq {\mathbb R}$ of length 
$l$ contains a point $\tau \in I$ such that, for every $j\in {\mathbb N}_{n},$ \eqref{poziv} holds with the number $\tau_{j}$ replaced by the number $\tau$ therein.
\end{itemize}
Furthermore, for any almost periodic function $F(\cdot),$ we have that for each $\epsilon>0$
there exists $l>0$ such that for each ${\bf t}_{0} \in \{ (t,t,\cdot \cdot \cdot,t) \in {\mathbb R}^{n} : t\in {\mathbb R} \}$ there exists ${\bf \tau} \in B({\bf t}_{0},l) \cap \Delta_{n}$ such that
\eqref{emojmarko0} holds (see Subsection \ref{lakolakop} for more details).
Any almost periodic function $F(\cdot)$ is bounded, the mean value
$$
M(F):=\lim_{T\rightarrow +\infty}\frac{1}{(2T)^{n}}\int_{{\bf s}+K_{T}}F({\bf t})\, d{\bf t}
$$
exists and it does not depend on $s\in {\mathbb R}^{n};$ here,  
$K_{T}:=\{ {\bf t}=(t_{1},t_{2},\cdot \cdot \cdot,t_{n}) \in {\mathbb R}^{n} :  |t_{i}|\leq T\mbox{ for }1\leq i\leq n\}.$ The Bohr-Fourier coefficient $F_{\lambda}\in X$ is defined by \index{Bohr-Fourier coefficient}
$$
F_{\lambda}:=M\Bigl(e^{-i\langle \lambda, {\bf \cdot}\rangle }F(\cdot)\Bigr),\quad \lambda \in {\mathbb R}^{n}.
$$
The Bohr spectrum of $F(\cdot),$ defined by\index{Bohr spectrum}
$$
\sigma(F):=\bigl\{ \lambda \in {\mathbb R}^{n} : F_{\lambda}\neq 0 \bigr\},
$$
is at most a countable set. By $AP({\mathbb R}^{n} : X)$ and $AP_{\Lambda}({\mathbb R}^{n} : X)$ \index{space!$AP({\mathbb R}^{n} : X)$} \index{space!$AP_{\Lambda}({\mathbb R}^{n} : X)$}
we denote respectively the Banach space consisting of all almost periodic functions $F : {\mathbb R}^{n} \rightarrow X,$ equipped with the sup-norm, and its subspace consisting of all almost periodic functions $F : {\mathbb R}^{n} \rightarrow X$
such that $\sigma(F) \subseteq \Lambda.$ The Wiener algebra $APW({\mathbb R}^{n} : X)$ is defined as the set of all functions $F : {\mathbb R}^{n} \rightarrow X$ such that its Fourier series converges absolutely; $APW_{\Lambda}({\mathbb R}^{n} : X)\equiv APW({\mathbb R}^{n} : X) \cap AP_{\Lambda}({\mathbb R}^{n} : X)$. It is well known that the Wiener algebra is a Banach algebra with respect to the Wiener norm $\| F\|:=\sum_{\lambda \in {\mathbb R}^{n}}|F_{\lambda}|,$ $F\in APW({\mathbb R}^{n} : X)$ as well as that
$APW({\mathbb R}^{n} : X)$ is dense in $AP({\mathbb R}^{n} : X).$ 
\index{Wiener algebra}\index{Wiener norm}

Now we will remind the readers of several important investigations of multi-dimensional almost periodic functions carried out so far: 

1. Problems of Nehari type and contractive extension problems for matrix-valued (Wiener) almost periodic functions of
several real variables have been considered 
by L. Rodman, I. M. Spitkovsky and  H. J. Woerdeman in
\cite{rodman1}, where the authors proved a generalization of the famous Sarason's theorem \cite{sarason}. In their analysis, the notion of a halfspace in ${\mathbb R}^{n}$ 
plays an important role: a non-empty subset $S\subseteq {\mathbb R}^{n}$ is said to be a halfspace 
if and only if the following four conditions hold:
\begin{itemize}
\item[(i)] ${\mathbb R}^{n}=S \cup (-S);$
\item[(ii)] $\{0\}=S \cap (-S);$
\item[(iii)] $S+S \subseteq S;$
\item[(iv)] $\alpha \cdot S \subseteq S$ for $\alpha \geq 0.$
\end{itemize}
For any halfspace $S$ we can always find a linear bijective mapping $D : {\mathbb R}^{n} \rightarrow {\mathbb R}^{n}$ such that $S=DE_{n},$ where $E_{n}$ is a very special halfspace defined on \cite[p. 3190]{rodman}. In \cite[Theorem 1.3]{rodman}, L. Rodman and I. M. Spitkovsky have proved that, if $S$ is a halfspace and $\Lambda \subseteq S,$ $0\in \Lambda$ and $\Lambda +\Lambda \subseteq \Lambda,$ then  $AP_{\Lambda}({\mathbb R}^{n} : {\mathbb C})$ and $APW_{\Lambda}({\mathbb R}^{n} : {\mathbb C})$ are Hermitian rings. See also the article \cite{rodman2}.
\index{theorem!Sarason}\index{matrix-valued almost periodic functions}\index{Hermitian ring}

2. Let us recall that a subset $\Lambda$ of ${\mathbb R}^{n}$ is called discrete if and only if any point $\lambda \in \Lambda$ is isolated in $\Lambda.$ By ${\mathcal V}_{\Lambda}$ we denote the vector space of all finite complex-valued trigonometric polynomials $\sum_{\lambda \in \Lambda}c(\lambda)e^{-\pi i \lambda \cdot}$ whose frequencies $\lambda$ belong to $\Lambda.$ The space of mean-periodic functions with the spectrum $\Lambda,$ denoted by ${\mathcal C}_{\Lambda},$ is defined as the closure of the space 
${\mathcal V}_{\Lambda}$ in the Fr\'echet space $C({\mathbb R}^{n}).$ Clearly,  $AP_{\Lambda}({\mathbb R}^{n} : {\mathbb C})$ is contained in ${\mathcal C}_{\Lambda}$ but the converse statement is not true, in general. The problem of describing structure of closed discrete sets $\Lambda$ for which the equality $AP_{\Lambda}({\mathbb R}^{n} : {\mathbb C})={\mathcal C}_{\Lambda}$ holds was proposed by J.-P. Kahane in 1957 (\cite{kahane}). For more details about this interesting problem, we refer the reader to the survey article \cite{meyer} by Y. Meyer. 

3. In 1971, B. Basit \cite{basit1971} proved that there exists a complex-valued almost periodic
function $f : {\mathbb R}^{2} \rightarrow {\mathbb C}$ such that the function $F : {\mathbb R}^{2} \rightarrow {\mathbb C},$ defined by
$F(x,y):=\int^{x}_{0}f(t,y)\, dt,$ $(x, y) \in {\mathbb R}^{2},$
is bounded but not almost periodic. 
This result was recently reconsidered by S. M. A. Alsulami in \cite[Theorem 2.2]{integral-als}, who proved that for a complex-valued almost periodic
function $f : {\mathbb R}^{2} \rightarrow {\mathbb C},$ the boundedness of the function $F(\cdot)$ in the whole plane implies its almost periodicity, provided that there exists a complex-valued almost periodic function $g : {\mathbb R}^{2} \rightarrow {\mathbb C}$ such that  
$f_{x}(x,y)=g_{y}(x,y)$ is a continuous function in the whole plane. This result was proved with the help of 
an old result of L. H. Loomis \cite{loomaja} which states that, for a bounded complex-valued 
function $f : {\mathbb R}^{n} \rightarrow {\mathbb C},$ the almost periodicity
of all its partial derivatives of the first order
implies the almost periodicity of $f(\cdot)$ itself. Our first original observation is that the above-mentioned result of S. M. A. Alsulami can be straightforwardly extended, with the same proof, to the almost periodic functions $f : {\mathbb R}^{n} \rightarrow {\mathbb C};$ in actual fact, if the function $f : {\mathbb R}^{n} \rightarrow {\mathbb C}$ is almost periodic, the function $F(x_{1},x_{2},\cdot \cdot \cdot, x_{n}):=\int^{x_{1}}_{0}f(t,x_{2},\cdot \cdot \cdot, x_{n})\, dt,$ $(x_{1},x_{2},\cdot \cdot \cdot, x_{n}) \in {\mathbb R}^{n}$
is bounded and there exist almost periodic functions $G_{i} : {\mathbb R}^{n} \rightarrow {\mathbb C}$ such that $F_{x_{i}}(x_{1},x_{2},\cdot \cdot \cdot, x_{n}) =(G_{i})_{x_{1}}(x_{1},x_{2},\cdot \cdot \cdot, x_{n}) $ is a continuous function on ${\mathbb R}^{n},$ for $2\leq i\leq n,$ then the function $F : {\mathbb R}^{n} \rightarrow {\mathbb C}$
is almost periodic.

4. In \cite{khasanov1}-\cite{khasanov3}, Yu. Kh. Khasanov has investigated the approximations of  uniformly almost periodic functions of two variables by partial sums of Fourier sums and Marcinkiewicz sums in the uniform metric, provided certain conditions. For previous works about summability of double Fourier series, we also refer the reader to the papers \cite{marc12345} by I. Marcinkewisz and  \cite{zizias} by L. V. Zhizhiashvily. 

5. In \cite{latiff1}-\cite{latiff2}, M. A. Latif and M. I. Bhatti have investigated several imprortant questions concerning almost periodic functions defined on ${\mathbb R}^{n}$ with values in locally convex spaces and fuzzy number type spaces, while almost periodic functions defined on ${\mathbb R}^{n}$ with values in $p$-Fr\'echet spaces, where $0<p<1,$ have been investigated by G. M. N'Gu\' er\' ekata, M. A. Latif and M. I. Bhatti in \cite{gastonlatif}.

\subsection{Applications}\label{stavisub1}

Concerning applications of multi-dimensional almost periodic type functions made so far, we recall the following:

1. The problem of
existence of almost periodic solutions for the system of linear partial
differential equations 
$$
\sum_{j=1}^{n}L_{ij}u_{j}=f_{i},\quad 1\leq i\leq n,
$$
on ${\mathbb R}^{m},$ where $L_{ij}$ is an arbitrary linear partial differential operator on ${\mathbb R}^{m},$ has been analyzed by G. R. Sell \cite{sell}-\cite{sell1} by using the results from the theory of almost periodic functions of several real variables. He
has extended the results obtained in the article \cite{sibiya} by Y. Sibuya, where the author has analyzed the almost periodic solutions of Poisson's equation. Sibuya's results have been also improved, in another direction, in the recent article \cite{nazarov} by \`E. Muhamadiev and M. Nazarov, where the authors have relaxed the assumption of the usual boundedness into boundedness in the sense of distributions. 

2. The almost periodic solutions of the (semilinear) systems of ordinary differential equations have been analyzed by A. M. Fink in \cite[Chapter 8]{fink} with the help of fixed point theorems.

3. In his doctoral dissertation \cite{integral-dis}, S. M. A. Alsulami has considered the question whether the boundedness of solutions of the following system of partial first-order differential equations 
\begin{align}\label{bounded-almost}
u_{s}(s,t)=Au(s,t)+f_{1}(s,t),\quad u_{t}(s,t)=Bu(s,t)+f_{2}(s,t); \ \ \quad (s,t)\in {\mathbb R}^{2}
\end{align}
implies the almost periodicity of solutions to \eqref{bounded-almost}. He has analyzed this question in the finite-dimensional setting and the infinite-dimensional setting, by using two different techniques; in both cases, $A$ and $B$ are bounded linear operators acting on the pivot space $X.$

4. In \cite{greg12345}-\cite{greg12345678}, G. Spradlin has provided several interesting results and applications regarding almost periodic functions of several real variables. For any positive integer $n\geq 2,$ he has constructed an almost periodic infinitely differentiable almost periodic function $F : {\mathbb R}^{n} \rightarrow {\mathbb R}$ with no local mimimum (it can be simply shown that this situation cannot occur in the one-dimensional case because any almost periodic function $f : {\mathbb R} \rightarrow {\mathbb R}$ must have infinitely many local minima; see also the list of open problems proposed on p. 381 of \cite{greg12345}). The existence of 
positive homoclinic-type solutions of the equation
$$
-\Delta u +u=H({\bf t})f(u), 
$$
where $H(\cdot)$ is almost periodic and the first antiderivative of $f(\cdot)$ satisfies certain superquadraticity and critical growth conditions, has been analyzed in \cite[Theorem 1.2]{greg12345678}.
The equations of type
\begin{align}\label{nijelose}
-\epsilon^{2}\Delta u +H({\bf t}) u=f(u), 
\end{align}
arise in the study of the nonlinear Schr\"odinger equations ($\epsilon>0$). A qualitative analysis of solutions of \eqref{nijelose} has been carried out in \cite{greg123456}, provided the almost periodicity of function
$H(\cdot)$ and several other  
assumptions.

5. In the homogenization theory, numerous research articles investigate the asymptotic behaviour of the solutions of the problem
\begin{align}\label{arhangel-sk}
\inf\Biggl\{ \int_{\Omega}f(hx,Du)+\int_{\Omega}\psi x : u(\cdot)\mbox{ Lipschitz continuous and }u=0\mbox{ on }\partial \Omega \Biggr\},
\end{align}
where $\emptyset \neq \Omega \subseteq {\mathbb R}^{n}$ is an open bounded set, $\psi(\cdot)$ is essentially bounded on $\Omega,$ and $f : {\mathbb R}^{n} \times {\mathbb R}^{n} \rightarrow [0,\infty)$ satisfies the usual Carath\' eodory conditions:\index{condition!Carath\' eodory}
\begin{align}\label{ubicemotene}
f(x,z)\mbox{ is measurable in }x\mbox{ and convex in }z,
\end{align}
\begin{align}\label{na-mrtvo-more}
f(\cdot,z)\mbox{ is }[0,1]^{n}-\mbox{periodic for every }z\in  {\mathbb R}^{n},
\end{align}
and
\begin{align*}
0\leq w(x)|z|^{p} \leq f(x,z)  &  \leq W(x) \bigl( 1+ |z|^{p}\bigr) \mbox{ for a.e. }x\in  {\mathbb R}^{n}\mbox{ and every }z\in {\mathbb R}^{n}; 
\\& p>1;\ w^{(-1)/(p-1)},\ W \in L^{1}_{loc}({\mathbb R}^{n}).
\end{align*}
Under some extra assumptions, including the Lipschitz type boundary of $\Omega,$ G. de Giorgi has proved in \cite{de-giorgi} that the values in \eqref{arhangel-sk} converges to
\begin{align}\label{ubicemote}
\inf\Biggl\{ \int_{\Omega}f_{\infty}(Du)+\int_{\Omega}\psi x : u(\cdot)\mbox{ Lipschitz continuous and }u=0\mbox{ on }\partial \Omega \Biggr\},
\end{align}
where $f_{\infty} : {\mathbb R}^{n} \rightarrow [0,\infty)$ is a convex function defined by
\begin{align*}
f_{\infty}(x)& :=\lim_{s\rightarrow \infty}s^{-n}\inf\Biggl\{ \int_{(0,s)^{n}}f(x,z+Du) : \\& u(\cdot)\mbox{ Lipschitz continuous and }u=0\mbox{ on } \partial\Bigl((0,s)^{n}\Bigr)\Biggr\}.
\end{align*} 
Condition \eqref{na-mrtvo-more} has been replaced with certain almost periodic assumptions in many research articles. In
\cite{dearhangel0}, the author has
assumed that \eqref{ubicemotene} holds, $f(\cdot,z)\in L_{loc}^{1}({\mathbb R}^{n})$ for every $z\in {\mathbb R}^{n},$
$|z| \leq f(x,z)$ for a.e. $x\in {\mathbb R}^{n}$ and every $z\in {\mathbb R}^{n},$
and 
the following almost periodic type condition, which needs to be valid for every $z\in {\mathbb R}^{n}:$ For every $\epsilon>0,$ there exists a finite real number $L_{\epsilon}>0$ such that, for every $x_{0}\in {\mathbb R}^{n},$
there exists $\tau \in x_{0}+B(0,L_{\epsilon})$ such that
$$
|f(x+\tau,z)-f(x,z)| \leq \epsilon(1+f(x,z)),\quad \mbox{ for a.e. }x\in  {\mathbb R}^{n}\mbox{ and every }z\in {\mathbb R}^{n}.
$$
Then, for every open convex set $\Omega$ and for every essentially bounded function $\psi(\cdot)$ on $\Omega,$ the values in \eqref{arhangel-sk}
converges to the value in \eqref{ubicemote}.

In \cite{ishii}, H. Ishii has analyzed
the asymptotic behavior, as the parameter $\epsilon$ tends to $0+,$ of
the solution $u^{\epsilon}$ 
of the Hamilton-Jacobi equation\index{equation!Hamilton-Jacobi}
\begin{equation}\label{visco-selosam}
u(x)+H(x,x/\epsilon, Du(x))=0,\quad x\in {\mathbb R}^{n},
\end{equation}
where $\epsilon>0$ is a positive real number. This equation describs a\index{Riemannian metric}\index{Hamiltonian}
sort of distance functions in the space where the Riemannian metric is oscillatory (for more details about the generalized solutions of the Hamilton-Jacobi equations, we refer the reader to the monograph \cite{pllions} by P. L. Lions). 
The basic assumption made in \cite{ishii} is that the Hamiltonian $H(x,y,p)$ is almost periodic with respect to the variable $y;$ 
since \eqref{visco-selosam} does not have classical solutions, the author has considered certain types of viscosity solutions of this equation. Under certain extra assumptions, the author has introduced the notion of effective Hamiltonian $\overline{H} : {\mathbb R}^{n} \times {\mathbb R}^{n} \rightarrow {\mathbb R},$
which satisfies the following estimates\index{Hamiltonian!effective}
$$
\inf_{y\in {\mathbb R}^{n}}H(x,y,p) \leq \overline{H}(x,y) \leq \sup_{y\in {\mathbb R}^{n}}H(x,y,p),\quad x,\ p\in {\mathbb R}^{n},
$$
and proved that the approximate solutions $u^{\epsilon} \rightarrow u$ locally uniformly as $\epsilon$ tends to $0+,$ where $u(\cdot)$ is a unique bounded, uniformly continuous solution of the equation
$$
u(x)+\overline{H}(x,Du(x))=0,\quad x\in {\mathbb R}^{n}.
$$
For more details about the relationship between almost periodicity and homogenization theory, we also refer the reader to the references cited in the appendix to the third chapter of the forthcoming monograph \cite{nova-man}.

6. The existence and uniqueness of almost periodic solutions for a class of boundary value problems
for hyperbolic equations have been investigated by B. I. Ptashnic and P. I. Shtabalyuk in \cite{bi-ptas}. In the region $D_{p}=(0,T) \times {\mathbb R}^{p}$ ($T>0$, $p\in {\mathbb N}$), they have analyzed the well-posedness of following initial value problem
\begin{align}\label{giperbol}
Lu\equiv \sum_{s=0}^{n}\sum_{|\alpha|=2s}a_{\alpha}\frac{\partial^{2n}u(t,x)}{\partial t^{2n-2s}\partial x_{1}^{\alpha_{1}}\cdot \cdot \cdot \partial x_{p}^{\alpha_{p}}}=0,
\end{align}
\begin{align}\label{popizd}
\frac{\partial^{j-1}u}{\partial t^{j-1}}\Biggl |_{t=0}=\varphi_{j}(x),\quad \frac{\partial^{j-1}u}{\partial t^{j-1}}\Biggl |_{t=T}=\varphi_{j+n}(x) \ \ (1\leq j\leq n).
\end{align}\index{Petrovsky hyperbolicity}
The basic assumption employed in \cite{bi-ptas} is that the equation \eqref{giperbol} is Petrovsky hyperbolic, i.e., that for each $\mu=(\mu_{1},\mu_{2},\cdot \cdot \cdot,\mu_{p}) \in {\mathbb R}^{p}$ all $\lambda$-zeroes of the equation
$$
\sum_{s=0}^{n}\sum_{|\alpha|=2s}a_{\alpha}\lambda^{2n-2s}\mu_{1}^{\alpha_{1}}\mu_{2}^{\alpha_{2}}\cdot \cdot \cdot \mu_{p}^{\alpha_{p}}=0
$$
are real. The basic function space used is the Banach space $C_{B}^{q}(\overline{D^{p}})$ consisting of all $q$-times continuously differentiable functions $u(t,x)$ in $\overline{D^{p}}$ that are Bohr almost periodic in variables $x_{1},x_{2},\cdot \cdot \cdot, x_{p},$
uniformly in $t\in [0,T],$\index{space!$C_{B}^{q}(\overline{D^{p}})$}
equipped with the norm
$$
\| u\|_{C_{B}^{q}(\overline{D^{p}})}:=\sup_{0\leq |s|\leq q}\sup_{(t,x)\in \overline{D^{p}}}\frac{\partial^{|s|}u(t,x)}{\partial t^{s_{0}}\partial x_{1}^{s_{1}}\cdot \cdot \cdot \partial x_{p}^{s_{p}}};
$$\index{space!$C_{B}^{q}({\mathbb R}^{p})$}
by $C_{B}^{q}({\mathbb R}^{p})$ the authors have designated the subspace of $C_{B}^{q}(\overline{D^{p}})$ consisting of those functions which do not depend on the variable $t.$ The existence and uniqueness of solutions of the initial value problem \eqref{giperbol}-\eqref{popizd} have been investigated in the space $C_{B}^{2n}(\overline{D^{p}}),$ under the assumption that $\varphi_{j}(x)\in C_{B}^{r}({\mathbb R}^{p})$ and $r\in {\mathbb N}$ is sufficiently large. If $M_{p}=\{\mu_{k} : k\in {\mathbb Z}^{p}\}$ is the union of spectrum of all functions $\varphi_{1}(x),\cdot \cdot \cdot, \varphi_{2n}(x),$ the solutions $u(t,x)$ of problem \eqref{giperbol}-\eqref{popizd} have been found in the form
$$
u(t,x)=\sum_{k\in {\mathbb Z}^{p}}u_{k}(t)e^{i\langle \mu_{k},x \rangle},\quad \mu_{k}\in M_{p},
$$  
where the functions $u_{k}(t)$ satisfy certain conditions and have the form \cite[(8), p. 670]{bi-ptas}. The uniqueness of solutions of problem \eqref{giperbol}-\eqref{popizd} has been considered in \cite[Theorem 1]{bi-ptas}, while the existence of solutions of problem \eqref{giperbol}-\eqref{popizd} has been considered in \cite[Theorem 2]{bi-ptas}. See also the research articles \cite{P. I. Shtabalyuk1}-\cite{P. I. Shtabalyuk2} by P. I. Shtabalyuk.

Concerning almost periodic functions of several real variables, we also refer the reader to the research monographs \cite{corduneanu} by C. Corduneanu, \cite{fink} by A. M. Fink, \cite{pankov} by A. A. Pankov and \cite{30} by S. Zaidman. The interested reader may also consult the paper \cite{bochnerbound} by S. Bochner, which concerns the extensions of the Riesz theorem to the analytic functions of several real variables and the almost periodic functions of several real variables.

\section{$({\mathrm R}_{X},{\mathcal B})$-Multi-almost periodic type functions and Bohr ${\mathcal B}$-almost periodic type functions}\label{maremare}

The main aim of this section is to analyze $({\mathrm R}_{X},{\mathcal B})$-multi-almost periodic type functions and Bohr ${\mathcal B}$-almost periodic type functions.
Let us recall that 
${\mathcal B}$ denotes a non-empty collection of non-empty subsets of $X,$ ${\mathrm R}$ denotes a non-empty collection of sequences in ${\mathbb R}^{n}$
and ${\mathrm R}_{\mathrm X}$ denotes a non-empty collection of sequences in ${\mathbb R}^{n} \times X.$

In the following two definitions, we introduce the notion of $({\mathrm R},{\mathcal B})$-multi-almost periodicity and one of its most important generalizations, the notion of $({\mathrm R}_{\mathrm X},{\mathcal B})$-multi-almost periodicity (the both notions can be introduced 
on general semitopological groups):

\begin{defn}\label{eovakoap}\index{function!$({\mathrm R},{\mathcal B})$-multi-almost periodic}
Suppose that $\emptyset \neq I \subseteq {\mathbb R}^{n},$ $F : I \times X \rightarrow Y$ is a continuous function, and the following condition holds:
\begin{align}\label{lepolepo}
\mbox{If}\ \ {\bf t}\in I,\ {\bf b}\in {\mathrm R}\ \mbox{ and }\ l\in {\mathbb N},\ \mbox{ then we have }\ {\bf t}+{\bf b}(l)\in I.
\end{align}
Then 
we say that the function $F(\cdot;\cdot)$ is $({\mathrm R},{\mathcal B})$-multi-almost periodic if and only if for every $B\in {\mathcal B}$ and for every sequence $({\bf b}_{k}=(b_{k}^{1},b_{k}^{2},\cdot \cdot\cdot ,b_{k}^{n})) \in {\mathrm R}$ there exist a subsequence $({\bf b}_{k_{l}}=(b_{k_{l}}^{1},b_{k_{l}}^{2},\cdot \cdot\cdot , b_{k_{l}}^{n}))$ of $({\bf b}_{k})$ and a function
$F^{\ast} : I \times X \rightarrow Y$ such that
\begin{align}\label{love12345678ap}
\lim_{l\rightarrow +\infty}F\bigl({\bf t} +(b_{k_{l}}^{1},\cdot \cdot\cdot, b_{k_{l}}^{n});x\bigr)=F^{\ast}({\bf t};x) 
\end{align}
uniformly for all $x\in B$ and ${\bf t}\in I.$ By $AP_{({\mathrm R},{\mathcal B})}(I \times X : Y)$  we denote the space consisting of all 
$({\mathrm R},{\mathcal B})$-multi-almost periodic functions.\index{space!$AP_{({\mathrm R},{\mathcal B})}(I \times X : Y)$} 
\end{defn}

\begin{defn}\label{eovakoap1}\index{function!$({\mathrm R}_{\mathrm X},{\mathcal B})$-multi-almost periodic}
Suppose that $\emptyset  \neq I \subseteq {\mathbb R}^{n},$ $F : I \times X \rightarrow Y$ is a continuous function, and the following condition holds:
\begin{align}\label{lepolepo121}
\mbox{If}\ \ {\bf t}\in I,\ ({\bf b};{\bf x})\in {\mathrm R}_{\mathrm X}\ \mbox{ and }\ l\in {\mathbb N},\ \mbox{ then we have }\ {\bf t}+{\bf b}(l)\in I.
\end{align}
Then 
we say that the function $F(\cdot;\cdot)$ is $({\mathrm R}_{\mathrm X}, {\mathcal B})$-multi-almost periodic if and only if for every $B\in {\mathcal B}$ and for every sequence $(({\bf b;{\bf x}})_{k}=((b_{k}^{1},b_{k}^{2},\cdot \cdot\cdot ,b_{k}^{n});x_{k})_{k}) \in {\mathrm R}_{\mathrm X}$ there exist a subsequence $(({\bf b;{\bf x}})_{k_{l}}=((b_{k_{l}}^{1},b_{k_{l}}^{2},\cdot \cdot\cdot , b_{k_{l}}^{n});x_{k_{l}})_{k_{l}})$ of $(({\bf b};{\bf x})_{k})$ and a function
$F^{\ast} : I \times X \rightarrow Y$ such that
\begin{align}\label{love12345678ap1}
\lim_{l\rightarrow +\infty}F\bigl({\bf t} +(b_{k_{l}}^{1},\cdot \cdot\cdot, b_{k_{l}}^{n});x+x_{k_{l}}\bigr)=F^{\ast}({\bf t};x) 
\end{align}
uniformly for all $x\in B$ and ${\bf t}\in I.$ By $AP_{({\mathrm R}_{\mathrm X},{\mathcal B})}(I \times X : Y)$  we denote the space consisting of all 
$({\mathrm R}_{\mathrm X},{\mathcal B})$-multi-almost periodic functions.\index{space!$AP_{({\mathrm R}_{\mathrm X},{\mathcal B})}(I \times X : Y)$} 
\end{defn}

In our further investigations of $({\mathrm R},{\mathcal B})$-multi-almost periodicity ($({\mathrm R}_{X},{\mathcal B})$-multi-almost periodicity), we will always tactily assume that \eqref{lepolepo} (\eqref{lepolepo121})
holds for $I$ and ${\mathrm R}$ ($I$ and ${\mathrm R}_{X}$).
Before we go any further, we would like to provide several useful observations about the notion introduced above:

\begin{rem}\label{soqqaza}
\begin{itemize}
\item[(i)]
The notion 
introduced in Definition \ref{eovakoap} is a special case of the notion introduced in Definition \ref{eovakoap1}. In order to see this, suppose that the function $F : I \times X \rightarrow Y$ is continuous.
Set
$$
{\mathrm R}_{\mathrm X}:=\bigl\{ b : {\mathbb N} \rightarrow {\mathbb R}^{n} \times X \, ;\, (\exists a\in {\mathrm R}) \, b(l)=(a(l);0)\mbox{ for all }l\in{\mathbb N}\bigr\}.
$$
Then it is clear that 
\eqref{lepolepo} holds for $I$ and ${\mathrm R}$ if and only if \eqref{lepolepo121} holds for $I$ and ${\mathrm R}_{X};$
furthermore,
$F(\cdot;\cdot)$ is $({\mathrm R},{\mathcal B})$-multi-almost periodic if and only if $({\mathrm R}_{\mathrm X},{\mathcal B})$-multi-almost periodic.
It is also clear that, if the function $F(\cdot;\cdot)$ is $({\mathrm R}_{\mathrm X}, {\mathcal B})$-multi-almost periodic, then we have
$F^{\ast}({\bf t}; x)\in \overline{R(F)}$ for all $x\in X$ and ${\bf t}\in I.$ 
\item[(ii)]
The domain $I$ from the above two definitions is rather general. For example,
if $n=1,$ $I=[0,\infty),$ $X=\{0\},$ ${\mathcal B}=\{X\}$ and ${\mathrm R}$ is the collection of all sequences in $[0,\infty),$ then the notion of $({\mathrm R},{\mathcal B})$-multi-almost periodicity is equivalent with the notion of
asymptotical almost periodicity considered usually since a function $f : [0,\infty) \rightarrow Y$ is asymptotically almost periodic
if and only if the set $H(f):=\{f(\cdot +s) : s\geq 0\}$ is relatively compact in $C_{b}([0,\infty):X),$ which means that for any sequence $(b_n)$ of non-negative real numbers there exists a subsequence $(a_{n})$ of $(b_n)$
such that $(f(\cdot+a_{n}))$ converges in $C_{b}([0,\infty): X).$ 
Moreover, if $I$ is a cone in ${\mathbb R}^{n},$ $X=\{0\},$ ${\mathcal B}=\{X\}$, $Y={\mathbb C}$ and ${\mathrm R}$ is a collection of all sequences in $I,$ then a well known result of 
K. deLeeuw and I. Glicksberg \cite[Theorem 9.1]{deLeeuw} says that any $({\mathrm R},{\mathcal B})$-multi-almost periodic function $F : I \rightarrow {\mathbb C}$  can be uniformly approximated by linear combinations of semicharacters of $I,$
which will be 
exponential functions in this case.
If $X=\{0\},$ then we also say that the function $F : I \rightarrow Y$ is ${\mathrm R}$-multi-almost periodic, resp. ${\mathrm R}_{\mathrm X}$-multi-almost periodic.
\item[(iii)]
It is clear that an ${\mathrm R}$-multi-almost periodic function need not be bounded in general; for example, if ${\mathrm R}$ is the collection of all bounded sequences in ${\mathbb R}^{n},$ then an application of the Bolzano-Weierstrass theorem shows that the identical mapping from ${\mathbb R}^{n}$ into ${\mathbb R}^{n}$
is ${\mathrm R}$-multi-almost periodic. The existence of an unbounded sequence $({\bf b}_{k}) \in {\mathrm R}$ does not imply the boundedness of $F(\cdot),$ as well; for example, any unbounded uniformly recurrent function $F : {\mathbb R}^{n} \rightarrow Y$ satisfying the estimate \eqref{jadnice} below with the sequence $({\bf \tau}_{k})$ in ${\mathbb R}^{n}$ satisfying $\lim_{k\rightarrow +\infty} |{\bf \tau}_{k}|=+\infty$ is ${\mathrm R}$-multi-almost periodic with ${\mathrm R}$ being the collection consisting of the sequence 
$({\bf \tau}_{k})$ and all its subsequences.
\item[(iv)]
Suppose $0\in I,$ $I+I\subseteq I,$ ${\mathrm R}_{\mathrm X}$ denotes the collection of all sequences in $I\times X$ and ${\mathcal B}=\{X\}.$
Let us recall that two sufficient conditions for a continuous function $F : I \times X \rightarrow Y$ to be $({\mathrm R}_{\mathrm X}, {\mathcal B})$-multi-almost periodic were obtained by P. Milnes in \cite[Theorem 2]{Milnes1} and T. Kayano in \cite[Theorem 3]{kayano};
some equivalent conditions for $F(\cdot;\cdot)$ to be $({\mathrm R}_{\mathrm X}, {\mathcal B})$-multi-almost periodic can be found in  \cite[Theorem 1(i)]{Milnes1} and \cite[Theorem 4(d)]{kayano}.
\end{itemize}
\end{rem} 

Let $k\in {\mathbb N}$ and $F_{i} : I \times X \rightarrow Y_{i}$ ($1\leq i\leq k$). Then we define the function $(F_{1},\cdot \cdot \cdot, F_{k}) : I \times X \rightarrow Y_{1}\times \cdot \cdot \cdot \times Y_{k}$ by
$$
(F_{1},\cdot \cdot \cdot, F_{k})({\bf t} ;x):=(F_{1}({\bf t};x) , \cdot \cdot \cdot, F_{k}({\bf t};x) ),\quad {\bf t} \in I,\ x\in X.
$$
Using an induction argument and an elementary argumentation, we may deduce the following:

\begin{prop}\label{kursk-kursk}
\begin{itemize}
\item[(i)] Suppose that $k\in {\mathbb N}$, $\emptyset \neq I \subseteq {\mathbb R}^{n},$ \eqref{lepolepo} holds and for any sequence  which belongs to 
${\mathrm R}$ we have that any its subsequence also belongs to ${\mathrm R}.$
If the function $F_{i}(\cdot;\cdot)$ is $({\mathrm R},{\mathcal B})$-multi-almost periodic for $1\leq i\leq k$, then the function $(F_{1},\cdot \cdot \cdot, F_{k})(\cdot;\cdot)$ is also $({\mathrm R},{\mathcal B})$-multi-almost periodic.
\item[(ii)]  Suppose that $k\in {\mathbb N}$, $\emptyset \neq I \subseteq {\mathbb R}^{n},$ \eqref{lepolepo} holds and for any sequence  which belongs to 
${\mathrm R}_{\mathrm X}$ we have that any its subsequence also belongs to ${\mathrm R}_{\mathrm X}.$
If the function $F_{i}(\cdot;\cdot)$ is $({\mathrm R}_{\mathrm X},{\mathcal B})$-multi-almost periodic for $1\leq i\leq k$, then the function $(F_{1},\cdot \cdot \cdot, F_{k})(\cdot;\cdot)$ is also $({\mathrm R}_{\mathrm X},{\mathcal B})$-multi-almost periodic.
\end{itemize}
\end{prop}

Concerning the convolution invariance of space consisting of all $({\mathrm R}_{\mathrm X},{\mathcal B})$-multi-almost periodic functions, we would like to state the following result:

\begin{prop}\label{convdiaggas}
Suppose that $h\in L^{1}({\mathbb R}^{n}),$ the function $F(\cdot ; \cdot)$ is $({\mathrm R}_{\mathrm X},{\mathcal B})$-multi-almost periodic 
and
for each bounded subset $D$ of $X$ there exists a constant $c_{D}>0$ such that 
$\|F({\bf t} ; x)\|_{Y}\leq c_{D}$ for all ${\bf t}\in {\mathbb R}^{n},$ $x\in D.$ Suppose, further, that for each sequence $(({\bf b;{\bf x}})_{k}=((b_{k}^{1},b_{k}^{2},\cdot \cdot\cdot ,b_{k}^{n});x_{k})_{k}) \in {\mathrm R}_{\mathrm X}$ and for each set $B\in {\mathcal B}$ we have that $B+\{x_{k} : k\in {\mathbb N}\}$ is a bounded set in $X.$ 
Then the function 
$$
(h\ast F)({\bf t};x):=\int_{{\mathbb R}^{n}}h(\sigma)F({\bf t}-\sigma;x)\, d\sigma,\quad {\bf t}\in {\mathbb R}^{n},\ x\in X
$$
is $({\mathrm R}_{\mathrm X},{\mathcal B})$-multi-almost periodic and satisfies that for each bounded subset $D$ of $X$ there exists a constant $c_{D}'>0$ such that 
$\|(h\ast F)({\bf t} ; x)\|_{Y}\leq c_{D}'$ for all ${\bf t}\in {\mathbb R}^{n},$ $x\in D.$  
\end{prop}

\begin{proof}
Since $h\in L^{1}({\mathbb R}^{n}),$ the prescribed assumptions imply that
the function $(h\ast F)(\cdot ;\cdot)$ is well defined as well as that
for each bounded subset $D$ of $X$ there exists a constant $c_{D}''>0$ such that 
$\|(h\ast F)({\bf t} ; x)\|_{Y}\leq c_{D}''$ for all ${\bf t}\in {\mathbb R}^{n},$ $x\in D.$
The continuity of function $(h\ast F)(\cdot ;\cdot)$ follows from the dominated convergence theorem and the same assumption on the function $F(\cdot;\cdot).$ Let the set $B\in {\mathcal B}$ be fixed. 
Then for each sequence
$(({\bf b;{\bf x}})_{k}=((b_{k}^{1},b_{k}^{2},\cdot \cdot\cdot ,b_{k}^{n});x_{k})_{k})\in {\mathrm R}_{\mathrm X}$ there exist a subsequence $(({\bf b;{\bf x}})_{k_{l}}=((b_{k_{l}}^{1},b_{k_{l}}^{2},\cdot \cdot\cdot , b_{k_{l}}^{n});x_{k_{l}})_{k_{l}})$ of $(({\bf b};{\bf x})_{k})$ and a function
$F^{\ast} : {\mathbb R}^{n} \times X \rightarrow Y$ such that
\eqref{love12345678ap1}
 holds. By our assumption, $B+\{x_{k} : k\in {\mathbb N}\}$ is a bounded set in $X$ so that there exists a finite real constant $c_{B}''' >0$ such that
$\|F^{\ast}({\bf t} ; x)\|_{Y}\leq c_{B}'''$ for all ${\bf t}\in {\mathbb R}^{n},$ $x\in B.$ Keeping this in mind and our standing hypothesis $X_{\mathcal B}=X$, we get that the function 
$(h\ast F^{\ast})(\cdot;\cdot)$ is well defined.
The remainder of proof can be deduced by using the estimate
\begin{align*}
\Bigl\| & (h\ast F)({\bf t}+{\bf b}_{k_{l}};x+x_{k_{l}}) -\bigl( h\ast F^{\ast})({\bf t}; x)\Bigr\|_{Y}
\\& \leq  \int_{{\mathbb R}^{n}}|h(\sigma)| \Bigl\|F({\bf t}+{\bf b}_{k_{l}}-\sigma;x+x_{k_{l}})-
F^{\ast}({\bf t}-\sigma; x)\Bigr\|_{Y}\, d\sigma,
\end{align*}
which holds for any ${\bf t}\in {\mathbb R}^{n},$ $l\in {\mathbb N}$ and $x\in X;$ see Definition \ref{eovakoap1}.
\end{proof}

Keeping in mind the notion introduced and analyzed in \cite{genralized-multiaa}, almost directly from the above definitions we may conclude the following:

\begin{itemize}
\item[RB1.] Let $I= {\mathbb R}^{n}.$ If the function $F(\cdot;\cdot)$ is $({\mathrm R},{\mathcal B})$-multi-almost periodic, then it is compactly $({\mathrm R},{\mathcal B})$-multi-almost periodic in the sense of \cite[Definition 2.1]{genralized-multiaa}.
\item[RB2.] If $X\in {\mathcal B},$ $I= {\mathbb R}^{n}$ and ${\mathrm R}_{\mathrm X}$ is a collection of all sequences in ${\mathbb R}^{n} \times X,$ then the notion of $({\mathrm R}_{\mathrm X},{\mathcal B})$-multi-almost periodicity is equivalent with the usual notion of 
almost periodicity (see e.g., \cite[p. 255]{188}). It is also clear that, if the function $F : {\mathbb R}^{n} \times X \rightarrow Y$ is almost periodic, then it is $({\mathrm R},{\mathcal B})$-multi-almost periodic, where ${\mathrm R}$ is the collection of all sequences $b(\cdot)$ in ${\mathbb R}^{n}$ and ${\mathcal B}$ is the collection of all subsets of $X$.
\item[RB3.]  Let $I= {\mathbb R}^{n}.$ If the function $F(\cdot;\cdot)$ is $({\mathrm R}_{\mathrm X},{\mathcal B})$-multi-almost periodic, then it is compactly $({\mathrm R}_{\mathrm X},{\mathcal B})$-multi-almost periodic in the sense of \cite[Definition 2.2]{genralized-multiaa}, provided that for every $B\in {\mathcal B},$ for every $x\in B$ and for every sequence $b_{2}: {\mathbb N} \rightarrow X$ for which there exist a sequence $b : {\mathbb N} \rightarrow {\mathbb R}^{n}\times X$ and a sequence $b_{1}: {\mathbb N} \rightarrow {\mathbb R}^{n}$
such that $b(l)=(b_{1}(l);b_{2}(l)),$ $l\in {\mathbb N}$ we have $x-b_{2}(l) \in B,$ $l\in {\mathbb N}.$ 
\end{itemize}

For the notion introduced in Definition \ref{eovakoap}, the supremum formula can be proved under the following conditions; see also \cite[Proposition 2.6]{genralized-multiaa} for the corresponding statement regarding $({\mathrm R},{\mathcal B})$-multi-almost automorphy:

\begin{prop}\label{netokaureap}
Suppose that $F : I \times X \rightarrow Y$ is $({\mathrm R},{\mathcal B})$-multi-almost periodic, $a\geq 0$ and $x\in X.$
If there exists a sequence $b(\cdot)$ in ${\mathrm R}$ whose any subsequence is unbounded and for which we have ${\bf T}-{\bf b}(l)\in I$ whenever ${\bf T}\in I$ and $l\in {\mathbb N},$  
then we have
\begin{align}\label{m91ap}
\sup_{{\bf t}\in I}\bigl\|F({\bf t};x) \bigr\|_{Y}=\sup_{{\bf t}\in I,|t|\geq a}\bigl\|F({\bf t};x) \bigr\|_{Y}.
\end{align}
\end{prop}

\begin{proof}
We will include all relevant details of the proof for the sake of completeness. Let $\epsilon>0,$ $a\geq 0$ and $x\in X$ be given. Then \eqref{m91ap} will hold if we prove that 
\begin{align}\label{lozoh}
\bigl\|F({\bf t};x) \bigr\|_{Y} \leq \epsilon+\sup_{{\bf t}\in I,|t|\geq a}\bigl\|F({\bf t};x) \bigr\|_{Y}.
\end{align}
Let $B\in {\mathcal B}$ be such that $x\in B,$ and let $b(\cdot)$ be any sequence in ${\mathrm R}$ with the prescribed assumptions. Then there exists an integer $l_{0} \in {\mathbb N}$ such that
$$
\Bigl\|  F\Bigl({\bf T}-\bigl(b_{k_{l_{0}}}^{1},\cdot \cdot\cdot, b_{k_{l_{0}}}^{n}\bigr) ; x\Bigr) - F\Bigl({\bf T}-\bigl(b_{k_{l}}^{1},\cdot \cdot\cdot, b_{k_{l}}^{n}\bigr); x\Bigr)\Bigr\|_{Y}\leq \epsilon,\quad l\geq l_{0},\ {\bf T}\in I,\ x\in B.
$$
Plugging ${\bf t}={\bf T}-\bigl(b_{k_{l_{0}}}^{1},\cdot \cdot\cdot, b_{k_{l_{0}}}^{n}\bigr),$ we simply obtain \eqref{lozoh}.
\end{proof}

Now we will prove the following result:

\begin{prop}\label{2.1.10ap}
Suppose that for each integer $j\in {\mathbb N}$ the function $F_{j}(\cdot ; \cdot)$ is $({\mathrm R}_{\mathrm X},{\mathcal B})$-multi-almost periodic and,  for every sequence  which belongs to 
${\mathrm R}_{\mathrm X},$ any its subsequence also belongs to ${\mathrm R}_{\mathrm X}.$ If 
the sequence $(F_{j}(\cdot ;\cdot))$ converges uniformly to a function $F(\cdot ;\cdot)$ on $X$, then the function $F(\cdot ;\cdot)$ is $({\mathrm R}_{\mathrm X},{\mathcal B})$-multi-almost periodic.
\end{prop}

\begin{proof}
The proof is very similar to the proof of \cite[Theorem 2.1.10]{gaston} but we will provide all relevant details.
Let ${\bf t}\in I$ and $x\in X$ be given. In order to prove that the function $F(\cdot ;\cdot)$ is continuous at $({\bf t} ; x),$ observe first that our standing assumption ${\mathcal B}_{X}=X$ gives the existence of a set $B\in {\mathcal B}$ such that $x\in B.$ Since the sequence $(F_{j}(\cdot ;\cdot))$ converges uniformly to a function $F(\cdot ;\cdot)$ on $X,$ we have the existence of a positive integer $n_{0}\in {\mathbb N}$ such that $\|F_{n_{0}}({\bf t}' ;x')-F({\bf t}' ;x')\|_{Y}\leq \epsilon/3$ for all ${\bf t}'\in I$ and $x'\in X.$ After that, it suffices to observe that
\begin{align}
\notag \bigl\|F&({\bf t} ;x)-F({\bf t}' ;x')\bigr\|_{Y}\leq \bigl\|F({\bf t}' ;x')-F_{n_{0}}({\bf t}' ;x')\bigr\|_{Y}+\bigl\|F_{n_{0}}({\bf t}' ;x')-F_{n_{0}}({\bf t} ;x)\bigr\|_{Y}\
\\\label{nijenorm} &+\bigl\|F_{n_{0}}({\bf t} ;x)-F({\bf t} ;x)\bigr\|_{Y},\quad {\bf t}'\in I,\ x'\in X,
\end{align}
as well as to employ the continuity of $F_{n_{0}}(\cdot ;\cdot)$ at $({\bf t} ; x).$ Further on,
let the set $B\in {\mathcal B}$ and the sequence $(({\bf b}_{k};x_{k})=((b_{k}^{1},b_{k}^{2},\cdot \cdot\cdot ,b_{k}^{n});x_{k}))\in {\mathrm R}_{\mathrm X}$ be given. Since we have assumed that,
for every sequence  which belongs to 
${\mathrm R}_{\mathrm X},$ any its subsequence also belongs to ${\mathrm R}_{\mathrm X},$
using the diagonal procedure we get the existence of 
a subsequence $(({\bf b}_{k_{l}};x_{k_{l}})=((b_{k_{l}}^{1},b_{k_{l}}^{2},\cdot \cdot\cdot , b_{k_{l}}^{n});x_{k_{l}}))$ of $(({\bf b}_{k};x_{k}))$ such that for each integer $j\in {\mathbb N}$ there exists a function
$F^{\ast}_{j} : I \times X \rightarrow Y$ such that
\begin{align}\label{metalac12345}
\lim_{l\rightarrow +\infty}\Bigl\| F_{j}\bigl({\bf t} +(b_{k_{l}}^{1},\cdot \cdot\cdot, b_{k_{l}}^{n});x+x_{k_{l}}\bigr)-F^{\ast}_{j}({\bf t};x) \Bigr\|_{Y}=0,
\end{align}
uniformly for $x\in B$ and ${\bf t}\in I.$  
Fix now a positive real number $\epsilon>0.$
Since 
\begin{align*}
\Bigl\| F_{i}^{\ast}({\bf t} ; x) &-F_{j}^{\ast}({\bf t} ; x)\Bigr\|_{Y}\leq \Bigl\| F_{i}^{\ast}({\bf t} ; x)-F_{i}({\bf t}+(b_{k_{l}}^{1},\cdot \cdot\cdot, b_{k_{l}}^{n}) ; x+x_{k_{l}})\Bigr\|_{Y}
\\&+\Bigl\| F_{i}({\bf t} +(b_{k_{l}}^{1},\cdot \cdot\cdot, b_{k_{l}}^{n}); x+x_{k_{l}})-F_{j}({\bf t} +(b_{k_{l}}^{1},\cdot \cdot\cdot, b_{k_{l}}^{n}); x+x_{k_{l}})\Bigr\|_{Y}
\\&+\Bigl\| F_{j}({\bf t} +(b_{k_{l}}^{1},\cdot \cdot\cdot, b_{k_{l}}^{n}); x+x_{k_{l}})-F_{j}^{\ast}({\bf t} ; x)\Bigr\|_{Y},
\end{align*}
and \eqref{metalac12345} holds, we can find a number $l_{0}\in {\mathbb N}$ such that for all integers $l\geq l_{0}$ we have:
\begin{align*}
\Bigl\| & F_{i}^{\ast}({\bf t} ; x)-F_{i}({\bf t}+(b_{k_{l}}^{1},\cdot \cdot\cdot, b_{k_{l}}^{n}) ; x+x_{k_{l}})\Bigr\|_{Y}
\\&+\Bigl\| F_{j}({\bf t} +(b_{k_{l}}^{1},\cdot \cdot\cdot, b_{k_{l}}^{n}); x+x_{k_{l}})-F_{j}^{\ast}({\bf t} ; x)\Bigr\|_{Y}<2\epsilon/3,
\end{align*}
uniformly for $x\in B$ and ${\bf t}\in I.$
Since the sequence $(F_{j}(\cdot ;\cdot))$ converges uniformly to a function $F(\cdot ;\cdot)$ on $X,$ there exists $N(\epsilon)\in {\mathbb N}$ such that
for all integers $i,\ j\in {\mathbb N}$ with $\min(i,j)\geq N(\epsilon)$ we have
\begin{align}\label{nutela}
\Bigl\| F_{i}({\bf t} +(b_{k_{l}}^{1},\cdot \cdot\cdot, b_{k_{l}}^{n}); x+x_{k_{l}})-F_{j}({\bf t} +(b_{k_{l}}^{1},\cdot \cdot\cdot, b_{k_{l}}^{n}); x+x_{k_{l}})\Bigr\|_{Y}<\epsilon/3,
\end{align}
uniformly for $x\in B$ and ${\bf t}\in I.$  
This implies that $(F_{j}^{\ast}({\bf t} ;x))$ is a Cauchy sequence in $Y$ and therefore convergent to an element  $F^{\ast}({\bf t} ;x),$ say. The above arguments simply yield that $\lim_{j\rightarrow +\infty}F_{j}^{\ast}({\bf t} ;x)=F^{\ast}({\bf t} ;x)$ uniformly for ${\bf t}\in I$ and $x\in B$.
Further on, observe that for each $j\in {\mathbb N}$ we have:
\begin{align*}
\Bigl\| F\bigl({\bf t} &+(b_{k_{l}}^{1},\cdot \cdot\cdot, b_{k_{l}}^{n});x+x_{k_{l}}\bigr)-F^{\ast}({\bf t};x) \Bigr\|_{Y}
\\& \leq \Bigl\| F\bigl({\bf t} +(b_{k_{l}}^{1},\cdot \cdot\cdot, b_{k_{l}}^{n});x+x_{k_{l}}\bigr)-F_{j}\bigl({\bf t} +(b_{k_{l}}^{1},\cdot \cdot\cdot, b_{k_{l}}^{n});x+x_{k_{l}}\bigr) \Bigr\|_{Y}
\\& +\Bigl\| F_{j}\bigl({\bf t} +(b_{k_{l}}^{1},\cdot \cdot\cdot, b_{k_{l}}^{n});x+x_{k_{l}}\bigr)-F^{\ast}_{j}({\bf t};x) \Bigr\|_{Y}
+ \Bigl\| F^{\ast}_{j}({\bf t};x) - F^{\ast}({\bf t};x) \Bigr\|_{Y}.
\end{align*}
It can be simply shown that there exists a number $j_{0}(\epsilon)\in {\mathbb N}$ such that for all integers $j\geq j_{0}$ we have
that the first addend and the third addend in the above estimate are less or greater than $\epsilon/3,$ uniformly for $x\in B$ and ${\bf t}\in I.$ For the second addend, take any integer $l\in {\mathbb N}$
such that
$$
\Bigl\| F_{j}\bigl({\bf t} +(b_{k_{l}}^{1},\cdot \cdot\cdot, b_{k_{l}}^{n});x+x_{k_{l}}\bigr)-F^{\ast}_{j}({\bf t};x) \Bigr\|_{Y}<\epsilon/3,\quad x\in B,\ {\bf t}\in I. 
$$
This completes the
proof in a routine manner.  
\end{proof}

We can similarly deduce the following:

\begin{prop}\label{2.1.10ap1}
Suppose that for each integer $j\in {\mathbb N}$ the function $F_{j}(\cdot ; \cdot)$ is $({\mathrm R},{\mathcal B})$-multi-almost periodic
and, for every sequence  which belongs to 
${\mathrm R},$ any its subsequence also belongs to ${\mathrm R}.$ If for each $B\in {\mathcal B}$ there exists $\epsilon_{B}>0$ such that
the sequence $(F_{j}(\cdot ;\cdot))$ converges uniformly to a function $F(\cdot ;\cdot)$ on the set $B^{\circ} \cup \bigcup_{x\in \partial B}B(x,\epsilon_{B}),$ then the function $F(\cdot ;\cdot)$ is $({\mathrm R},{\mathcal B})$-multi-almost periodic.
\end{prop}

\begin{proof}
The proof is almost completely the same as the proof of the previous proposition and we will only emphasize the main differences.
The first difference is with regards to the continuity of function $F(\cdot;\cdot)$ at 
$({\bf t};x),$ 
where
${\bf t}\in I$ and $x\in X$ are given in advance. As above, we have the existence of a set $B\in {\mathcal B}$ such that $x\in B.$ Since the sequence $(F_{j}(\cdot ;\cdot))$ converges uniformly to a function $F(\cdot ;\cdot)$ on the set $B'\equiv B^{\circ} \cup \bigcup_{x\in \partial B}B(x,\epsilon_{B}),$ we have the existence of a positive integer $n_{0}\in {\mathbb N}$ such that $\|F_{n_{0}}({\bf t}' ;x')-F({\bf t}' ;x')\|_{Y}\leq \epsilon/3$ for all ${\bf t}'\in I$ and $x'\in B'.$ After that, it suffices to apply \eqref{nijenorm} and the continuity of 
$F_{n_{0}}(\cdot;\cdot)$ at 
$({\bf t};x)$ (it should be noted that this part can be applied for proving the continuity of function $F(\cdot;\cdot)$ at 
$({\bf t};x)$ in the previous proposition under this weaker condition). The second difference is with regards to the 
uniform continuity in the equation \eqref{nutela}; in Proposition \ref{2.1.10ap}, it is necessary to assume that the sequence $(F_{j}(\cdot ;\cdot))$ converges uniformly to a function $F(\cdot ;\cdot)$ on the whole space $X.$ In the newly arisen situation, it suffices to assume that the sequence $(F_{j}(\cdot ;\cdot))$ converges uniformly to a function $F(\cdot ;\cdot)$ on the set $B,$ only.  
\end{proof}

The following special cases will be very important for us in the sequel:
\begin{itemize}
\item[L1.] ${\mathrm R}=\{b : {\mathbb N} \rightarrow {\mathbb R}^{n} \, ; \, \mbox{ for all }j\in {\mathbb N}\mbox{ we have }b_{j}\in \{ (a,a,a,\cdot \cdot \cdot, a) \in {\mathbb R}^{n} : a\in {\mathbb R}\}\}.$  
If $n=2$ and ${\mathcal B}$ denotes the collection of all bounded subsets of $X,$ then we also say that the function $F(\cdot;\cdot)$ is bi-almost periodic. 

The notion of bi-almost periodicity plays an incredible role in the research study \cite{koyunch} by
H. C. Koyuncuo\v glu and M. Adıvar, where the authors have analyzed the existence of almost periodic solutions for a class of discrete Volterra systems and 
the research study \cite{pinto-vidal} by M. Pinto and C. Vidal, where the authors have used the notion of integrable bi-almost-periodic Green functions of linear homogeneous differential equations and the Banach contraction 
principle to show the existence of almost and pseudo-almost periodic mild solutions for a class of the abstract differential equations with constant delay (see also the research article \cite{chavezinjo}, where A. Ch\'avez, S. Castillo and M. Pinto
have used the notion of  bi-almost automorphy
in their investigation of
almost automorphic solutions of abstract differential
equations with piecewise constant arguments).
The notion of $k$-bi-almost periodicity was introduced by M. Pinto in \cite{k-bi-almost}
and further analyzed in \cite[Section 4]{coronel}, where the authors have analyzed the existence and uniqueness of weighted pseudo almost periodic solutions for a class of abstract integro-differential equations.
\item[L2.] ${\mathrm R}$ is the collection of all sequences $b(\cdot)$ in ${\mathbb R}^{n},$ resp. ${\mathrm R}_{\mathrm X}$ is the collection of all sequences in ${\mathbb R}^{n} \times X$. This is the limit case in our analysis because, in this case, any 
$({\mathrm R},{\mathcal B})$-multi-almost periodic function, resp. $({\mathrm R}_{\mathrm X},{\mathcal B})$-multi-almost periodic function, is automatically $({\mathrm R}_{1},{\mathcal B})$-multi-almost periodic, resp. $({\mathrm R_{1}}_{\mathrm X},{\mathcal B})$-multi-almost periodic, for any other collection ${\mathrm R}_{1}$ of sequences $b(\cdot)$ in ${\mathbb R}^{n},$ resp. any other collection ${\mathrm R_{1}}_{\mathrm X}$ is the collection of sequences in ${\mathbb R}^{n} \times X$. 
\end{itemize}

Concerning Bohr type definitions, we will consider first the following notion (see also the paper \cite{perov} by A. I. Perov and T. K. Kacaran): 

\begin{defn}\label{nafaks1234567890}
Suppose that $\emptyset  \neq I \subseteq {\mathbb R}^{n},$ $F : I \times X \rightarrow Y$ is a continuous function and $I +I \subseteq I.$ Then we say that:
\begin{itemize}
\item[(i)]\index{function!Bohr ${\mathcal B}$-almost periodic}
$F(\cdot;\cdot)$ is Bohr ${\mathcal B}$-almost periodic if and only if for every $B\in {\mathcal B}$ and $\epsilon>0$
there exists $l>0$ such that for each ${\bf t}_{0} \in I$ there exists ${\bf \tau} \in B({\bf t}_{0},l) \cap I$ such that
\begin{align}\label{emojmarko}
\bigl\|F({\bf t}+{\bf \tau};x)-F({\bf t};x)\bigr\|_{Y} \leq \epsilon,\quad {\bf t}\in I,\ x\in B.
\end{align}
\item[(ii)] \index{function!${\mathcal B}$-uniformly recurrent}
$F(\cdot;\cdot)$ is ${\mathcal B}$-uniformly recurrent if and only if for every $B\in {\mathcal B}$ 
there exists a sequence $({\bf \tau}_{k})$ in $I$ such that $\lim_{k\rightarrow +\infty} |{\bf \tau}_{k}|=+\infty$ and
\begin{align}\label{jadnice}
\lim_{k\rightarrow +\infty}\sup_{{\bf t}\in I;x\in B} \bigl\|F({\bf t}+{\bf \tau}_{k};x)-F({\bf t};x)\bigr\|_{Y} =0.
\end{align}
\end{itemize}
If $X\in {\mathcal B},$ then it is also said that $F(\cdot;\cdot)$ is Bohr almost periodic (uniformly recurrent).
\end{defn}

\begin{rem}\label{multi33}
Suppose that $F : I \times X \rightarrow Y$ is a continuous function. 
If ${\mathcal B}'$ is a certain collection of subsets of $X$ which contains ${\mathcal B},$ 
${\mathrm R}'$ is a certain collection of sequences in ${\mathbb R}^{n}$ which contains ${\mathrm R}$ and the equation \eqref{lepolepo} holds with the family ${\mathrm R}$ replaced with the family ${\mathrm R}'$ therein, resp. ${\mathrm R}'_{\mathrm X}$ is a certain collection of sequences in ${\mathbb R}^{n} \times X$ which contains ${\mathrm R}_{\mathrm X}$ and  the equation \eqref{lepolepo121} holds with the family ${\mathrm R}_{\mathrm X}$ replaced with the family ${\mathrm R}'_{\mathrm X}$ therein. If
$F(\cdot ;\cdot)$ is $({\mathrm R}',{\mathcal B}')$-multi-almost periodic, resp. $({\mathrm R}'_{\mathrm X},{\mathcal B}')$-multi-almost periodic, then 
$F(\cdot ;\cdot)$ is $({\mathrm R},{\mathcal B})$-multi-almost periodic, resp. $({\mathrm R}_{\mathrm X},{\mathcal B})$-multi-almost periodic. Similarly, if $F(\cdot;\cdot)$ is Bohr ${\mathcal B}'$-almost periodic (${\mathcal B}$-uniformly recurrent) for some family ${\mathcal B}'$ which contains ${\mathcal B},$ then $F(\cdot;\cdot)$ is Bohr ${\mathcal B}$-almost periodic (${\mathcal B}$-uniformly recurrent). Therefore, it is important to know the maximal collections ${\mathcal B},$ ${\mathrm R}$ and ${\mathrm R}_{X},$ with the meaning clear, for which the function $F(\cdot;\cdot)$ is $({\mathrm R},{\mathcal B})$-multi-almost periodic, $({\mathrm R}_{\mathrm X},{\mathcal B})$-multi-almost periodic, Bohr ${\mathcal B}$-almost periodic or ${\mathcal B}$-uniformly recurrent.
\end{rem}

It is clear that any Bohr (${\mathcal B}$-)almost periodic function is (${\mathcal B}$-)uniformly recurrent; in general, the converse statement does not hold (\cite{nova-man}). It is also clear
that, if $F(\cdot;\cdot)$ is ${\mathcal B}$-uniformly recurrent and $x\in X$, then we have the following supremum formula
\begin{align}\label{tupak12345}
\sup_{{\bf t}\in I}\bigl\|F({\bf t};x) \bigr\|_{Y}=\sup_{{\bf t}\in I,|t|\geq a}\bigl\|F({\bf t};x) \bigr\|_{Y},
\end{align}
which in particular shows that for each $x\in X$ the function $F(\cdot;x)$ is identically equal to zero provided that the function $F(\cdot;\cdot)$ is ${\mathcal B}$-uniformly recurrent and $\lim_{t\in I,|t|\rightarrow +\infty}F({\bf t};x)=0.$ The statements of \cite[Theorem 7, p. 3]{besik} and Proposition \ref{convdiaggas} can be reformulated in this framework, as well.

Keeping in mind the proof of \cite[Property 4, p. 3]{18}, the following result can be proved as in the one-dimensional case:

\begin{prop}\label{superstebag}
Suppose that $F : I \times X \rightarrow Y$ is $({\mathrm R},{\mathcal B})$-multi-almost periodic, resp. $({\mathrm R}_{\mathrm X},{\mathcal B})$-multi-almost periodic (Bohr ${\mathcal B}$-almost periodic/${\mathcal B}$-uniformly recurrent), and $\phi : Y \rightarrow Z$ is uniformly continuous on $\overline{R(F)}$.
Then $\phi \circ F : I \times X \rightarrow Z$ is $({\mathrm R},{\mathcal B})$-multi-almost periodic, resp. $({\mathrm R}_{\mathrm X},{\mathcal B})$-multi-almost periodic (Bohr ${\mathcal B}$-almost periodic/${\mathcal B}$-uniformly recurrent).
\end{prop}

We continue by providing several illustrative examples and useful observations:

\begin{example}\label{pwerqwerzeka}
In contrast with the class of Bohr ${\mathcal B}$-almost periodic functions, we can simply construct a great number of multi-dimensional ${\mathcal B}$-uniformly recurrent functions by using Proposition \ref{superstebag} and
the fact that for any given tuple ${\bf a}=(a_{1},\cdot \cdot \cdot,a_{n}) \in {\mathbb R}^{n} \neq 0$, the linear function
$$
g({\bf t}):=a_{1}t_{1}+\cdot \cdot \cdot +a_{n}t_{n},\quad {\bf t}=(t_{1},\cdot \cdot \cdot,t_{n})\in {\mathbb R}^{n}
$$
is uniformly recurrent provided that $n>1.$ To verify this, it suffices to observe that the set $W:=\{(t_{1},\cdot \cdot \cdot,t_{n})\in {\mathbb R}^{n} : a_{1}t_{1}+\cdot \cdot \cdot +a_{n}t_{n}=0\}$ is a non-trivial linear submanifold of ${\mathbb R}^{n}$
as well as that $g({\bf t}+{\bf t}')=g({\bf t})$ for all ${\bf t}\in {\mathbb R}^{n}$ and ${\bf t}'\in W.$ Therefore, for any uniformly continuous function $\phi : {\mathbb R} \rightarrow X,$ we have that the function
$\phi \circ g : {\mathbb R}^{n} \rightarrow X$ is uniformly recurrent.
\end{example}

\begin{example}\label{pwerqwer}
\begin{itemize}
\item[(i)] Suppose that $F_{j} : X \rightarrow Y$ is a continuous function, for each $B\in {\mathcal B}$ we have $\sup_{x\in B}\| F_{j}(x) \|_{Y}<\infty$ 
and the complex-valued mapping $t\mapsto \int_{0}^{t}f_{j}(s)\, ds,$ $t\geq 0$ is almost periodic ($1\leq j \leq n$). Set
\begin{align*}
F\bigl(t_{1},\cdot \cdot \cdot,t_{n+1}; x\bigr):=\sum_{j=1}^{n}\int_{t_{j}}^{t_{j+1}}f_{j}(s)\, ds \cdot F_{j}(x)\ \mbox{ for all }x\in X \mbox{ and } t_{j}\geq 0,\ 1\leq j\leq n.
\end{align*}
Then the mapping $F: [0,\infty)^{n+1} \times X \rightarrow X$ is Bohr ${\mathcal B}$-almost periodic. In actual fact, for every $B\in {\mathcal B}$ and $\epsilon>0,$ we have 
\begin{align*}
\Bigl\|F\bigl(& t_{1}+\tau_{1},\cdot \cdot \cdot,t_{n+1}+\tau_{n+1}; x\bigr)-F\bigl(t_{1},\cdot \cdot \cdot,t_{n+1}; x\bigr)\Bigr\|_{Y}
\\ & \leq  \sum_{j=1}^{n}\Biggl| \int^{t_{j+1}+\tau_{j+1}}_{t_{j}+\tau_{j}}f_{j}(s)\, ds-\int^{t_{j+1}}_{t_{j}}f_{j}(s)\, ds\Biggr| \cdot \bigl\|F_{j}(x)\bigr\|_{Y}
\\ &\leq   \sum_{j=1}^{n}\Biggl\{\Biggl| \int^{t_{j+1}+\tau_{j+1}}_{0}f_{j}(s)\, ds-\int^{t_{j+1}}_{0} f_{j}(s)\, ds\Biggr|
\\ & + \Biggl|\int^{t_{j}+\tau_{j}}_{0}f_{j}(s)\, ds-\int^{t_{j}}_{0}f_{j}(s)\, ds\Biggr|\Biggr\} \cdot \bigl\|F_{j}(x)\bigr\|_{Y}
\\ & \leq   M\sum_{j=1}^{n}\Biggl\{\Biggl| \int^{t_{j+1}+\tau_{j+1}}_{0}f_{j}(s)\, ds-\int^{t_{j+1}}_{0} f_{j}(s)\, ds\Biggr|
\\ & + \Biggl|\int^{t_{j}+\tau_{j}}_{0}f_{j}(s)\, ds-\int^{t_{j}}_{0}f_{j}(s)\, ds\Biggr|\Biggr\},
\end{align*}
where $M=\sup_{x\in B, 1\leq j\leq n}\bigl\|F_{j}(x)\bigr\|_{Y}.$ The corresponding statement follows by considering the common $\epsilon/(2nM)$-periods $\tau_{j}$ 
of the functions $\int_{0}^{\cdot}f_{j-1}(s)\, ds$ and $\int_{0}^{\cdot}f_{j}(s)\, ds$ for $2\leq j \leq n,$ the $\epsilon/(2nM)$-periods $\tau_{1}$ of the function $\int_{0}^{\cdot}f_{1}(s)\, ds$ and 
the $\epsilon/(2nM)$-periods $\tau_{n+1}$ of the function $\int_{0}^{\cdot}f_{n+1}(s)\, ds.$
Further on,
let us denote by $G_{j}(\cdot)$ the unique almost periodic extension of the function $t\mapsto \int_{0}^{t}f_{j}(s)\, ds,$ $t\geq 0$ to the whole real line ($1\leq j \leq n$). Let $({\bf b}_{k})$ be any sequence in ${\mathbb R}^{n+1}.$ Then we can use Theorem \ref{Bochner123456} below to conclude that there exists a subsequence
$({\bf b}_{k_{l}})$ of $({\bf b}_{k})$ such that, for every $j\in {\mathbb N}_{n},$   $F_{j}(t_{j}+b_{k_{l}}^{j},t_{j+1}+b_{k_{l}}^{j+1})=G_{j}(t_{j+1}+b_{k_{l}}^{j+1})-G_{j}(t_{j}+b_{k_{l}}^{j})$
converges to a function $F_{j}^{\ast}(t_{j},t_{j+1})$ as $l\rightarrow +\infty ,$ uniformly for $(t_{j},t_{j+1})\in {\mathbb R}^{2}.$ Define
\begin{align*}
F\bigl(t_{1},\cdot \cdot \cdot,t_{n+1}; x\bigr):=\sum_{j=1}^{n}F_{j}^{\ast}\bigl(t_{j},t_{j+1}\bigr)F_{j}(x),\mbox{ for all }x\in X \mbox{ and } t_{j}\geq 0,\ 1\leq j\leq n.
\end{align*} 
Let $B\in {\mathcal B}$ be fixed. Then it can be simply shown that 
$$
\lim_{l\rightarrow +\infty}F\bigl(t_{1}+b_{k_{l}}^{1},\cdot \cdot \cdot,t_{n+1} +b_{k_{l}}^{n+1};x\bigr)=F\bigl(t_{1},\cdot \cdot \cdot,t_{n+1};x\bigr)
$$
uniformly for $x\in B$ and ${\bf t}=(t_{1},\cdot \cdot \cdot,t_{n+1}) \in {\mathbb R}^{n+1}.$ Hence, 
the function $F_{1}(\cdot)$ is ${\mathrm R}$-multi-almost periodic with ${\mathrm R}$ being the collection of all sequences in ${\mathbb R}^{n+1}.$
\item[(ii)] Suppose that $F : X \rightarrow Y$ is a continuous function, for each $B\in {\mathcal B}$ we have $\sup_{x\in B}\| F(x) \|_{Y}<\infty$ and the complex-valued mapping $t\mapsto f_{j}(t),$ $t\geq 0$ is almost periodic, resp. bounded and uniformly recurrent ($1\leq j \leq n$). Set
\begin{align*}
F\bigl(t_{1},\cdot \cdot \cdot,t_{n}; x\bigr):=\prod_{j=1}^{n}f_{j}\bigl(t_{j}\bigr)\cdot F(x)\ \mbox{ for all }x\in X \mbox{ and } t_{j}\geq 0,\ 1\leq j\leq n.
\end{align*}
Then the mapping $F: [0,\infty)^{n} \times X \rightarrow X$ is Bohr ${\mathcal B}$-almost periodic, resp. ${\mathcal B}$-uniformly recurrent. In actual fact, for every $B\in {\mathcal B}$ and $\epsilon>0,$ we have
\begin{align*}
& \Bigl\| F\bigl( t_{1}+\tau_{1},\cdot \cdot \cdot,t_{n}+\tau_{n}; x\bigr)-F\bigl(t_{1},\cdot \cdot \cdot,t_{n}; x\bigr)\Bigr\|_{Y}
\\ & \leq M\Biggl\{ \bigl| f_{1}(t_{1}+\tau_{1}) -f_{1}(t_{1})\bigr| \cdot \prod_{j=2}^{n}\bigl|f_{j}(t_{j}+\tau_{j})\bigr|
 +\bigl| f_{1}(t_{1}) \bigr| \cdot  \prod_{j=2}^{n}\bigl| f_{j}(t_{j}+\tau_{j})-f_{j}(t_{j})\bigr|\Biggr\}
\\& \leq M\bigl| f_{1}(t_{1}+\tau_{1}) -f_{1}(t_{1})\bigr| \cdot \prod_{j=2}^{n}\bigl\|f_{j}\bigr\|_{\infty}
 +\bigl\| f_{1} \bigr\|_{\infty} \cdot \prod_{j=2}^{n}\bigl|f_{j}(t_{j}+\tau_{j})-f_{j}(t_{j})\bigr|,
\end{align*}
where $M=\sup_{x\in B}\bigl\|F(x)\bigr\|_{Y}.$ Repeating this procedure, we simply get the required statement; furthermore, we can use the usual Bochner criterion and repeat the above calculus in order to see that the function $F_{2}(\cdot)$ is ${\mathrm R}$-multi-almost periodic with ${\mathrm R}$ being the collection of all sequences in ${\mathbb R}^{n}.$
\item[(iii)] Suppose that $G : [0,\infty)^{n} \rightarrow {\mathbb C}$ is almost periodic, resp. bounded and uniformly recurrent, $F: [0,\infty) \times X \rightarrow Y$ is Bohr ${\mathcal B}$-almost periodic, resp. ${\mathcal B}$-uniformly recurrent, and for each set $B\in {\mathcal B}$ we have $\sup_{t\geq 0; x\in B}\| F(t;x) \|_{Y}<\infty.$ Set
\begin{align*}
F\bigl(t_{1},\cdot \cdot \cdot,t_{n+1}; x\bigr)&:=G\bigl(t_{1},\cdot \cdot \cdot, t_{n}\bigr)\cdot F\bigl(t_{n+1};x\bigr)\\& \mbox{ for all }x\in X \mbox{ and } t_{j}\geq 0,\ 1\leq j\leq n+1.
\end{align*}
Then the mapping $F: [0,\infty)^{n+1} \times X \rightarrow Y$ is Bohr ${\mathcal B}$-almost periodic, resp. ${\mathcal B}$-uniformly recurrent, which can be simply shown by using the estimate ($t_{i},\ \tau_{i} \geq 0$ for $1\leq i\leq n+1;$ $x\in X$):
\begin{align*}
& \Bigl\| F\bigl( t_{1}+\tau_{1},\cdot \cdot \cdot,t_{n+1}+\tau_{n+1}; x\bigr)-F\bigl(t_{1},\cdot \cdot \cdot,t_{n+1}; x\bigr)\Bigr\|_{Y}
\\ & \leq \Bigl| G\bigl( t_{1}+\tau_{1},\cdot \cdot \cdot,t_{n}+\tau_{n} \bigr) - G\bigl( t_{1},\cdot \cdot \cdot,t_{n} \bigr) \Bigr| \cdot \Bigl\| F\bigl( t_{n+1}+\tau_{n+1};x\bigr)\Bigr\|_{Y}
\\& + \Bigl| G\bigl( t_{1},\cdot \cdot \cdot,t_{n}\bigr)\Bigr| \cdot \Bigl\| F\bigl( t_{n+1}+\tau_{n+1};x \bigr)-F\bigl( t_{n+1} ;x \bigr)\Bigr\|_{Y}, 
\end{align*}
the boundedness of function $G(\cdot,...,\cdot)$ and the assumption that  for each set $B\in {\mathcal B}$ we have $\sup_{t\geq 0; x\in B}\| F(t;x) \|_{Y}<\infty$ (see also Proposition \ref{bounded-pazi} below).
\end{itemize}
\end{example}

It is worth noting that we can extend the notion introduced in Definition \ref{nafaks1234567890} as follows:

\begin{defn}\label{nafaks123456789012345}
Suppose that $\emptyset  \neq I'\subseteq I \subseteq {\mathbb R}^{n},$ $F : I \times X \rightarrow Y$ is a continuous function and $I +I' \subseteq I.$ Then we say that:
\begin{itemize}
\item[(i)]\index{function!Bohr $({\mathcal B},I')$-almost periodic}
$F(\cdot;\cdot)$ is Bohr $({\mathcal B},I')$-almost periodic if and only if for every $B\in {\mathcal B}$ and $\epsilon>0$
there exists $l>0$ such that for each ${\bf t}_{0} \in I'$ there exists ${\bf \tau} \in B({\bf t}_{0},l) \cap I'$ such that
\begin{align*}
\bigl\|F({\bf t}+{\bf \tau};x)-F({\bf t};x)\bigr\|_{Y} \leq \epsilon,\quad {\bf t}\in I,\ x\in B.
\end{align*}
\item[(ii)] \index{function!$({\mathcal B},I')$-uniformly recurrent}
$F(\cdot;\cdot)$ is $({\mathcal B},I')$-uniformly recurrent if and only if for every $B\in {\mathcal B}$ 
there exists a sequence $({\bf \tau}_{n})$ in $I'$ such that $\lim_{n\rightarrow +\infty} |{\bf \tau}_{n}|=+\infty$ and
$$
\lim_{n\rightarrow +\infty}\sup_{{\bf t}\in I;x\in B} \bigl\|F({\bf t}+{\bf \tau}_{n};x)-F({\bf t};x)\bigr\|_{Y} =0.
$$
\end{itemize}
If $X\in {\mathcal B},$ then it is also said that $F(\cdot;\cdot)$ is Bohr $I'$-almost periodic ($I'$-uniformly recurrent).
\end{defn}

Clearly, 
the notion from Definition \ref{nafaks1234567890} is recovered by plugging $I'=I$ and any 
$({\mathcal B},I')$-uniformly recurrent function is $({\mathcal B},I)$-uniformly recurrent
provided that $I+I\subseteq I.$ This is not true for almost periodicity: we can simply construct a great number of corresponding examples showing that the notion of $({\mathcal B},I')$-almost periodicity is neither stronger nor weaker than the notion of $({\mathcal B},I)$-almost periodicity, provided that $I+I\subseteq I.$  
In many concrete situations, the situation in which $I'\neq I$ can occur;
for example, we have the following:

\begin{example}\label{rajkomilice}
\begin{itemize}
\item[(i)]
Suppose that 
the complex-valued mapping $t\mapsto \int_{0}^{t}f_{j}(s)\, ds,$ $t\geq 0$ is almost periodic, resp. bounded and uniformly recurrent ($1\leq j \leq n$). Set
\begin{align*}
F_{1}\bigl(t_{1},\cdot \cdot \cdot,t_{2n}\bigr):=\prod_{j=1}^{n}\int_{t_{j}}^{t_{j+n}}f_{j}(s)\, ds\ \mbox{ and } t_{j}\in {\mathbb R},\ 1\leq j\leq 2n.
\end{align*}
Then the argumentation used in Example \ref{pwerqwer}(i)-(ii) shows that the mapping $F_{1}: {\mathbb R}^{2n} \rightarrow {\mathbb C}$ is both Bohr $I'$-almost periodic, resp. $I'$-uniformly recurrent, where
$I'=\{({\bf \tau},{\bf \tau}) : {\bf \tau} \in {\mathbb R^{n}} \};$ in the case of consideration of almost periodicity, we can use \cite[Theorem 2.1.1(xiv)]{nova-mono}
in order to see that $F_{1}(\cdot)$ is also Bohr $I''$-almost periodic, where
$I''=\{(a,a,\cdot \cdot \cdot, a) \in {\mathbb R}^{2n} : a\in {\mathbb R}\}.
$ Furthermore, in the same case, we can use Theorem \ref{Bochner123456} below and the usual Bochner criterion for the functions of one real variable to see that the function 
$F_{1}(\cdot)$ is Bohr almost periodic because it is ${\mathrm R}$-almost periodic with ${\mathrm R}$ being the collection of all sequences in ${\mathbb R}^{2n}.$
\item[(ii)] Suppose that 
an $X$-valued mapping $t\mapsto \int_{0}^{t}f_{j}(s)\, ds,$ $t\in {\mathbb R}$ is almost periodic, resp. bounded and uniformly recurrent, as well as that a strongly continuous operator family $(T_{j}(t))_{t\in {\mathbb R}}\subseteq L(X,Y)$ is uniformly bounded ($1\leq j \leq n$).
Set
\begin{align*}
F_{2}\bigl(t_{1},\cdot \cdot \cdot,t_{2n}\bigr)&:=\sum_{j=1}^{n}T_{j}(t_{j}-t_{j+n})\int_{t_{j}}^{t_{j+n}}f_{j}(s)\, ds\\&  \mbox{ and } t_{j}\in {\mathbb R},\ 1\leq j\leq 2n.
\end{align*}
Since, for every $t_{i},\ \tau_{i}\in {\mathbb R}$ ($1\leq j\leq 2n$) with $\tau_{j}=\tau_{j+n}$ ($1\leq j\leq n$), we have
\begin{align*}
& \bigl\|F_{2}\bigl(t_{1}+\tau_{1},\cdot \cdot \cdot,t_{2n}+\tau_{2n}\bigr)-F_{2}\bigl(t_{1},\cdot \cdot \cdot,t_{2n}\bigr)\bigr\|_{Y}
\\& \leq M\sum_{j=1}^{n}\Biggl\{\Biggl| \int^{t_{j}+\tau_{j}}_{0}f_{j}(s)\, ds - \int^{t_{j}}_{0}f_{j}(s)\, ds\Biggr|+\Biggl| \int^{t_{j+n}+\tau_{j}}_{0}f_{j}(s)\, ds - \int^{t_{j+n}}_{0}f_{j}(s)\, ds\Biggr|\Biggr\},
\end{align*}
where $M=\sup_{1\leq j\leq n}\sup_{t\in {\mathbb R}}\|T_{j}(t)\|,$
we may conclude as above that the mapping $F_{2}: {\mathbb R}^{2n} \rightarrow {\mathbb C}$ is Bohr $I'$-almost periodic, resp. $I'$-uniformly recurrent, where $I'=\{({\bf \tau},{\bf \tau}) : {\bf \tau} \in {\mathbb R^{n}} \}$, but not generally almost periodic,
in the case of consideration of almost periodicity; in this case,
we also have that $F_{2}(\cdot)$ is
Bohr $I''$-almost periodic, where
$I''=\{(a,a,\cdot \cdot \cdot, a) \in {\mathbb R}^{2n} : a\in {\mathbb R}\},
$ and that the function $F_{2}(\cdot)$ is ${\mathrm R}$-multi-almost periodic with ${\mathrm R}$ being the collection of all sequences in $I'.$ 
\item[(iii)] Suppose that 
an $X$-valued mapping $t\mapsto f_{j}(t),$ $t\in {\mathbb R}$ is almost periodic, resp. bounded and uniformly recurrent, $F_{j} : X\rightarrow X$ is continuous as well as that a strongly continuous operator family $(T_{j}(t))_{t\in {\mathbb R}}\subseteq L(X,Y)$ satisfies that $\|T_{j}(t)\| \leq    
M_{j}e^{-\omega_{j}|t|},$ $t\in {\mathbb R} $ with $\omega_{j}>\|f_{j}\|_{\infty}$ ($1\leq j \leq n$). Set
\begin{align*}
{\mathbb F}_{3}\bigl(t_{1},\cdot \cdot \cdot,t_{2n};x\bigr)&:=\sum_{j=1}^{n}e^{\int_{t_{j}}^{t_{j+n}}f_{j}(s)\, ds}T_{j}(t_{j}-t_{j+n})F_{j}(x)\\& \mbox{ for all }x\in X \mbox{ and } t_{j}\in {\mathbb R},\ 1\leq j\leq 2n.
\end{align*}
Suppose, additionally, that for each set $B\in {\mathcal B}$ we have $$\sup_{1\leq j\leq n;x\in B}\|F_{j}(x)\|<\infty.$$
Arguing as above (see also \cite[Example 4.1]{prcko-baja}), we may conclude with the help of the elementary inequality $|e^{z}-1|\leq |z| \cdot e^{|z|},$ $z\in {\mathbb C}$ that the mapping ${\mathbb F}_{3}: {\mathbb R}^{2n} \times X \rightarrow Y$ is Bohr $({\mathcal B},I')$-almost periodic, resp. $({\mathcal B},I')$-uniformly recurrent, where $I'=\{({\bf \tau},{\bf \tau}) : {\bf \tau} \in {\mathbb R^{n}} \}$, but not generally Bohr ${\mathcal B}$-almost periodic,  
in the case of consideration of almost periodicity; in this case,
we also have that ${\mathbb F}_{3}(\cdot;\cdot)$ is
Bohr $({\mathcal B},I'')$-almost periodic, where
$I''=\{(a,a,\cdot \cdot \cdot, a) \in {\mathbb R}^{2n} : a\in {\mathbb R}\},
$ and that the function $F_{2}(\cdot)$ is ${\mathrm R}$-multi-almost periodic with ${\mathrm R}$ being the collection of all sequences in $I'.$
\item[(iv)] Suppose that $-\infty \leq \alpha<\beta \leq +\infty$ and $f : \Omega \equiv \{z\in {\mathbb C}: \alpha<\Re z<\beta \} \rightarrow X$ is an analytic almost periodic function (see e.g., \cite[Appendix 3]{pankov}). 
Set, for $\alpha<\alpha'<\beta'<\beta,$ $I_{\alpha',\beta'}:=[\alpha',\beta'] \times {\mathbb R},$  $I_{\alpha',\beta'}':=\{0\} \times {\mathbb R}$ and $F(x,y):=f(x+iy),$ $(x,y)\in I_{\alpha',\beta'}.$ Then $F(\cdot,\cdot)$ is Bohr  $I_{\alpha',\beta'}'$-almost periodic.
\item[(v)]
In connection with Example \ref{pwerqwerzeka} and the notion introduced in Definition \ref{nafaks123456789012345}, the following should be stated:
Suppose that $\emptyset  \neq I \subseteq {\mathbb R}^{n},$ 
$ I_{0}=[0,\infty)$ or $ I_{0}={\mathbb R},$ ${\bf a}=(a_{1},\cdot \cdot \cdot,a_{n}) \in {\mathbb R}^{n} \neq 0$ and the linear function
$
g({\bf t}):=a_{1}t_{1}+\cdot \cdot \cdot +a_{n}t_{n},$ ${\bf t}=(t_{1},\cdot \cdot \cdot,t_{n})\in I
$
maps surjectively the region $I$
onto $ I_{0}.$ Suppose, further, that $f: I_{0} \rightarrow X$ is a uniformly recurrent function as well as that a sequence $(\alpha_{k})$ in $I_{0}$ satisfies 
that $\lim_{k\rightarrow +\infty} |\alpha_{k}|=+\infty$ and 
$\lim_{k\rightarrow +\infty}\sup_{t\in I_{0}} \bigl\|f(t+\alpha_{k})-f(t)\bigr\| =0.$ 
Define $I':=g^{-1}(\{\alpha_{k} : k\in {\mathbb N}\})$ and
$F : I \rightarrow X$ by
$
F({\bf t}):=f(g({\bf t})),$ ${\bf t}\in I.
$
Then $F(\cdot)$ is $I'$-uniformly recurrent, and $F(\cdot )$ is not almost periodic provided that $f(\cdot)$ is not almost periodic.
In order to see this,
observe that the surjectivity of mapping $g : I \rightarrow I_{0}$ implies the existence of a sequence $({\bf \tau}_{k})$ in $I'$ such that $g({\bf \tau}_{k})=\alpha_{k}$ for all $k\in {\mathbb N}.$ Due to the Cauchy-Schwarz inequality, we have
$|{\bf \tau}_{k}|\geq |\alpha_{k}|/|a| \rightarrow +\infty$ as $k\rightarrow +\infty.$ Furthermore, for every ${\bf t}\in I,$ we have:
\begin{align*}
\Bigl\| & F({\bf t}+{\bf \tau}_{k})-F({\bf t})\Bigr\|
\\& \leq \Bigl\| f\bigl(g({\bf t}+{\bf \tau}_{k})\bigr)-f\bigl(g({\bf t})+\alpha_{k}\bigr) \Bigr\|+\Bigl\| f\bigl(g({\bf t})+\alpha_{k}\bigr) -f\bigl(g({\bf t})\bigr)\Bigr\|
\\ & =\Bigl\| f\bigl(g({\bf t}+{\bf \tau}_{k})\bigr)-f\bigl(g({\bf t}+{\bf \tau}_{k})\bigr) \Bigr\|+\Bigl\| f\bigl(g({\bf t})+\alpha_{k}\bigr) -f\bigl(g({\bf t})\bigr)\Bigr\|
\\& =\Bigl\| f\bigl(g({\bf t})+\alpha_{k}\bigr) -f\bigl(g({\bf t})\bigr)\Bigr\| \leq \sup_{t\in I_{0}}\Bigl\| f\bigl(t+\alpha_{k}\bigr) -f\bigl(t\bigr)\Bigr\|\rightarrow 0,
\end{align*}
as $k\rightarrow +\infty.$ Suppose now that $f(\cdot)$ is not almost periodic. We will prove that $F(\cdot)$ is not almost periodic, as well. Let $l>0$ be arbitrary. Due to our assumption,
there exists $\epsilon>0$ such that there exists a subinterval $I''\subseteq I_{0}$ of length $2|a|l$ such that for each $\tau \in I''$ there exists $t\in I_{0}$ such that $\|f(t+\tau)-f(t)\|>\epsilon.$
Let $i''$ be the center of $I''.$ Then there exists ${\bf t}_{0}\in I$ such that $g({\bf t}_{0})=i''$ and this simply implies that for each ${\bf \alpha} \in B({\bf t}_{0},l) \cap I$ we have $g({\bf \alpha})\in I''.$ Therefore, for fixed ${\bf \alpha}$ from this range, we can find $t\in I_{0}$ such that $\|f(t+\tau)-f(t)\|>\epsilon,$ where $\tau=g({\bf \alpha}).$ By surjectivity of $g(\cdot),$ we have the existence of a tuple ${\bf t}\in I$ such that $g({\bf t})=t,$ which gives the required. 
\item[(vi)] To present a concrete application of our consideration from the previous point, let us recall that A. Haraux and P. Souplet have proved, in \cite[Theorem 1.1]{haraux}, that the
function $f: {\mathbb R}\rightarrow {\mathbb R},$ given by
\begin{align*}
f(t):=\sum_{n=1}^{\infty}\frac{1}{n}\sin^{2}\Bigl(\frac{t}{2^{n}} \Bigr)\, dt,\quad t\in {\mathbb R},
\end{align*}
is uniformly continuous, uniformly recurrent (the sequence $(\alpha_{k}\equiv 2^{k}\pi)_{k\in {\mathbb N}}$ can be chosen in definition of uniform recurrence) and unbounded. Let ${\bf a}=(a_{1},\cdot \cdot \cdot,a_{n}) \in {\mathbb R}^{n} \neq 0,$ let $I'=g^{-1}(\{2^{k}\pi : k\in {\mathbb N}\})$ and let
$F : {\mathbb R}^{n} \rightarrow {\mathbb R}$ be given by $F({\bf t}):=f(a_{1}t_{1}+\cdot \cdot \cdot +a_{n}t_{n}),$ ${\bf t}=(t_{1},\cdot \cdot \cdot,t_{n})\in {\mathbb R}^{n}.$
Then the function $F(\cdot)$ is $I'$-uniformly recurrent and not almost periodic; furthermore, it can be easily shown that the function $F(\cdot)$ is uniformly continuous and unbounded.
\item[(vii)] Suppose that $K$ is a bounded Lebesgue measurable set and $I+K\subseteq I.$ Then we can simply prove that the Bohr $({\mathcal B},I')$-almost periodicity, resp. $({\mathcal B},I')$-uniform recurrence, of function $F :  I\times X\rightarrow Y$ implies the  Bohr $({\mathcal B},I')$-almost periodicity, resp. $({\mathcal B},I')$-uniform recurrence, of function $G :  I\times X\rightarrow Y$ defined by 
$$
G({\bf t}; x):=\int^{{\bf t}+K}_{{\bf t}}F(\sigma ;x)\, d\sigma=\int_{K}F(\sigma +{\bf t}; x)\, d\sigma,\quad {\bf t}\in I,\ x\in X,
$$ 
which extends the conclusions established in \cite[Example 7, p. 33]{30} to the multi-dimensional case; furthermore, if 
$F :  I\times X\rightarrow Y$ is $({\mathrm R},{\mathcal B})$-multi-almost periodic and for each $x\in X$ we have $\sup_{{\bf t} \in I}\| F({\bf t}; x)\|_{Y}<\infty$, resp. $F :  I\times X\rightarrow Y$ is $({\mathrm R}_{X},{\mathcal B})$-multi-almost periodic and for each $B\in {\mathcal B}$,  $x\in B$ and each sequence $(x_{k})$ in $X$ for which there exists 
a sequence $({\bf b}_{k})
$ in $I$ such that $({\bf b}_{k};x_{k})\in {\mathrm R}_{\mathrm X}$
we have $\sup_{{\bf t} \in I}\| F({\bf t}+{\bf b}_{k}; x+x_{k})\|_{Y}<\infty$, then the use of dominated convergence theorem shows that the function $G(\cdot;\cdot)$ is likewise $({\mathrm R},{\mathcal B})$-multi-almost periodic, resp. $({\mathrm R}_{X},{\mathcal B})$-multi-almost periodic.
\item[(viii)] The notion of Bloch $(p,k)$-periodicity can be simply
extended to the functions of several real variables as follows (\cite{has2}-\cite{has1}): a
bounded continuous
function $F:I\rightarrow X$ is said to be Bloch $({\bf p},{\bf k})$-periodic, or Bloch
periodic with period ${\bf p}$ and Bloch wave vector or Floquet exponent ${\bf k},$ where ${\bf p}\in I$ and ${\bf k}\in {\mathbb R}^{n}$ if and only if 
$
F({\bf x}+{\bf p})=e^{i\langle {\bf k}, {\bf p}\rangle}F({\bf x}),$ ${\bf x}\in I
$ (of course, we assume here that ${\bf p}+I \subseteq I$).
Arguing as in \cite[Remark 1]{chelj}, we may conclude that the Bloch $({\bf p},{\bf k})$-periodicity of function $F(\cdot)$ implies the Bohr $({\mathcal B},I')$-almost periodicity of function
$e^{-i\langle {\bf k},\cdot\rangle}F(\cdot)$ with $I'$ being the intersection of $I$ and the one-dimensional submanifold generated  by the vector ${\bf p};$ furthermore, if
${\bf k}$ is orthogonal to ${\bf p},$ then the function $F(\cdot)$ is Bohr $({\mathcal B},I')$-almost periodic.  
\end{itemize}
\end{example}

The previous examples show that the notions of Bohr $I'$-almost periodicity and Bohr $I'$-uniform recurrence are extremely important in the case that $I'\neq I.$ But, we feel it is our duty to say that
the fundamental properties of Bohr $I'$-almost periodic functions and Bohr $I'$-uniformly recurrent functions cannot be so simply clarified in the case that $I'\neq I.$ Because of that, we will basically assume henceforth that $I'=I;$
further analyses of  Bohr $I'$-almost periodicity and Bohr $I'$-uniform recurrence in case $I'\neq I$ will be carried out somewhere else.

It can be simply shown that the subsequent proposition is applicable if $I=[0,\infty)^{n}$ or $I={\mathbb R}^{n}:$

\begin{prop}\label{bounded-pazi}
Suppose that $\emptyset  \neq I \subseteq {\mathbb R}^{n},$ $I +I \subseteq I,$ $I$ is closed,
$F : I \times X \rightarrow Y$ is Bohr ${\mathcal B}$-almost periodic and ${\mathcal B}$ is any family of compact subsets of $X.$ If
\begin{align}
\notag
(\forall l>0) \, (\exists {\bf t_{0}}\in I)\, (\exists k>0) &\, (\forall {\bf t} \in I)(\exists {\bf t_{0}'}\in I)\,
\\\label{profice-zeron} & (\forall {\bf t_{0}''}\in B({\bf t_{0}'},l) \cap I)\, {\bf t}- {\bf t_{0}''} \in B({\bf t_{0}},kl) \cap I.
\end{align}
Then for each $B\in {\mathcal B}$ we have 
that the set $\{ F({\bf t}; x) : {\bf t} \in I,\ x\in B\}$ is relatively compact in $Y;$
in particular,
$\sup_{{\bf t}\in I;x\in B}\|F({\bf t}; x)\|_{Y}<\infty.$
\end{prop}

\begin{proof}
Let $B\in {\mathcal B}$ and $\epsilon>0$ be given. Then we can find a finite number $l>0$ such that 
for each ${\bf s} \in I$ there exists ${\bf \tau} \in B({\bf s},l) \cap I$ such that \eqref{emojmarko} holds with ${\bf t}_{0}={\bf s}.$ Let ${\bf t}_{0}\in I$ and $k>0$ be such that \eqref{profice-zeron} holds.
Since $I$ is closed and $B$ is compact, we have that the set $\{ F({\bf s}; x) : {\bf s} \in B({\bf t_{0}},kl) \cap I,\ x\in B\}$ is
compact in $Y.$
Let ${\bf t}\in I$ be fixed.
By our assumption, there exists ${\bf t_{0}'}\in I$ such that, for every ${\bf t_{0}''}\in B({\bf t_{0}'},l) \cap I,$ we have ${\bf t}\in {\bf t_{0}''} +[B({\bf t_{0}},kl) \cap I].$ On the other hand, there exists ${\bf \tau}={\bf t_{0}''}  \in B({\bf t}_{0}',l) \cap I$ such that $\|F({\bf s}+{\bf \tau};x)-F({\bf s};x)\bigr\|_{Y} \leq \epsilon,$ ${\bf s}\in I,$ $x\in B$. Clearly, ${\bf s}={\bf t}-{\bf \tau} \in B({\bf t_{0}},kl) \cap I,$ which simply implies from the last estimate that 
$\{ F({\bf s}; x) : {\bf s} \in B({\bf t_{0}},kl) \cap I,\ x\in B\}$ is an $\epsilon$-net for $\{ F({\bf t}; x) : {\bf t} \in I,\ x\in B\},$ which completes the proof in a routine manner. 
\end{proof}

Suppose now that $F : {\mathbb R}^{n} \times X \rightarrow Y$ is a Bohr ${\mathcal B}$-almost periodic function, where ${\mathcal B}$ is any family of compact subsets of $X.$
Let $B\in {\mathcal B}$ be fixed. We will consider the Banach space $l_{\infty}(B : Y)$ consisting of all bounded functions $f : B \rightarrow Y,$ equipped with the sup-norm.
Define the function
$F_{B} : {\mathbb R}^{n} \rightarrow l_{\infty}(B : Y)$ by 
\begin{align}\label{sarajevo-london}
\bigl[F_{B}({\bf t})\bigr](x):=F({\bf t}; x),\quad {\bf t}\in {\mathbb R}^{n},\ x\in B.
\end{align}
By Proposition \ref{bounded-pazi}, this mapping is well defined. Furthermore, this mapping satisfies that for each $\epsilon>0$
there exists $l>0$ such that for each ${\bf t}_{0} \in {\mathbb R}^{n}$ there exists ${\bf \tau} \in B({\bf t}_{0},l) \cap {\mathbb R}^{n}$ such that
\begin{align*}
d\bigl(F_{B}({\bf t}+{\bf \tau}), F_{B}({\bf t})\bigr) \leq \epsilon,\quad {\bf t}\in I.
\end{align*}
Hence, $F_{B}(\cdot)$ is Bohr almost periodic in the sense of definition given in \cite[Subsection 1.2, p. 7]{pankov}. By \cite[Theorem 1.2, p. 7]{pankov}, it follows that $F : {\mathbb R}^{n} \times X \rightarrow Y$ is $({\mathrm R}, {\mathcal B})$-multi-almost periodic with ${\mathrm R}$ being the collection of all sequences in ${\mathbb R}^{n}$ (case [L2]). The converse statement can be deduced similarly, and therefore, the following Bochner criterion holds good:

\begin{thm}\label{Bochner123456}
Suppose that $F : {\mathbb R}^{n} \times X \rightarrow Y$ is continuous, ${\mathcal B}$ is any family of compact subsets of $X$ and ${\mathrm R}$ is the collection of all sequences in ${\mathbb R}^{n}.$ Then $F(\cdot ;\cdot)$ is Bohr ${\mathcal B}$-almost periodic if and only if $F(\cdot ;\cdot)$ is $({\mathrm R}, {\mathcal B})$-multi-almost periodic.
\end{thm}

As a direct consequence of Proposition \ref{kursk-kursk}(i) and Theorem \ref{Bochner123456}, we have the following important result for our further investigations (see \cite[pp. 17-24]{fink} for the notion of a uniformly almost periodic family, where the corresponding problematic has been considered for infinite number of almost periodic functions; for the sake of brevity, we will not consider this topic here):

\begin{prop}\label{dekartovproizvod}
Suppose that $k\in {\mathbb N}$ and ${\mathcal B}$ is any family of compact subsets of $X.$
If the function $F_{i}: {\mathbb R}^{n} \times X \rightarrow Y_{i}$ is Bohr ${\mathcal B}$-almost periodic for $1\leq i\leq k$, then the function $(F_{1},\cdot \cdot \cdot, F_{k})(\cdot;\cdot)$ is also Bohr ${\mathcal B}$-almost periodic.
\end{prop}

As a consequence, we have that the Bohr ${\mathcal B}$-almost periodic functions $F_{i}(\cdot;\cdot)$ can share the same $\epsilon$-periods in Definition \ref{nafaks1234567890}(i), i.e., for every $B\in {\mathcal B}$ and $\epsilon>0$
there exists $l>0$ such that for each ${\bf t}_{0} \in {\mathbb R}^{n}$ there exists ${\bf \tau} \in B({\bf t}_{0},l) \cap {\mathbb R}^{n}$ such that
\eqref{emojmarko} holds for all $F=F_{i}$ and $Y=Y_{i},$ $1\leq i\leq k$ (observe that the original proof of H. Bohr, see e.g. \cite[pp. 36-38]{bohr}, does not work in the multi-dimensional case $n>1$).

Now we can simply prove the following:

\begin{prop}\label{pointwise-prod-rn}
Suppose that $f : {\mathbb R}^{n} \rightarrow {\mathbb C}$ is Bohr almost periodic and $F : {\mathbb R}^{n} \times X \rightarrow Y$ is Bohr ${\mathcal B}$-almost periodic, where ${\mathcal B}$ is any family of compact subsets of $X.$
Define $F_{1}({\bf t};x):=f({\bf t})F({\bf t}; x),$ ${\bf t}\in {\mathbb R}^{n},$ $x\in X.$ Then $F_{1}(\cdot ;\cdot)$ is Bohr ${\mathcal B}$-almost periodic.
\end{prop}

\begin{proof}
Let $B\in {\mathcal B}$ and $\epsilon>0$ be fixed.
Due to Proposition \ref{bounded-pazi}, there exists a finite real constant $M\geq 1$ such that $|f({\bf t})|+\|F({\bf t}; x)\|_{Y}\leq M$ for all ${\bf t}\in {\mathbb R}^{n}$ and $x\in B.$ Let $\tau \in {\mathbb R}^{n}$ be a common $(\epsilon/2M)$-period for the functions
$f(\cdot)$
and
$F(\cdot;\cdot).$
Then the required statement simply follows from the next estimates:
\begin{align*}
\bigl\| & f({\bf t}+\tau)F({\bf t}+\tau; x) -f({\bf t})F({\bf t}; x)\bigr\|_{Y} 
\\& \leq \bigl|f({\bf t}+\tau)-f({\bf t})\bigr| \cdot \| F({\bf t}+\tau; x)\|_{Y}+|f({\bf t}+\tau)| \cdot \bigl\| F({\bf t}+\tau; x)-F({\bf t}; x)\bigr\|_{Y}
\\& \leq 2M\Bigl[ \bigl|f({\bf t}+\tau)-f({\bf t})\bigr| +\bigl\| F({\bf t}+\tau; x)-F({\bf t}; x)\bigr\|_{Y}\Bigr]\leq 2M\epsilon/2M=\epsilon.
\end{align*}
\end{proof}

We can similarly prove the following analogue of Proposition \ref{pointwise-prod-rn} for $({\mathrm R},{\mathcal B})$-multi-almost periodic functions:

\begin{prop}\label{pointwise-prod-rn12345}
Suppose that $\emptyset  \neq I \subseteq {\mathbb R}^{n},$ $f : I \rightarrow {\mathbb C}$ is bounded ${\mathrm R}$-multi-almost periodic and $F : I \times X \rightarrow Y$ is a $({\mathrm R},{\mathcal B})$-multi-almost periodic function whose restriction to any set $I\times B,$ where $B\in {\mathcal B}$ is arbitrary, is bounded. 
Define $F_{1}({\bf t};x):=f({\bf t})F({\bf t}; x),$ ${\bf t}\in I,$ $x\in X.$ Then $F_{1}(\cdot ;\cdot)$ is $({\mathrm R},{\mathcal B})$-multi-almost periodic, provided that for each sequence $({\bf b}_{k})$ in ${\mathrm R}$ we have that any its subsequence also belongs to ${\mathrm R}.$
\end{prop}

Using the decomposition
\begin{align*}
\bigr\| F({\bf t'};x)&-  F({\bf t''};y)\bigl\|_{Y}\leq\bigr\| F({\bf t'};x)-  F({\bf t'}+\tau;x)\bigl\|_{Y} +\bigr\| F({\bf t'}+\tau;x)-  F({\bf t''}+\tau;y)\bigl\|_{Y} \\&+\bigr\| F({\bf t''}+\tau;y)-  F({\bf t''};y)\bigl\|_{Y},\quad {\bf t'},\ {\bf t''} \in I,\ x,\ y\in X,
\end{align*}
and the assumptions clarified below, we can repeat almost literally the argumentation contained in the proof of \cite[Theorem 5, p. 2]{besik} in order to see that the following result holds (unfortunately, the situation in which $I=[0,\infty)^{n}$ is not covered by this result in contrast with the usually considered case $I={\mathbb R}^{n}$):

\begin{prop}\label{nijenaivno}
Suppose that $\emptyset  \neq I \subseteq {\mathbb R}^{n},$ $I +I \subseteq I,$ $I$ is closed and $F : I \times X \rightarrow Y$ is Bohr ${\mathcal B}$-almost periodic, where ${\mathcal B}$ is a family consisting of some compact subsets of $X.$ If the following condition holds
\begin{align}
\notag (\exists {\bf t_{0}}\in I)\, (\forall \epsilon>0)
(\forall l>0) \, (\exists l'>0) &\, (\forall {\bf t'},\ {\bf t''} \in I) \\\notag & B({\bf t_{0}},l) \cap I \subseteq B({\bf t_{0}-t'},l') \cap  B({\bf t_{0}-t''},l') ,
\end{align}
then for each $B\in {\mathcal B}$ the function $F(\cdot;\cdot)$ is uniformly continuous on $I\times B.$ 
\end{prop}

The following is a multi-dimensional extension of \cite[Lemma 1.3(f)]{biologycar011}:

\begin{example}\label{crnicaj}
Suppose that $f : {\mathbb R}^{n} \rightarrow X$ and $g : {\mathbb R}^{n} \rightarrow {\mathbb R}^{n}$ are Bohr almost periodic functions. Define the function
$$
F({\bf t}):=f({\bf t}-g({\bf t})),\quad {\bf t}\in {\mathbb R}^{n}.
$$
Then the function $F(\cdot)$ is  Bohr almost periodic, as well. We can show this similarly as in the proof of the above-mentioned lemma, with appealing to Proposition \ref{dekartovproizvod} and Proposition \ref{nijenaivno}.
\end{example}

\subsection{Strongly ${\mathcal B}$-almost periodic functions}\label{lakolakop}

This subsection analyzes the class of strongly ${\mathcal B}$-almost periodic functions, 
introduced  as follows:

\begin{defn}\label{strong-app} 
Suppose that $\emptyset  \neq I \subseteq {\mathbb R}^{n}$ and $F : I \times X \rightarrow Y$ is a continuous function.
Then we say that $F(\cdot;\cdot)$ is strongly ${\mathcal B}$-almost periodic if and only if for each $B\in {\mathcal B}$ there exists a sequence $(P_{k}^{B}({\bf t};x))$ of trigonometric polynomials 
such that $\lim_{k\rightarrow +\infty}P_{k}^{B}({\bf t};x)=F({\bf t};x),$ uniformly for ${\bf t}\in I$ and $x\in B.$ Here, by a trigonometric polynomial $P : I\times X \rightarrow Y$ we mean any linear combination of functions like
\begin{align}\label{otkriose-gotovje}
e^{i[\lambda_{1}t_{1}+\lambda_{2}t_{2}+\cdot \cdot \cdot +\lambda_{n}t_{n}]}c(x),
\end{align}
where $\lambda_{i}$ are real numbers ($1\leq i \leq n$) and $c: X \rightarrow Y$ is a continuous mapping.
\end{defn}\index{function!strongly ${\mathcal B}$-almost periodic}

The following proposition clarifies a fundamental relationship between strongly ${\mathcal B}$-almost periodic functions and Bohr ${\mathcal B}$-almost periodic functions (Bohr $({\mathcal B},I')$-almost periodic functions):

\begin{prop}\label{postava}
Suppose that $\emptyset  \neq I \subseteq {\mathbb R}^{n}$ and $F :  I \times X \rightarrow Y$ is a  strongly ${\mathcal B}$-almost periodic function, where ${\mathcal B}$ is any collection of bounded subsets of $X.$ Then we have the following:
\begin{itemize}
\item[(i)] for every $j\in {\mathbb N}_{n}$ and $\epsilon>0,$ there exists a finite real number $l>0$ such that every interval $S\subseteq {\mathbb R}$ of length 
$l$ contains a point $\tau_{j} \in I$ such that 
\begin{align}\label{pozivx}
\Bigl\| F\bigl( t_{1},t_{2},\cdot \cdot \cdot, t_{j}+\tau_{j},\cdot \cdot \cdot,t_{n} ;x\bigr) -F\bigl( t_{1},t_{2},\cdot \cdot \cdot, t_{j},\cdot \cdot \cdot, t_{n} ;x \bigr)\Bigr\|\leq \epsilon,
\end{align} 
provided $ (t_{1},t_{2},\cdot \cdot \cdot, t_{j}+\tau_{j},\cdot \cdot \cdot,t_{n})\in I ,\ ( t_{1},\cdot \cdot \cdot, t_{n} )\in I$ and $x\in B;$
\item[(ii)] for every $\epsilon>0,$ there exists a finite real number $l>0$ such that, for every $j\in {\mathbb N}_{n}$ and every interval $S\subseteq {\mathbb R}$ of length 
$l,$ there exists a point $\tau_{j} \in I$ such that \eqref{pozivx} holds provided that $ (t_{1},t_{2},\cdot \cdot \cdot, t_{j}+\tau_{j},\cdot \cdot \cdot,t_{n}) \in I$ for all $j\in {\mathbb N}_{n},$ $( t_{1},\cdot \cdot \cdot, t_{n} )\in I$ and $x\in B;$
\item[(iii)] for every $\epsilon>0,$ there exists a finite real number $l>0$ such that every interval $S\subseteq {\mathbb R}$ of length 
$l$ contains a point $\tau \in I$ such that, for every $j\in {\mathbb N}_{n},$ \eqref{poziv} holds with the number $\tau_{j}$ replaced by the number $\tau$ therein;
\item[(iv)] $F(\cdot;\cdot)$ is Bohr ${\mathcal B}$-almost periodic provided that $I+I\subseteq I$ and that, for every points $( t_{1},\cdot \cdot \cdot, t_{n} )\in I$ and $( \tau_{1},\cdot \cdot \cdot, \tau_{n} )\in I,$ 
the points $( t_{1},t_{2}+\tau_{2},\cdot \cdot \cdot, t_{n}+\tau_{n} ),$ $( t_{1},t_{2},t_{3}+\tau_{3},\cdot \cdot \cdot, t_{n}+\tau_{n} ),\cdot \cdot \cdot ,$ $( t_{1},t_{2},\cdot \cdot \cdot, t_{n-1}, t_{n}+\tau_{n} ),$ also belong to $I;$
\item[(v)] $F(\cdot;\cdot)$ is Bohr $({\mathcal B},I \cap \Delta_{n})$-almost periodic provided that $I \cap \Delta_{n}\neq \emptyset,$ $I+(I\cap \Delta_{n})\subseteq I$ and that, for every points $( t_{1},\cdot \cdot \cdot, t_{n} )\in I$ and $( \tau,\cdot \cdot \cdot, \tau )\in I\cap \Delta_{n},$ 
the points $( t_{1},t_{2}+\tau,\cdot \cdot \cdot, t_{n}+\tau ),$ $( t_{1},t_{2},t_{3}+\tau,\cdot \cdot \cdot, t_{n}+\tau ),\cdot \cdot \cdot ,$ $( t_{1},t_{2},\cdot \cdot \cdot, t_{n-1}, t_{n}+\tau ),$ also belong to $I\cap \Delta_{n}.$
\end{itemize}
\end{prop}

\begin{proof}
The proof is  not difficult and we will present the most relevant details, only.
For the proof of (iii), we can verify first the validity of this statement for any trigonometric polynomial $P(\cdot;\cdot)$ by using the fact
that for each set $B\in {\mathcal B},$ which is bounded due to our assumption, we have that the set $c(B)$ is bounded in $Y$ for any addend of $P(\cdot;\cdot)$ of form \eqref{otkriose-gotovje}
as well as the fact
that any finite set of almost periodic functions of one real variable has a common set of joint $\epsilon$-periods for each $\epsilon>0.$ 
In general case, there exists a sequence $(P_{k}^{B}({\bf t};x))$ of trigonometric polynomials 
such that $\lim_{k\rightarrow +\infty}P_{k}^{B}({\bf t};x)=F({\bf t};x),$ uniformly for ${\bf t}\in I$ and $x\in B.$ Then, for a real number $\epsilon>0$ given in advance, we can find an integer $k_{0}\in {\mathbb N}$ such that 
$\|P_{k_{0}}^{B}({\bf t};x)-F({\bf t};x)\|_{Y}\leq \epsilon/3,$ for every ${\bf t}\in I,$ $x\in B$ and $k\in {\mathbb N}$ with $k\geq k_{0}.$
Using the well known estimate (${\bf t}',\ {\bf t}''\in I;$ $x\in B$):
\begin{align*}
&\bigl\|F({\bf t}';x)-F({\bf t}'';x)\bigr\|_{Y}
\\& \leq \bigl\|F({\bf t}';x)-P_{k_{0}}^{B}({\bf t}';x)\bigr\|_{Y}+\bigl\|P_{k_{0}}^{B}({\bf t}';x)-P_{k_{0}}^{B}({\bf t}'';x)\bigr\|_{Y}+\bigl\|P_{k_{0}}^{B}({\bf t}'';x)-F({\bf t}'';x)\bigr\|_{Y},
\end{align*}
the required statement readily follows; 
the proofs of (i) and (ii) are analogous. For the proof of (iv), we can use (i) and the estimates{\small
\begin{align*}
&  \Bigl\| F\bigl( t_{1}+\tau_{1},t_{2}+\tau_{2},\cdot \cdot \cdot, t_{j}+\tau_{j},\cdot \cdot \cdot,t_{n}+\tau_{n} \bigr) -F\bigl( t_{1},t_{2},\cdot \cdot \cdot, t_{n} \bigr)\Bigr\|
\\& \leq 
\Bigl\| F\bigl( t_{1}+\tau_{1},t_{2}+\tau_{2},\cdot \cdot \cdot, t_{j}+\tau_{j},\cdot \cdot \cdot,t_{n}+\tau_{n} \bigr)  -F\bigl( t_{1},t_{2}+\tau_{2},\cdot \cdot \cdot, t_{j}+\tau_{j},\cdot \cdot \cdot,t_{n}+\tau_{n}\bigr)\Bigr\|
\\& +\Bigl\| F\bigl( t_{1},t_{2}+\tau_{2},\cdot \cdot \cdot, t_{j}+\tau_{j},\cdot \cdot \cdot,t_{n}+\tau_{n}\bigr)-F\bigl( t_{1},t_{2},\cdot \cdot \cdot, t_{n} \bigr)\Bigr\|
\\& \leq 
\Bigl\| F\bigl( t_{1}+\tau_{1},t_{2}+\tau_{2},\cdot \cdot \cdot, t_{j}+\tau_{j},\cdot \cdot \cdot,t_{n}+\tau_{n} \bigr)  -F\bigl( t_{1},t_{2}+\tau_{2},\cdot \cdot \cdot, t_{j}+\tau_{j},\cdot \cdot \cdot,t_{n}+\tau_{n}\bigr)\Bigr\|
\\& +\Bigl\| F\bigl( t_{1},t_{2}+\tau_{2},\cdot \cdot \cdot, t_{j}+\tau_{j},\cdot \cdot \cdot,t_{n}+\tau_{n}\bigr)-F\bigl( t_{1},t_{2},t_{3}+\tau_{3},\cdot \cdot \cdot, t_{n} +\tau_{n}\bigr)\Bigr\|
\\& +\Bigl\| F\bigl( t_{1},t_{2},t_{3}+\tau_{3},\cdot \cdot \cdot, t_{n} +\tau_{n}\bigr)F\bigl( t_{1},t_{2},\cdot \cdot \cdot, t_{n} \bigr)\Bigr\|
 \leq \cdot \cdot \cdot,
\end{align*}}
while for the proof of (v), we can use (iii) and the above estimates with $\tau=\tau_{1}=\tau_{2}=\cdot \cdot \cdot=\tau_{n}.$
\end{proof}

Concerning Proposition \ref{postava}(iv)-(v), it is natural to ask the following:
Suppose that $\emptyset  \neq I \subseteq {\mathbb R}^{n}$ and $F :  I \times X \rightarrow Y$ is a  Bohr ${\mathcal B}$-almost periodic (Bohr $({\mathcal B},I \cap \Delta_{n})$-almost periodic) function. What conditions ensure the strong ${\mathcal B}$-almost periodicity of $F(\cdot;\cdot)?$ After proving Theorem \ref{lenny-jasson}, it will be clear from a combination with Proposition \ref{postava}(iv) that the notion of strong Bohr ${\mathcal B}$-almost periodicity for continuous functions $F : I \rightarrow Y$ coincides with the notion of  Bohr ${\mathcal B}$-almost periodicity, provided that $I$ is a convex polyhedral; as a simple consequence of the last mentioned theorem, we also have that the uniform convergence of a  sequence of scalar-valued trigonometric polynomials on a convex polyhedral in ${\mathbb R}^{n}$ always implies the uniform convergence  
of this sequence on the whole space ${\mathbb R}^{n}$ (in the present state of our knowledge, we really do not know whether this result was known before stating above).
 
The interested reader may try to formulate some statements concerning the relationship between the strong ${\mathcal B}$-almost periodicity and the $({\mathrm R}_{X},{\mathcal B})$-multi-almost periodicity.

\subsection{${\mathbb D}$-asymptotically $({\mathrm R}_{X},{\mathcal B})$-multi-almost periodic type functions}\label{gade-negade}

We start this subsection by introducing the following definition:

\begin{defn}\label{kompleks12345}
Suppose that 
${\mathbb D} \subseteq I \subseteq {\mathbb R}^{n}$ and the set ${\mathbb D}$  is unbounded. By $C_{0,{\mathbb D},{\mathcal B}}(I \times X :Y)$ we denote the vector space consisting of all continuous functions $Q : I \times X \rightarrow Y$ such that, for every $B\in {\mathcal B},$ we have $\lim_{t\in {\mathbb D},|t|\rightarrow +\infty}Q({\bf t};x)=0,$ uniformly for $x\in B.$\index{space!$C_{0,{\mathbb D}}(I \times X :Y)$}
\end{defn}

Now we are ready to introduce the following notion:

\begin{defn}\label{braindamage12345}
Suppose that the set ${\mathbb D} \subseteq I \subseteq {\mathbb R}^{n}$ is unbounded, and
$F : I \times X \rightarrow Y$ is a continuous function. Then we say that $F(\cdot ;\cdot)$ is 
${\mathbb D}$-asymptotically $({\mathrm R},{\mathcal B})$-multi-almost periodic, resp. ${\mathbb D}$-asymptotically $({\mathrm R}_{\mathrm X},{\mathcal B})$-multi-almost periodic,
if and only if there exist an $({\mathrm R},{\mathcal B})$-multi-almost periodic function $G : I \times X \rightarrow Y$, resp. an $({\mathrm R}_{\mathrm X},{\mathcal B})$-multi-almost periodic function $G : I \times X \rightarrow Y$, and a function
$Q\in C_{0,{\mathbb D},{\mathcal B}}(I\times X :Y)$ such that
$F({\bf t} ; x)=G({\bf t} ; x)+Q({\bf t} ; x)$ for all ${\bf t}\in I$ and $x\in X.$

Let $I = {\mathbb R}^{n}.$ Then it is said that $F(\cdot ;\cdot)$ is 
asymptotically $({\mathrm R},{\mathcal B})$-multi-almost periodic, resp. asymptotically $({\mathrm R}_{\mathrm X},{\mathcal B})$-multi-almost periodic, if and only if $F(\cdot ;\cdot)$ is 
${\mathbb R}^{n}$-asymptotically $({\mathrm R},{\mathcal B})$-multi-almost periodic, resp. ${\mathbb R}^{n}$-asymptotically $({\mathrm R}_{\mathrm X},{\mathcal B})$-multi-almost periodic.
\index{function!${\mathbb D}$-asymptotically $({\mathrm R},{\mathcal B})$-multi-almost periodic}
\index{function!${\mathbb D}$-asymptotically $({\mathrm R}_{\mathrm X},{\mathcal B})$-multi-almost periodic}
\end{defn}
\index{function!${\mathbb D}$-asymptotically Bohr ${\mathcal B}$-almost periodic} \index{function!${\mathbb D}$-asymptotically uniformly recurrent}
We similarly introduce the notions of (${\mathbb D}$-)asymptotical Bohr ${\mathcal B}$-almost periodicity, (${\mathbb D}$-)asymptotical uniform recurrence, (${\mathbb D}$-)asymptotical Bohr $({\mathcal B},I')$-almost periodicity and
(${\mathbb D}$-)-asymptotical $({\mathcal B},I')$-uniform recurrence. If $X\in {\mathcal B},$ then we omit the term ${\mathcal B}$
from the notation introduced, with the meaning clear.

Assuming that $F(\cdot;\cdot)$ is  $I$-asymptotically uniformly recurrent, $G : I \times X \rightarrow Y$, 
$Q\in C_{0,I,{\mathcal B}}(I\times X :Y)$ and
$F({\bf t} ; x)=G({\bf t} ; x)+Q({\bf t} ; x)$ for all ${\bf t}\in I$ and $x\in X,$ we can simply show that, for every $x\in X$, we have
\begin{align}\label{josnestoizv}
\overline{\bigl\{G({\bf t};x) : {\bf t}\in I,\ x\in X \bigr\}} \subseteq  \overline{\bigl\{F({\bf t};x) : {\bf t}\in I,\ x\in X\bigr\}}.
\end{align}

The first part of following proposition can be deduced as in the one-dimensional case; keeping in mind the inclusion \eqref{josnestoizv} and the argumentation used in the proof of \cite[Theorem 4.29]{diagana}, we can simply deduce the second part (see also Proposition \ref{2.1.10ap}, Corollary \ref{2.1.10ap1} and Proposition \ref{2.1.10obazacap} for the corresponding results regarding the classes of $({\mathrm R},{\mathcal B})$-multi-almost periodic functions and $({\mathrm R}_{\mathrm X},{\mathcal B})$-multi-almost periodic functions):

\begin{prop}\label{mismodranini}
\begin{itemize}
\item[(i)]
Suppose that for each integer $j\in {\mathbb N}$ the function $F_{j}(\cdot ; \cdot)$ is  Bohr ${\mathcal B}$-almost periodic (${\mathcal B}$-uniformly recurrent). If for each $B\in {\mathcal B}$ there exists $\epsilon_{B}>0$ such that
the sequence $(F_{j}(\cdot ;\cdot))$ converges uniformly to a function $F(\cdot ;\cdot)$ on the set $B^{\circ} \cup \bigcup_{x\in \partial B}B(x,\epsilon_{B}),$ then the function $F(\cdot ;\cdot)$ is Bohr ${\mathcal B}$-almost periodic (${\mathcal B}$-uniformly recurrent).
\item[(ii)] Suppose that for each integer $j\in {\mathbb N}$ the function $F_{j}(\cdot ; \cdot)$ is $I$-asymptotically  Bohr ${\mathcal B}$-almost periodic ($I$-asymptotically  ${\mathcal B}$-uniformly recurrent). If for each $B\in {\mathcal B}$ there exists $\epsilon_{B}>0$ such that
the sequence $(F_{j}(\cdot ;\cdot))$ converges uniformly to a function $F(\cdot ;\cdot)$ on the set $B^{\circ} \cup \bigcup_{x\in \partial B}B(x,\epsilon_{B}),$ then the function $F(\cdot ;\cdot)$ is $I$-asymptotically Bohr ${\mathcal B}$-almost periodic ($I$-asymptotically ${\mathcal B}$-uniformly recurrent).
\end{itemize}
\end{prop}

The proof of following proposition, which can be also clarified for the classes of ${\mathbb D}$-asymptotically almost periodic type functions introduced in Definition \ref{braindamage12345}, is simple and therefore omitted:

\begin{prop}\label{kontinuitetap}
\begin{itemize}
\item[(i)] Suppose that $c\in {\mathbb C}$ and $F(\cdot ; \cdot)$ is $({\mathrm R},{\mathcal B})$-multi-almost periodic, resp. $({\mathrm R}_{\mathrm X},{\mathcal B})$-multi-almost periodic (Bohr ${\mathcal B}$-almost periodic/${\mathcal B}$-uniformly recurrent).
Then $cF(\cdot ; \cdot)$ is $({\mathrm R},{\mathcal B})$-multi-almost periodic, resp. $({\mathrm R}_{\mathrm X},{\mathcal B})$-multi-almost periodic (Bohr ${\mathcal B}$-almost periodic/${\mathcal B}$-uniformly recurrent). 
\item[(ii)] 
\begin{itemize}
\item[(a)]
Suppose that $\tau\in {\mathbb R}^{n},$ $\tau+I\subseteq I$ and $F(\cdot ; \cdot)$ is $({\mathrm R},{\mathcal B})$-multi-almost periodic, resp. $({\mathrm R}_{\mathrm X},{\mathcal B})$-multi-almost periodic (Bohr ${\mathcal B}$-almost periodic/${\mathcal B}$-uniformly recurrent).
Then $F(\cdot +\tau; \cdot)$ is $({\mathrm R},{\mathcal B})$-multi-almost periodic, resp. $({\mathrm R}_{\mathrm X},{\mathcal B})$-multi-almost periodic (Bohr ${\mathcal B}$-almost periodic/${\mathcal B}$-uniformly recurrent).
\item[(b)]  Suppose that $x_{0}\in X$ and $F(\cdot ; \cdot)$ is $({\mathrm R},{\mathcal B})$-multi-almost periodic, resp. $({\mathrm R}_{\mathrm X},{\mathcal B})$-multi-almost periodic (Bohr ${\mathcal B}$-almost periodic/${\mathcal B}$-uniformly recurrent).
Then $F(\cdot; \cdot+x_{0})$ is $({\mathrm R},{\mathcal B}_{x_{0}})$-multi-almost periodic, resp. $({\mathrm R}_{\mathrm X},{\mathcal B}_{x_{0}})$-multi-almost periodic (Bohr ${\mathcal B}_{x_{0}}$-almost periodic/${\mathcal B}_{x_{0}}$-uni-\\fomly recurrent), where ${\mathcal B}_{x_{0}}\equiv \{-x_{0}+B : B\in {\mathcal B}\}.$
\item[(c)] Suppose that $\tau\in {\mathbb R}^{n},$ $\tau+I\subseteq I$, $x_{0}\in X$ and $F(\cdot ; \cdot)$ is $({\mathrm R},{\mathcal B})$-multi-almost periodic, resp. $({\mathrm R}_{\mathrm X},{\mathcal B})$-multi-almost periodic (Bohr ${\mathcal B}$-almost periodic/${\mathcal B}$-uniformly recurrent).
Then $F(\cdot +\tau; \cdot +x_{0})$ is $({\mathrm R},{\mathcal B}_{x_{0}})$-multi-almost periodic, resp. $({\mathrm R}_{\mathrm X},{\mathcal B}_{x_{0}})$-multi-almost periodic (Bohr ${\mathcal B}_{x_{0}}$-almost periodic/${\mathcal B}_{x_{0}}$-uniformly recurrent).
\end{itemize}
\item[(iii)] \begin{itemize}
\item[(a)]
Suppose that $c\in {\mathbb C} \setminus \{0\},$ $cI\subseteq I$ and $F(\cdot ; \cdot)$ is $({\mathrm R},{\mathcal B})$-multi-almost periodic, resp. $({\mathrm R}_{\mathrm X},{\mathcal B})$-multi-almost periodic (Bohr ${\mathcal B}$-almost periodic/${\mathcal B}$-uniformly recurrent).
Then $F(c \cdot ; \cdot)$ is $({\mathrm R}_{c},{\mathcal B})$-multi-almost periodic, resp. $({\mathrm R}_{\mathrm X,c},{\mathcal B})$-multi-almost periodic (Bohr ${\mathcal B}$-almost periodic/${\mathcal B}$-uni-\\formly recurrent), where ${\mathrm R}_{c}\equiv \{c^{-1}{\bf b} : {\bf b}\in {\mathrm R}\}$ and ${\mathrm R}_{{\mathrm X},c}\equiv \{c^{-1}{\bf b} : {\bf b}\in {\mathrm R}_{\mathrm X}\}.$  
\item[(b)] Suppose that $c_{2}\in {\mathbb C}\setminus \{0\},$ and $F(\cdot ; \cdot)$ is $({\mathrm R},{\mathcal B})$-multi-almost periodic, resp. $({\mathrm R}_{\mathrm X},{\mathcal B})$-multi-almost periodic (Bohr ${\mathcal B}$-almost periodic/${\mathcal B}$-uniformly recurrent).
Then $F(\cdot ; c_{2}\cdot)$ is $({\mathrm R},{\mathcal B}_{c_{2}})$-multi-almost periodic, resp. $({\mathrm R}_{\mathrm X},{\mathcal B}_{c_{2}})$-multi-almost periodic (Bohr ${\mathcal B}_{c_{2}}$-almost periodic/${\mathcal B}_{c_{2}}$-uniformly recurrent), where ${\mathcal B}_{c_{2}}\equiv \{c_{2}^{-1}B : B\in {\mathcal B}\}.$ 
\item[(c)]  Suppose that $c_{1}\in {\mathbb C}\setminus \{0\},$ $c_{2}\in {\mathbb C}\setminus \{0\},$ 
$c_{1}I\subseteq I$ and $F(\cdot ; \cdot)$ is $({\mathrm R},{\mathcal B})$-multi-almost periodic, resp. $({\mathrm R}_{\mathrm X},{\mathcal B})$-multi-almost periodic (Bohr ${\mathcal B}$-almost periodic/${\mathcal B}$-uniformly recurrent).
Then $F(c_{1}\cdot ; c_{2}\cdot)$ is
 $({\mathrm R}_{c_{1}},{\mathcal B}_{c_{2}})$-multi-almost periodic, resp. $({\mathrm R}_{{\mathrm X},c_{1}},{\mathcal B}_{c_{2}})$-multi-almost periodic (Bohr ${\mathcal B}_{c_{2}}$-almost periodic/${\mathcal B}_{c_{2}}$-uniformly recurrent).
\end{itemize}
\item[(iv)] Suppose that $\alpha,\ \beta \in {\mathbb C}$ and, for every sequence which belongs to ${\mathrm R}$ (${\mathrm R}_{\mathrm X}$), we have that any its subsequence belongs to ${\mathrm R}$  (${\mathrm R}_{\mathrm X}$). If $F(\cdot ; \cdot)$ and $G(\cdot ; \cdot)$ are $({\mathrm R},{\mathcal B})$-multi-almost periodic, resp. $({\mathrm R}_{\mathrm X},{\mathcal B})$-multi-almost periodic, then 
$(\alpha F+\beta G)(\cdot ; \cdot)$ is $({\mathrm R},{\mathcal B})$-multi-almost periodic, resp. $({\mathrm R}_{\mathrm X},{\mathcal B})$-multi-almost periodic.
\item[(v)] Suppose that $\alpha,\ \beta \in {\mathbb C}.$ If $F : {\mathbb R}^{n} \times X \rightarrow Y$ and $G : {\mathbb R}^{n} \times X \rightarrow Y$ are Bohr ${\mathcal B}$-almost periodic, then 
$(\alpha F+\beta G)(\cdot ; \cdot)$ is Bohr ${\mathcal B}$-almost periodic.
\end{itemize}
\end{prop}

Due to Proposition \ref{2.1.10ap1} and Proposition \ref{kontinuitetap}(ii),  we may conclude that, in the case that $X=\{0\},$ the limit function $F^{\ast}(\cdot)$ in \eqref{lepolepo} is likewise ${\mathrm R}$-multi-almost periodic. In such a way, we can extend the statements of \cite[Theorem 1]{bochner} and \cite[Lemma 1]{sell} for vector-valued functions; the statement of \cite[Lemma 3]{sell} also holds for vector-valued functions.

Using Proposition \ref{kontinuitetap}(iv) and the supremum formula
clarified in Proposition \ref{netokaureap}, 
we can simply deduce that the decomposition in Definition \ref{braindamage12345} is unique under certain assumptions:

\begin{prop}\label{okjeradeap}
\begin{itemize}
\item[(i)]
Suppose that there exist a function $G_{i}(\cdot ;\cdot)$ which is $({\mathrm R},{\mathcal B})$-multi-almost periodic and a function
$Q_{i}\in C_{0,I,{\mathcal B}}(I\times X :Y)$ such that
$F({\bf t} ; x)=G_{i}({\bf t} ; x)+Q_{i}({\bf t} ; x)$ for all ${\bf t}\in I$ and $x\in X$ ($i=1,2$). 
Suppose that,
for every sequence which belongs to ${\mathrm R},$ any its subsequence belongs to ${\mathrm R}.$
If there exists a sequence $b(\cdot)$ in ${\mathrm R}$ whose any subsequence is unbounded and for which we have ${\bf T}-{\bf b}(l)\in I$ whenever ${\bf T}\in I$ and $l\in {\mathbb N},$  
then $G_{1}\equiv G_{2}$ and $Q_{1}\equiv Q_{2}.$ 
\item[(ii)] Suppose that ${\mathcal B}$ is any collection of compact subsets of $X$, there exist a Bohr ${\mathcal B}$-almost periodic function $G_{i}: {\mathbb R}^{n} \times X \rightarrow Y$ and a function
$Q_{i}\in C_{0,I,{\mathcal B}}(I\times X :Y)$ such that
$F({\bf t} ; x)=G_{i}({\bf t} ; x)+Q_{i}({\bf t} ; x)$ for all ${\bf t}\in I$ and $x\in X$ ($i=1,2$). 
Then $G_{1}\equiv G_{2}$ and $Q_{1}\equiv Q_{2}.$ 
\end{itemize}
\end{prop}

For the sequel, we need the following auxiliary lemma (see also \cite[Lemma 2.12]{genralized-multiaa}): 

\begin{lem}\label{obazac}
Suppose that there exist an $({\mathrm R},{\mathcal B})$-multi-almost periodic function $G(\cdot ;\cdot)$ and a function
$Q\in C_{0,I,{\mathcal B}}(I\times X :Y)$ such that $F({\bf t} ; x)=G({\bf t} ; x)+Q({\bf t} ; x)$ for all ${\bf t}\in I$ and $x\in X.$
Then \eqref{josnestoizv} holds provided that
that for each sequence ${\bf b}\in {\mathrm R}$ we have $I\pm {\bf b}(l)\in I,$ $l\in {\mathbb N}$ and there exists a sequence in ${\mathrm R}$ whose any subsequence is unbounded.
\end{lem}

Now we are in a position to clarify the following result:

\begin{prop}\label{2.1.10obazacap}
Suppose that,
for every sequence ${\bf b}(\cdot)$ which belongs to ${\mathrm R},$ any its subsequence belongs to ${\mathrm R}$ and ${\bf T}-{\bf b}(l)\in I$ whenever ${\bf T}\in I$ and $l\in {\mathbb N}.$
Suppose, further, that there exists a sequence in ${\mathrm R}$ whose any subsequence is unbounded. If for each integer $j\in {\mathbb N}$ the function $F_{j}(\cdot ; \cdot)$ is $I$-asymptotically $({\mathrm R},{\mathcal B})$-multi-almost periodic and for each $B\in {\mathcal B}$ there exists $\epsilon_{B}>0$ such that
the sequence $(F_{j}(\cdot ;\cdot))$ converges uniformly to a function $F(\cdot ;\cdot)$ on the set $B^{\circ} \cup \bigcup_{x\in \partial B}B(x,\epsilon_{B}),$ then the function $F(\cdot ;\cdot)$ is $I$-asymptotically $({\mathrm R},{\mathcal B})$-multi-almost periodic.
\end{prop}

\begin{proof}
Due to Proposition \ref{okjeradeap}, 
we know that there exist a uniquely determined function $G(\cdot ;\cdot)$ which is $({\mathrm R},{\mathcal B})$-multi-almost periodic and a uniquely determined function
$Q\in C_{0,I,{\mathcal B}}(I \times X :Y)$ such that
$F({\bf t} ; x)=G({\bf t} ; x)+Q({\bf t} ; x)$ for all ${\bf t}\in I$ and $x\in X$. Furthermore, we have
$$
F_{j}({\bf t} ; x)-F_{m}({\bf t} ; x)=\bigl[ G_{j}({\bf t} ; x)-G_{m}({\bf t} ; x)\bigr]+\bigl[Q_{j}({\bf t} ; x)-Q_{m}({\bf t} ; x) \bigr],
$$
for all ${\bf t}\in I,$ $x\in X$ and $j,\ m\in {\mathbb N}.$ Due to Proposition \ref{kontinuitetap}(iv), we have that 
the function $G_{j}(\cdot ;\cdot)-G_{m}(\cdot ;\cdot)$ is $({\mathrm R},{\mathcal B})$-multi-almost periodic ($j,\ m\in {\mathbb N}$). Keeping in mind this fact as well as Lemma \ref{obazac} and the argumentation used in the proof of \cite[Theorem 4.29]{diagana}, we get that
\begin{align*}
3\sup_{{\bf t} \in I,x\in X}&\Bigl\|  F_{j}({\bf t} ; x)-F_{m}({\bf t} ; x) \Bigr\|_{Y}
\\& \geq \sup_{{\bf t} \in I,x\in X}\Bigl\|  G_{j}({\bf t} ; x)-G_{m}({\bf t} ; x) \Bigr\|_{Y}+\sup_{{\bf t} \in I,x\in X}\Bigl\|  Q_{j}({\bf t} ; x)-Q_{m}({\bf t} ; x) \Bigr\|_{Y},
\end{align*}
for any $j,\ m\in {\mathbb N}.$
This implies that the sequences $(G_{j}(\cdot;\cdot))$ and $(Q_{j}(\cdot;\cdot))$ converge uniformly to the functions $G(\cdot;\cdot)$ and $Q(\cdot;\cdot),$ respectively. Due to Proposition \ref{2.1.10ap1}, we get that the function $G(\cdot;\cdot)$ is $({\mathrm R},{\mathcal B})$-multi-almost periodic.
The final conclusion follows from the obvious equality $F=G+Q$ and the fact that $C_{0,I,{\mathcal B}}(I \times X : Y)$ is a Banach space.
\end{proof}

Before we move ourselves to the next subsection, we would like to recall that the notion of asymptotical almost periodicity in a Bohr like manner, for scalar-valued functions of one real variable,
was introduced by A. S. Kovanko \cite{286} in 1929; the usually employed definition of asymptotical almost periodicity was introduced later, by M. Fr\'echet \cite{193} in 1941.
Now we will introduce the following general definition following the approach obeyed in \cite{193}; for any set $\Lambda \subseteq {\mathbb R}^{n},$ we define $\Lambda_{M}:=\{ \lambda \in \Lambda \, ; \, |\lambda|\geq M \}:$

\begin{defn}\label{nafaks123456789012345123}
Suppose that 
${\mathbb D} \subseteq I \subseteq {\mathbb R}^{n}$ and the set ${\mathbb D}$ is unbounded, as well as
$\emptyset  \neq I'\subseteq I \subseteq {\mathbb R}^{n},$ $F : I \times X \rightarrow Y$ is a continuous function and $I +I' \subseteq I.$ Then we say that:
\begin{itemize}
\item[(i)]\index{function!${\mathbb D}$-asymptotically Bohr $({\mathcal B},I')$-almost periodic of type $1$}
$F(\cdot;\cdot)$ is ${\mathbb D}$-asymptotically Bohr $({\mathcal B},I')$-almost periodic  of type $1$ if and only if for every $B\in {\mathcal B}$ and $\epsilon>0$
there exist $l>0$ and $M>0$ such that for each ${\bf t}_{0} \in I'$ there exists ${\bf \tau} \in B({\bf t}_{0},l) \cap I'$ such that
\begin{align}\label{emojmarko145}
\bigl\|F({\bf t}+{\bf \tau};x)-F({\bf t};x)\bigr\|_{Y} \leq \epsilon,\mbox{ provided } {\bf t},\ {\bf t}+\tau \in {\mathbb D}_{M},\ x\in B.
\end{align}
\item[(ii)] \index{function!${\mathbb D}$-asymptotically $({\mathcal B},I')$-uniformly recurrent  of type $1$}
$F(\cdot;\cdot)$ is ${\mathbb D}$-asymptotically $({\mathcal B},I')$-uniformly recurrent  of type $1$ if and only if for every $B\in {\mathcal B}$ 
there exist a sequence $({\bf \tau}_{n})$ in $I'$ and a sequence $(M_{n})$ in $(0,\infty)$ such that $\lim_{n\rightarrow +\infty} |{\bf \tau}_{n}|=\lim_{n\rightarrow +\infty}M_{n}=+\infty$ and
$$
\lim_{n\rightarrow +\infty}\sup_{{\bf t},{\bf t}+{\bf \tau}_{n}\in {\mathbb D}_{M_{n}};x\in B} \bigl\|F({\bf t}+{\bf \tau}_{n};x)-F({\bf t};x)\bigr\|_{Y} =0.
$$
\end{itemize}
If $I'=I,$ then we also say that
$F(\cdot;\cdot)$ is ${\mathbb D}$-asymptotically Bohr ${\mathcal B}$-almost periodic of type $1$ (${\mathbb D}$-asymptotically ${\mathcal B}$-uniformly recurrent  of type $1$); furthermore, if $X\in {\mathcal B},$ then it is also said that $F(\cdot;\cdot)$ is ${\mathbb D}$-asymptotically Bohr $I'$-almost periodic  of type $1$ (${\mathbb D}$-asymptotically $I'$-uniformly recurrent  of type $1$). If $I'=I$ and $X\in {\mathcal B}$, then we also say that $F(\cdot;\cdot)$ is ${\mathbb D}$-asymptotically Bohr almost periodic  of type $1$ (${\mathbb D}$-asymptotically uniformly recurrent of type $1$). As before, we remove the prefix ``${\mathbb D}$-'' in the case that ${\mathbb D}=I$ and remove the prefix ``$({\mathcal B},)$''  in the case that $X\in {\mathcal B}.$ 
\end{defn}

The proof of following proposition is trivial and therefore omitted:

\begin{prop}\label{okeje}
Suppose that 
${\mathbb D} \subseteq I \subseteq {\mathbb R}^{n}$ and the set ${\mathbb D}$ is unbounded, as well as
$\emptyset  \neq I'\subseteq I \subseteq {\mathbb R}^{n},$ $F : I \times X \rightarrow Y$ is a continuous function and $I +I' \subseteq I.$ If  
$F(\cdot;\cdot)$ is ${\mathbb D}$-asymptotically Bohr $({\mathcal B},I')$-almost periodic, resp. ${\mathbb D}$-asymptotically $({\mathcal B},I')$-uniformly recurrent, then $F(\cdot;\cdot)$ is ${\mathbb D}$-asymptotically Bohr $({\mathcal B},I')$-almost periodic of type $1,$
resp. ${\mathbb D}$-asymptotically $({\mathcal B},I')$-uniformly recurrent of type $1$.
\end{prop}

Suppose now that the general assumptions from the preamble of Definition \ref{nafaks123456789012345123} hold true. Keeping in mind Proposition \ref{okeje} and Remark \ref{soqqaza}(i)-(ii), it is natural to ask the following:
\begin{itemize}
\item[(i)]
In which cases the ${\mathbb D}$-asymptotical Bohr $({\mathcal B},I')$-almost periodicity of type $1,$
resp. ${\mathbb D}$-asymptotical $({\mathcal B},I')$-uniform recurrence of type $1,$ implies the ${\mathbb D}$-asymptotical Bohr $({\mathcal B},I')$-almost periodicity,
resp. ${\mathbb D}$-asymptotical $({\mathcal B},I')$-uniform recurrence of function $F(\cdot;\cdot)$? 
\item[(ii)] In which cases the asymptotical Bohr ${\mathcal B}$-almost periodicity (of type $1$) implies the $({\mathrm R},{\mathcal B})$-multi-almost periodicity of $F(\cdot;\cdot),$ where ${\mathrm R}$ denotes the collection of all sequences in 
$I?$
\item[(iii)] In which cases the asymptotical Bohr ${\mathcal B}$-almost periodicity (of type $1$) is a consequence of the $({\mathrm R},{\mathcal B})$-multi-almost periodicity of $F(\cdot;\cdot),$ where ${\mathrm R}$ denotes the collection of all sequences in 
$I?$
\end{itemize}

Concerning the item (ii), it is well known that the answer is negative provided that $X=\{0\},$ ${\mathcal B}=X$ and $I={\mathbb R}$ because, in this case, the asymptotical Bohr ${\mathcal B}$-almost periodicity of $F : {\mathbb R} \rightarrow Y$ is equivalent with the asymptotical Bohr ${\mathcal B}$-almost periodicity of type $1$ of $F(\cdot)$, i.e., the usual asymptotical almost periodicity of $F(\cdot)$, while the $({\mathrm R},{\mathcal B})$-multi-almost periodicity of $F(\cdot)$ is equivalent in this case with the usual almost periodicity of $F(\cdot);$ see \cite[Definition 2.2, Definition 2.3; Theorem 2.6]{zhang} for the notion used. With the regards to the items (i) and (ii), we have the following statement which can be applied in the particular case $I=[0,\infty)^{n}:$

\begin{thm}\label{bounded-paziem}
Suppose that $\emptyset  \neq I \subseteq {\mathbb R}^{n},$ $I +I \subseteq I,$ $I$ is closed and
$F : I \times X \rightarrow Y$ is a uniformly continuous,
$I$-asymptotically Bohr ${\mathcal B}$-almost periodic function of type $1,$ where ${\mathcal B}$ is any family of compact subsets of $X.$ If
\begin{align}
\notag
(\forall l>0) \, (\forall M>0) \, (\exists {\bf t_{0}}\in I)\, (\exists k>0) &\, (\forall {\bf t} \in I_{M+l})(\exists {\bf t_{0}'}\in I)\,
\\\label{profice-zeronem} & (\forall {\bf t_{0}''}\in B({\bf t_{0}'},l) \cap I)\, {\bf t}- {\bf t_{0}''} \in B({\bf t_{0}},kl) \cap I_{M},
\end{align}
there exists $L>0$ such that $I_{kL}\setminus I_{(k+1)L}  \neq \emptyset$ for all $k\in {\mathbb N}$ and $I_{M}+I\subseteq  I_{M}$ for all $M>0,$
then the function $F(\cdot; \cdot)$ is $({\mathrm R},{\mathcal B})$-multi-almost periodic, where ${\mathrm R}$ denotes the collection of all sequences in 
$I.$ Furthermore, if $X=\{0\}$ and ${\mathcal B}=\{X\},$ then $F(\cdot)$ is $I$-asymptotically Bohr almost periodic function.
\end{thm}

\begin{proof}
Let $B\in {\mathcal B}$ and $\epsilon>0$ be fixed. Since $F(\cdot;\cdot)$ is uniformly continuous, we have that the function ${\bf F}_{B}(\cdot),$ given by \eqref{sarajevo-london}, is likewise uniformly continuous. 
Arguing as in the proof of Proposition \ref{bounded-pazi}, the assumption \eqref{profice-zeronem} enables one to deduce that the set $\{F({\bf t};x) : {\bf t} \in I,\ x\in B\}$ is relatively compact in $Y$ as well as that the set $\{F_{B}({\bf t}) : {\bf t} \in  I\}$ is relatively compact in the Banach space $BUC(B:Y),$ consisting of all bounded, uniformly continuous functions from $B$ into $Y,$ equipped with the sup-norm. We know that there exist $l>0$ and $M>0$ such that for each ${\bf t}_{0} \in I$ there exists ${\bf \tau} \in B({\bf t}_{0},l) \cap I$ such that
\eqref{emojmarko145} holds with ${\mathbb D}=I.$
Using these facts, we can slightly modify the first part of the proof of \cite[Theorem 3.3]{RUESS-1} (with the segment $[N,3N]$ replaced therein with the set $I_{N}\setminus I_{3N}$, where $N=\max(L,l,M),$ and the number $\tau_{k}\in [kN,(k+1)N]$ replaced therein by the number ${\bf \tau}_{k}\in I_{kL}\setminus I_{(k+1)L};$ we need condition $I_{M}+I\subseteq  I_{M},$ $M>0$ in order to see that the estimate given on \cite[l. 2, p. 23]{RUESS-1} holds in our framework) in order to obtain that the set of translations $\{ F_{B}(\cdot +\tau) : \tau \in I \}$ is relatively compact in $BUC(B:Y),$
which simply implies that 
$F(\cdot;\cdot)$ is $({\mathrm R},{\mathcal B})$-multi-almost periodic, where ${\mathrm R}$ denotes the collection of all sequences in 
$I.$ Suppose now that $X=\{0\}$ and ${\mathcal B}=\{X\}.$ Then 
for each integer $k\in {\mathbb N}$  
there exist $l_{k}>0$ and $M_{k}>0$ such that for each ${\bf t}_{0} \in I$ there exists ${\bf \tau} \in B({\bf t}_{0},l) \cap I$ such that
\eqref{emojmarko145} holds with $\epsilon=1/k$ and ${\mathbb D}=I.$
Let ${\bf \tau}_{k}$ be any fixed element of $I$ such that $|{\bf \tau}_{k}|>M_{k}+k^{2}$ and \eqref{emojmarko145} holds with $\epsilon=1/k$ and ${\mathbb D}=I$ ($k\in {\mathbb N}$).
Then the first part of proof yields the existence of a subsequence $({\bf \tau}_{k_{l}})$ of $({\bf \tau}_{k})$ and a function
$F^{\ast} : I \rightarrow Y$
such that $\lim_{l\rightarrow +\infty}F({\bf t}+{\bf \tau}_{k_{l}})=F^{\ast}({\bf t}),$ uniformly for $t\in I.$ The mapping $F^{\ast}(\cdot)$ is clearly continuous and now we will prove that $F^{\ast}(\cdot)$ is Bohr almost periodic. 
Let $\epsilon>0$ be fixed, and let $l>0$ and $M>0$ be such that for each ${\bf t}_{0} \in I$ there exists ${\bf \tau} \in B({\bf t}_{0},l) \cap I$ such that
\eqref{emojmarko145} holds with ${\mathbb D}=I$ and the number $\epsilon$ replaced therein by $\epsilon/3.$ Let ${\bf t }\in I$ be fixed, and let $l_{0}\in {\mathbb N}$ be such that $|{\bf t}+{\bf \tau}_{k_{l_{0}}}|\geq M$ and $|{\bf t}+{\bf \tau}+{\bf \tau}_{k_{l_{0}}}|\geq M.$ Then we have 
\begin{align*}
\Bigl\| & F^{\ast}({\bf t}+{\bf \tau})- F^{\ast}({\bf t})\Bigr\|
\\& \leq \Bigl\|  F^{\ast}({\bf t}+{\bf \tau})- F^{\ast}\bigl({\bf t}+{\bf \tau}+{\bf \tau}_{k_{l_{0}}}\bigr)\Bigr\|+\Bigl\|  F^{\ast}\bigl({\bf t}+{\bf \tau}+{\bf \tau}_{k_{l_{0}}}\bigr)- F^{\ast}\bigl({\bf t}+{\bf \tau}_{k_{l_{0}}}\bigr)\Bigr\|
\\&+
\Bigl\|  F^{\ast}\bigl({\bf t}+{\bf \tau}_{k_{l_{0}}}\bigr)-F^{\ast}({\bf t}) \Bigr\|\leq 3\cdot (\epsilon/3)=\epsilon,
\end{align*}
as required. The fact that the function ${\bf t} \mapsto F({\bf t} )-F^{\ast}({\bf t} ),$ ${\bf t} \in I$ belongs to the space $C_{0,I}(I: Y)$ follows trivially by definition of $F^{\ast}(\cdot).$ The proof of theorem
is thereby complete.
\end{proof}

\begin{rem}\label{remarka}
Suppose that the requirements of Theorem \ref{bounded-paziem} hold with $X=\{0\}$ and ${\mathcal B}=\{X\}.$ Suppose further that, for every ${\bf t}'\in {\mathbb R}^{n},$ there exist $\delta>0$ and $l_{0}\in {\mathbb N}$ such that the sequence
$(\tau_{k})$ from the above proof satisfies that ${\bf t}''+\tau_{k_{l}} \in I$ for all $l\in {\mathbb N}$ with $l\geq l_{0}$ and ${\bf t}''\in B({\bf t}',\delta).$
Then the limit $\lim_{l\rightarrow +\infty}F({\bf t}'+{\bf \tau}_{k_{l}}):=\tilde{F^{\ast}}({\bf t}')$ exists for all ${\bf t}'\in {\mathbb R}^{n},$ which can be easily seen from the estimate
\begin{align}
\notag\Bigl\| & F\bigl({\bf t}'+{\bf \tau}_{k_{l_{1}}}\bigr)-F\bigl({\bf t}'+{\bf \tau}_{k_{l_{2}}}\bigr) \Bigr\|_{Y}
\\\notag & \leq \Bigl\|  F\bigl({\bf t}'+{\bf \tau}_{k_{l_{1}}}\bigr)-F\bigl({\bf t}'+{\bf \tau}_{k_{l_{1}}}+\tau\bigr) \Bigr\|_{Y}+\Bigl\| F\bigl({\bf t}'+{\bf \tau}_{k_{l_{1}}}+\tau\bigr)-F\bigl({\bf t}'+{\bf \tau}_{k_{l_{2}}}+\tau\bigr) \Bigr\|_{Y}
\\\label{milo-dritan} & + \Bigl\|  F\bigl({\bf t}'+{\bf \tau}_{k_{l_{2}}}+\tau\bigr)-F\bigl({\bf t}'+{\bf \tau}_{k_{l_{2}}}\bigr) \Bigr\|_{Y}
\\\notag & \leq 3\cdot (\epsilon/3)=\epsilon,
\end{align}
which is valid for all numbers $\tau$ such that there exist $l>0$ and $M>0$ such that for each ${\bf t}_{0} \in I$ there exists ${\bf \tau} \in B({\bf t}_{0},l) \cap I$ such that
\eqref{emojmarko145} holds with the number $\epsilon$ replaced therein with the number $\epsilon/3$ and ${\mathbb D}=I,$
all sufficiently large natural numbers $l_{1}$ and $l_{2}$ depending on $\tau,$ where we have also applied the Cauchy criterion of convergence for the limit $\lim_{l\rightarrow +\infty}F({\bf t}+{\bf \tau}_{k_{l}})=F^{\ast}({\bf t}),$ uniform in $t\in I$
and our assumption $I+I\subseteq I.$ The function $\tilde{F^{\ast}}(\cdot)$ is clearly continuous and it can be easily seen that it is Bohr $I$-almost periodic. Furthermore, if for every ${\bf t}'\in {\mathbb R}^{n}$ and $M_{1},\ M_{2}>0$ there exists $l_{0}\in {\mathbb N}$ such that ${\bf t}'+\tau_{k_{l}}-\tau \in I_{M_{2}}$ for all $l\in {\mathbb N}$ with $l\geq l_{0}$, then $\tilde{F^{\ast}}(\cdot)$ is Bohr $(I \cup (-I))$-almost periodic. Using a simple translation argument, the above gives an extension of \cite[Theorem 3.4]{RUESS-1} in Banach spaces.
\end{rem}

Keeping in mind the proof of Theorem \ref{bounded-paziem} and our analysis from Remark \ref{remarka}, we can also deduce the following result concerning the extensions of Bohr $I'$-almost periodic functions: 

\begin{thm}\label{lenny-jasson}
Suppose that $ I'\subseteq I \subseteq {\mathbb R}^{n},$ $I +I' \subseteq I,$ the set $I'$ is unbounded, $F : I  \rightarrow Y$ is a uniformly continuous, Bohr $I'$-almost periodic function, resp. a uniformly continuous, $I'$-uniformly recurrent function, $S\subseteq {\mathbb R}^{n}$ is bounded and the following condition holds:
\begin{itemize}
\item[(AP-E)]\index{condition!(AP-E)}
For every ${\bf t}'\in {\mathbb R}^{n},$
there exists a finite real number $M>0$ such that
${\bf t}'+I'_{M}\subseteq I.$
\end{itemize}
Define $\Omega_{S}:=[(I'\cup (-I'))+(I'\cup (-I'))] \cup S.$
Then there exists a uniformly continuous, Bohr $\Omega_{S}$-almost periodic, resp. a uniformly continuous, $\Omega_{S}$-uniformly recurrent, function
$\tilde{F} : {\mathbb R}^{n}  \rightarrow Y$ such that $\tilde{F}({\bf t})=F({\bf t})$ for all ${\bf t}\in I;$ furthermore, in almost periodic case, the uniqueness of such a function $\tilde{F}(\cdot)$ holds provided that ${\mathbb R}^{n} \setminus \Omega_{S}$ is a bounded set. 
\end{thm}

\begin{proof}
We will prove the theorem only for uniformly continuous, Bohr $I'$-almost periodic functions.
In this case,
for each natural number $k\in {\mathbb N}$  
there exists a point $\tau_{k}\in I'$ such that
$
\|F({\bf t}+{\bf \tau}_{k})-F({\bf t})\|_{Y} \leq 1/k$ for all ${\bf t}\in I
$ and $k\in {\mathbb N};$ furthermore, since the set $I'$ is unbounded, we may assume without loss of generality that $\lim_{k\rightarrow +\infty}|{\bf \tau}_{k}|=+\infty.$ 
Hence, one has
$\lim_{k\rightarrow +\infty}F({\bf t}+{\bf \tau}_{k})=F({\bf t}),$ uniformly for $t\in I.$ If ${\bf t}'\in {\mathbb R}^{n},$ then we can use our assumption on the existence of a finite real number $M>0$ such that
${\bf t}'+I'_{M}\subseteq I,$ and
the corresponding argumentation from Remark \ref{remarka} (see \eqref{milo-dritan}), in order to conclude that 
$\lim_{k\rightarrow +\infty}F({\bf t}'+{\bf \tau}_{k}):=\tilde{F}({\bf t}')$ exists.
The function $\tilde{F}(\cdot)$ is clearly uniformly continuous because $F(\cdot)$ is uniformly continuous; furthermore, by construction, we have that $\tilde{F}({\bf t})=F({\bf t})$ for all ${\bf t}\in I.$ 
Now we will prove that the function $\tilde{F}(\cdot)$
is Bohr $\Omega_{S}$-almost periodic. Suppose that a number $\epsilon>0$ is given. Then we know that there exists $l>0$ such that for each ${\bf t}_{0} \in I'$ there exists ${\bf \tau} \in B({\bf t}_{0},l) \cap I'$ such that
$
\|F({\bf t}+{\bf \tau})-F({\bf t})\|_{Y} \leq \epsilon/2$ for all $ {\bf t}\in I.$ Let ${\bf t}'\in {\mathbb R}^{n}$ be fixed. For any such numbers ${\bf t}_{0} \in I'$ and ${\bf \tau} \in B({\bf t}_{0},l) \cap I',$ we have
\begin{align*}
\bigl\|\tilde{F}({\bf t}'&+{\bf \tau})-\tilde{F}({\bf t}')\bigr\|_{Y} = \Bigl\|\lim_{k\rightarrow +\infty}\bigl[F({\bf t}'+{\bf \tau}+{\bf \tau}_{k})-F({\bf t}'+{\bf \tau}_{k})\bigr]\Bigr\|_{Y} 
\\& \leq \limsup_{k\rightarrow +\infty}\bigl\|F({\bf t}'+{\bf \tau}+{\bf \tau}_{k})-F({\bf t}'+{\bf \tau}_{k})\bigr\|_{Y}  
\leq \epsilon/2,\quad {\bf t}'\in {\mathbb R}^{n}.
\end{align*}
This clearly implies 
\begin{align*}
\bigl\|\tilde{F}({\bf t}'&-{\bf \tau})-\tilde{F}({\bf t}')\bigr\|_{Y} 
\leq \epsilon/2,\quad {\bf t}'\in {\mathbb R}^{n},
\end{align*}
which further implies that $F(\cdot)$ is Bohr $(I' \cup (-I'))$-almost periodic since $-{\bf t}_{0} \in I'$ and $-{\bf \tau} \in B(-{\bf t}_{0},l) \cap (-I').$ Take now any 
number $\tau \in \Omega;$ then $\tau$ can be
written  as a sum of two elements $\tau_{1}$ and $\tau_{2}$ from the set $(I'\cup (-I'))$ and, as a such, it will satisfy
\begin{align*}
&\bigl\|\tilde{F}({\bf t}'+{\bf \tau})-\tilde{F}({\bf t}')\bigr\|_{Y} = \bigl\|F({\bf t}'+{\bf \tau}_{1}+{\bf \tau}_{2})-F({\bf t}')\bigr\|_{Y} 
\\& \leq \bigl\|F({\bf t}'+{\bf \tau}_{1}+{\bf \tau}_{2})-F({\bf t}'+{\bf \tau}_{1})\bigr\|_{Y}  +\bigl\|F({\bf t}'+{\bf \tau}_{1})-F({\bf t}')\bigr\|_{Y} 
\leq 2\cdot ( \epsilon/2)=\epsilon,
\end{align*}
for any ${\bf t}'\in {\mathbb R}^{n}.$ Therefore, $F(\cdot)$ is Bohr $\Omega$-almost periodic, which clearly implies that $F(\cdot)$ is Bohr $\Omega_{S}$-almost periodic, as well.

Assume, finally, that the set ${\mathbb R}^{n}\setminus \Omega_{S}$ is bounded. 
Then the function $\tilde{F}(\cdot)$ is Bohr almost periodic and any
function  
$\tilde{G} : {\mathbb R}^{n}  \rightarrow Y$ which extends the function $F(\cdot)$ to the whole space and satisfies the above requirements must be Bohr almost periodic. 
Therefore, $\tilde{G}(\cdot)$ is 
compactly
almost automorphic so that the sequence 
$(\tau_{k})$ has a subsequence $(\tau_{k_{l}})$
such that
$$
\lim_{l\rightarrow +\infty}\lim_{m\rightarrow +\infty}\tilde{G}\bigl( {\bf t}'+\tau_{k_{m}}-\tau_{k_{l}}\bigr)=\tilde{G}\bigl({\bf t}'\bigr),\quad {\bf t}'\in {\mathbb R}^{n}.
$$
But, for every $l\in {\mathbb N}$ and ${\bf t}'\in {\mathbb R}^{n},$ we have that $\lim_{m\rightarrow +\infty}\tilde{G}( {\bf t}'+\tau_{k_{m}}-\tau_{k_{l}})=\lim_{m\rightarrow +\infty}F( {\bf t}'+\tau_{k_{m}}-\tau_{k_{l}}),$ so that the final conclusion follows from the almost automorphy of function
$\tilde{F}(\cdot)$ and the equality
$$
\lim_{l\rightarrow +\infty}\lim_{m\rightarrow +\infty}\tilde{F}\bigl( {\bf t}'+\tau_{k_{m}}-\tau_{k_{l}}\bigr)=\tilde{F}\bigl({\bf t}'\bigr),
$$
which holds pointwise on ${\mathbb R}^{n}.$ 
\end{proof}

In particular, if $({\bf v}_{1},\cdot \cdot \cdot ,{\bf v}_{n})$ is a basis of ${\mathbb R}^{n},$  and
$$
I'=I=\bigl\{ \alpha_{1} {\bf v}_{1} +\cdot \cdot \cdot +\alpha_{n}{\bf v}_{n}  : \alpha_{i} \geq 0\mbox{ for all }i\in {\mathbb N}_{n} \bigr\}
$$ 
is a convex polyhedral in ${\mathbb R}^{n},$ then we have $\Omega={\mathbb R}^{n}$ and, in this case, there exists a unique Bohr almost periodic extension of the function
$F : I \rightarrow Y$
to the whole Euclidean space, so that Proposition \ref{dekartovproizvod}, Proposition \ref{pointwise-prod-rn}, Proposition \ref{pointwise-prod-rn12345} and Proposition \ref{kontinuitetap}(v) continue to hold in this framework.\index{convex polyhedral}

We will also state the following important corollary of Theorem \ref{lenny-jasson} (see also \eqref{tupak12345}):

\begin{cor}\label{lenny-jasson1} (The uniqueness theorem for almost periodic functions) \index{the uniqueness theorem for almost periodic functions}
Suppose that $ I \subseteq {\mathbb R}^{n},$ $I +I \subseteq I,$ condition \emph{(AP-E)}
holds with $I'=I,$
and ${\mathbb R}^{n} \setminus [(I\cup (-I))+(I\cup (-I))]$ is a bounded set. 
If
$F :{\mathbb R}^{n}  \rightarrow Y$ and $G :{\mathbb R}^{n}  \rightarrow Y$  are two Bohr almost periodic functions and $F({\bf t})=G({\bf t})$ for all ${\bf t}\in I,$ then
$F({\bf t})=G({\bf t})$ for all ${\bf t}\in {\mathbb R}^{n}.$
\end{cor}

Now we would like to propose the following definition:

\begin{defn}\label{ujseadmis}
Suppose that $\emptyset \neq I \subseteq {\mathbb R}^{n}$ and $I+I \subseteq I.$ Then we say that $I$ is admissible with respect to the almost periodic extensions if and only if for any complex Banach space $Y$ and for any 
uniformly continuous, Bohr almost periodic function $F : I\rightarrow Y$ there exists a unique Bohr almost periodic function $\tilde{F} : {\mathbb R}^{n} \rightarrow Y$ such that $\tilde{F}({\bf t})=F({\bf t})$ for all ${\bf t}\in I.$ 
\end{defn}

By the foregoing, it is clear that any non-empty subset $I$ of ${\mathbb R}^{n}$ which is closed under addition and satisfies that 
condition (AP-E)
holds with $I'=I$ as well as
the set ${\mathbb R}^{n} \setminus [(I\cup (-I))+(I\cup (-I))]$ is bounded (in particular, this holds for convex polyhedrals)
has to be admissible with respect to the almost periodic extensions. But, it is clear that the set $I=[0,\infty) \times \{0\}\subseteq {\mathbb R}^{2}$ is not admissible with respect to the almost periodic extensions since there is no almost periodic extension
of the function $F(x,y)=y,$ $(x,y)\in I$ to the whole plane. Further analyses of the notion introduced in Definition \ref{ujseadmis} is without scope of this paper.

Concerning the item (iii), we will clarify the following result:

\begin{thm}\label{izgubljeni-em}
Suppose that $0\in  I \subseteq {\mathbb R}^{n},$ $I$ is closed, $I+I\subseteq I$ and
$\emptyset  \neq I'\subseteq I.$
Suppose, further, that
the set ${\mathbb D} \subseteq I$ is unbounded and condition \emph{(MD)} holds,
where:
\begin{itemize}
\item[(MD)] For each $M_{0}>0$ 
there exists a finite real number $M_{1}>M_{0}$ such that ${\mathbb D}_{M_{1}}-{\bf t}\in I$ and
$I'_{M_{1}}-{\bf t}\in I'$
for all ${\bf t}\in I \setminus I_{M_{0}}.$
\end{itemize}
Let ${\mathrm R}$ denote the collection of all sequences in $I,$ and let ${\mathcal B}$ denote any family of compact subsets of $X.$
Then any $({\mathrm R},{\mathcal B})$-multi-almost periodic function $F : I \times X \rightarrow Y$ is ${\mathbb D}$-asymptotically Bohr $({\mathcal B},I')$-almost periodic of type $1.$
\end{thm}

\begin{proof}
Let $B\in {\mathcal B}$ and $\epsilon>0$ be fixed. Since $I$ is closed, we have that the restriction of function
$F(\cdot;\cdot)$ to the set $I\times B$ is uniformly continuous, which easily implies that the function $F_{B} : I \rightarrow BUC(B:Y),$  given by \eqref{sarajevo-london}, is well defined and uniformly continuous.
Now we will prove that the function $F_{B}(\cdot)$ has a relatively compact range.
Denote $K_{k}=[-k,k]^{n}$ for all integers $k\in {\mathbb N}.$ Since the set $F_{B}(K_{k} \cap I)$ is relatively compact in $BUC(B:Y)$ for all integers $k\in {\mathbb N},$
it suffices to show that 
there exists $k\in {\mathbb N}$ such that, for every ${\bf t}\in I$, there exists a point ${\bf s}\in I\cap K_{k}$ such that $\|F({\bf t};x)-F({\bf s};x)\|_{Y}\leq \epsilon $ for all $x\in B.$ Suppose the contrary. Then for each $k\in {\mathbb N}$ there exists ${\bf t}_{k}\in I$ such that, for every ${\bf s}\in I\cap K_{k},$ there exists $x\in B$ with $\|F({\bf t}_{k};x)-F({\bf s};x)\|_{Y}> \epsilon.$ 
Define ${\bf b}_{k}:={\bf t}_{k}$ for all $k\in {\mathbb N}.$ Due to our assumption, there exists a subsequence $({\bf b}_{k_{l}})$ of $({\bf b}_{k})$ such that \eqref{love12345678ap} holds true. Since $0\in I,$ this implies the existence of a number $l_{0}(\epsilon)\in {\mathbb N}$ such that 
$$
\Bigl \| F\bigl( {\bf t}_{k_{l}} ;x \bigr) -F\bigl( {\bf t}_{k_{m}} ;x \bigr) \Bigr\|_{Y}\leq \epsilon,\quad l,\ m\in {\mathbb N},\ l,\ m\geq l_{0}(\epsilon),
$$ 
uniformly for $x\in B.$ In particular, we have 
$$
\Bigl \| F\bigl( {\bf t}_{k_{l}} ;x \bigr) -F\bigl( {\bf t}_{k_{l_{0}(\epsilon)}} ;x\bigr) \Bigr\|_{Y}\leq \epsilon,\quad l\in {\mathbb N},\ l\geq l_{0}(\epsilon),\ x\in B.
$$
Therefore, ${\bf t}_{k_{l_{0}(\epsilon)}} \notin K_{l}$ for all $l\in {\mathbb N}$ with $l\geq l_{0}(\epsilon),$ which is a contradiction. Now it is quite simply to prove with the help of Cauchy criterion of convergence and the
$({\mathrm R},{\mathcal B})$-multi-almost periodicity of $F(\cdot;\cdot)$
that the set of translations  $\{ F_{B}(\cdot +\tau) : \tau \in I \}$ is relatively compact in $BUC(B:Y).$ Applying \cite[Theorem 2.2; see 1. and 3.(ii)]{RUESS-1} (see also the second part of the proof of \cite[Theorem 3.3]{RUESS-1}), we get that there exist a finite cover $(T_{i})_{i=1}^{k}$ of the set $I_{1}$ and points ${\bf t}_{i}\in T_{i}$ ($1\leq i\leq k$) such that $\|F_{B}({\bf t}+\omega)-F_{B}({\bf t}_{i}+\omega)\|_{BUC(B:Y)}\leq \epsilon$ for all 
$\omega \in I$ and
$t\in T_{i}$ ($1\leq i\leq k$). Let $M_{0}:=l:=1+\max\{|{\bf t}_{i}| : 1\leq i \leq k\},$ and let $M_{1}>0$ satisfy condition (MD) with this $M_{0}.$ Set $M:=2M_{1}+l.$ Suppose that ${\bf t},\ {\bf t}+\tau \in {\mathbb D}_{M}$ and ${\bf t}_{0}\in I'_{M}.$ Then
there exists $i\in {\mathbb N}_{k}$ such that ${\bf t}_{0}\in T_{i}$ so that $\tau={\bf t}_{0}-{\bf t}_{i}\in T_{i}-{\bf t}_{i}\in B({\bf t}_{0},l)\cap I'$ due to the first condition in (MD) and the obvious inequality $|{\bf t}_{i}|\leq l.$ Furthermore, the second condition in (MD) implies
${\bf t}-t_{i}\in I$ and therefore
\begin{align*}
\bigl\| F_{B}({\bf t}+\tau)-& F_{B}({\bf t})\bigr\|_{BUC(B:Y)}=\bigl\| F_{B}({\bf t}+{\bf t}_{0}-{\bf t}_{i})-F_{B}({\bf t})\bigr\|_{BUC(B:Y)}
\\& =\bigl\| F_{B}({\bf t}_{0}+[{\bf t}-{\bf t}_{i}])-F_{B}({\bf t}_{i}+[{\bf t}-{\bf t}_{i}])\bigr\|_{BUC(B:Y)}\leq \epsilon,
\end{align*}
which simply completes the proof.
\end{proof}

\begin{rem}\label{opaska}
\begin{itemize}
\item[(i)] In \cite[Theorem 3.3]{RUESS-1}, W. M. Ruess and W.  H. Summers have considered the situation in which $I=[a,\infty),$ $X=\{0\}$ 
and the set of all translations $\{ f(\cdot +\tau) : \tau\geq 0 \}$ is relatively compact in $BUC(I:Y).$ But, the obtained result is a simple consequence of the corresponding result with $I=[0,\infty),$ which follows from a simple translation argument. In Theorem \ref{izgubljeni-em}, which therefore provides a proper extension of the corresponding result from \cite[Theorem 3.3]{RUESS-1} with ${\mathbb D}=I'=I=[0,\infty)$, we have decided to consider
the collection ${\mathrm R}$ of all sequences in $I,$ only. The interested reader may try to further analyze the assumption in which the function $F(\cdot;\cdot)$ is $({\mathrm R},{\mathcal B})$-multi-almost periodic
with ${\mathrm R}$ being the collection of all sequences in a certain subset $I''$ of ${\mathbb R}^{n}$ which contains $0$ and satisfies $I+I''\subseteq I.$
\item[(ii)] In the multi-dimensional framework, we cannot expect  the situation in which ${\mathbb D}=I'=I$. The main problem lies in the fact that condition (MD) does not hold in this case; but, if $I=[0,\infty)^{n}, $ for example, then the conclusion of Theorem \ref{izgubljeni-em} holds for any proper subsector $I'$ of $I,$ with the meaning clear, and ${\mathbb D}=I'.$
\end{itemize}
\end{rem}

\subsection{Differentiation and integration of $({\mathrm R}_{X},{\mathcal B})$-multi-almost periodic functions}\label{london-calling}

Concerning the partial derivatives of (${\mathbb D}$-asymptotically) $({\mathrm R}_{\mathrm X},{\mathcal B})$-multi-almost periodic functions, we would like to state the following result (by $(e_{1},e_{2},\cdot \cdot \cdot,e_{n})$ we denote the standard basis of ${\mathbb R}^{n}$) which can be also formulated for the notion introduced in Definition \ref{nafaks1234567890}:

\begin{prop}\label{aljaska12345ap}
\begin{itemize}
\item[(i)] Suppose that the function $F(\cdot; \cdot)$ is $({\mathrm R}_{\mathrm X},{\mathcal B})$-multi-almost periodic, for every sequence which belongs to ${\mathrm R}_{\mathrm X},$ we have that any its subsequence belongs to ${\mathrm R}_{\mathrm X},$ the partial derivative 
$$
\frac{\partial F(\cdot ; \cdot)}{\partial t_{i}}:=\lim_{h\rightarrow 0}\frac{F(\cdot +he_{i};\cdot)-F(\cdot; \cdot)}{h}
$$
exists on  $I\times  X$ and it is uniformly continuous on ${\mathcal B},$ i.e.,
\begin{align*}
(\forall B\in {\mathcal B})& \ (\forall \epsilon>0) \ (\exists \delta > 0) \ (\forall {\bf t'},\ {\bf t''} \in I)\ (\forall x\in B)
\\& \Biggl( \bigl| {\bf t'}-{\bf t''} \bigr|<\delta \Rightarrow \Bigl\| \frac{\partial F({\bf t'} ; x)}{\partial t_{i}}-\frac{\partial F({\bf t''} ; x)}{\partial t_{i}}\Bigr\|<\epsilon  \Biggr).
\end{align*}
Then the function $
\frac{\partial F(\cdot ; \cdot)}{\partial t_{i}}$ is  $({\mathrm R}_{\mathrm X},{\mathcal B})$-multi-almost periodic.
\item[(ii)] Suppose that,
for every sequence ${\bf b}(\cdot)$ which belongs to ${\mathrm R},$ any its subsequence belongs to ${\mathrm R}$ and ${\bf T}-{\bf b}(l)\in I$ whenever ${\bf T}\in I$ and $l\in {\mathbb N}.$
Suppose, further, that there exists a sequence in ${\mathrm R}$ whose any subsequence is unbounded as well as that the function $F(\cdot; \cdot)$ is $I$-asymptotically $({\mathrm R},{\mathcal B})$-multi-almost periodic, 
the partial derivative 
$
\frac{\partial F({\bf t} ; x)}{\partial t_{i}}
$
exists for all ${\bf t}\in I$, $ x\in X$ and it is uniformly continuous on ${\mathcal B}.$
Then the function $
\frac{\partial F(\cdot ; \cdot)}{\partial t_{i}}$ is $I$-asymptotically $({\mathrm R},{\mathcal B})$-multi-almost periodic.
\end{itemize}
\end{prop}

\begin{proof}
We will prove only (i) because (ii) follows similarly, by appealing to Proposition \ref{2.1.10obazacap} in place of 
Proposition \ref{2.1.10ap1} (observe that we only need here the uniform convergence of the sequence of functions $(F_{j}(\cdot;\cdot))$ to the function $\frac{\partial F(\cdot ; \cdot)}{\partial t_{i}}$ as $j\rightarrow +\infty,$ on the individual sets $B\in {\mathcal B};$ see the proof of Proposition \ref{2.1.10ap1}). The proof immediately follows from the fact that the sequence $(F_{j}(\cdot;\cdot)\equiv j[F(\cdot +j^{-1}e_{i}; \cdot)-F(\cdot;\cdot)])$ of $({\mathrm R}_{\mathrm X},{\mathcal B})$-multi-almost periodic functions converges uniformly to the function $\frac{\partial F(\cdot ; \cdot)}{\partial t_{i}}$ as $j\rightarrow +\infty.$ This can be shown as in the one-dimesional case, by observing that
$$
F_{j}(\cdot;\cdot)-\frac{\partial F(\cdot ; \cdot)}{\partial t_{i}}=j\int^{1/j}_{0}\Biggl[ \frac{\partial F(\cdot +se_{i}; \cdot)}{\partial t_{i}}-\frac{\partial F(\cdot ; \cdot)}{\partial t_{i}}\Biggr]\, ds.
$$
\end{proof}

We continue by stating the following extension of S. M. A. Alsulami's result \cite[Theorem 3.2]{integral-als}:

\begin{thm}\label{ubico-ubij-se}
Suppose that the function $F : {\mathbb R}^{n}
\times X \rightarrow Y$ is continuous as well as that  $
\frac{\partial F(\cdot ; \cdot)}{\partial t_{i}} : {\mathbb R}^{n}
\times X \rightarrow Y$ is a Bohr ${\mathcal B}$-almost periodic function, where ${\mathcal B}$ is any colection of compact subset of $X.$ 
Suppose that for each $B\in {\mathcal B}$ we have that at least one of the following two conditions holds:
\begin{itemize}
\item[(C1)] The Banach space $l_{\infty}(B : Y)$ does not contain $c_{0};$
\item[(C2)] The range of function $F_{B}(\cdot),$ given by \eqref{sarajevo-london}, is weakly relatively compact in $l_{\infty}(B : Y).$
\end{itemize}
Then the function $F(\cdot;\cdot)$ is Bohr ${\mathcal B}$-almost periodic.
\end{thm}

\begin{proof}
Let $B\in {\mathcal B}$ and ${\bf a}=(a_{1},\cdot \cdot \cdot,a_{n})\in {\mathbb R}^{n}$ be fixed. As easily approved, it suffices to show that the mapping $F_{B}(\cdot)$ is almost periodic. Since we have assumed (C1) or (C2), an application of an old result of B. Basit (see e.g., \cite[Theorem 3.1]{integral-als}) shows that we only need to prove that the function 
$$
{\bf t} \mapsto F_{B}({\bf t}+{\bf a})-F_{B}({\bf t}),\quad {\bf t} \in {\mathbb R}^{n}
$$
is almost periodic.
So, let $({\bf b}_{k})$ be a sequence in ${\mathbb R}^{n}.$ Since the mapping 
$$
\Biggl(\Bigl(\frac{\partial F(\cdot ; \cdot)}{\partial t_{1}}\Bigr)_{B},\cdot \cdot \cdot, \Bigl(\frac{\partial F(\cdot ; \cdot)}{\partial t_{n}}\Bigr)_{B}\Biggr) : {\mathbb R}^{n} \rightarrow \bigl(l_{\infty}(B : Y)\bigr)^{n}
$$
is almost periodic (see Proposition \ref{dekartovproizvod}), 
there exists a subsequence $({\bf b}_{k_{l}})$
of $({\bf b}_{k})$ such that
\begin{align}\label{clash-bgd-ns}
\lim_{l\rightarrow +\infty} \sup_{x\in B}\Biggl\|\Bigl(\frac{\partial F(\cdot ; x)}{\partial t_{i}}\Bigr)\bigl( {\bf s}+{\bf b}_{k_{l}} \bigr)-
\Bigl(\frac{\partial F(\cdot ; x)}{\partial t_{i}}\Bigr)\bigl( {\bf s} \bigr)\Biggr\|_{Y}=0,
\end{align}
uniformly in ${\bf s}\in {\mathbb R}^{n}$ and $1\leq i\leq n.$ Since, for every $x\in B,$ we have
\begin{align*}
&\Biggl\{\Bigl[ F({\bf t}+{\bf a}+{\bf b}_{k_{l}})-F_{B}({\bf t}+{\bf b}_{k_{l}})\Bigr]-\Bigl[ F_{B}({\bf t}+{\bf a})-F_{B}({\bf t})\Bigr]\Biggr\}(x)
\\& =\sum_{i=1}^{n}\int^{a_{i}}_{0}F_{t_{i}}\bigl(s_{1}+{\bf b}_{k_{l}}^{1},\cdot\cdot \cdot,s_{i-1}+{\bf b}_{k_{l}}^{i-1},s_{i}+{\bf b}_{k_{l}}^{i}+v,\\& s_{i+1}+{\bf b}_{k_{l}}^{i+1}+a_{i+1},\cdot \cdot \cdot ,s_{n}+{\bf b}_{k_{l}}^{n}+a_{n};x\bigl)\, dv
\\&-\sum_{i=1}^{n}\int^{a_{i}}_{0}F_{t_{i}}\bigl(s_{1},\cdot\cdot \cdot,s_{i-1},s_{i}+v,s_{i+1}+a_{i+1},\cdot \cdot \cdot ,s_{n}+a_{n};x\bigl)\, dv,
\end{align*}
applying \eqref{clash-bgd-ns} we simply get that
$$
\lim_{l\rightarrow +\infty}  \Biggl\| \Bigl[ F({\bf t}+{\bf a}+{\bf b}_{k_{l}})-F_{B}({\bf t}+{\bf b}_{k_{l}})\Bigr]-\Bigl[ F_{B}({\bf t}+{\bf a})-F_{B}({\bf t})\Bigr]\Biggr\|_{l_{\infty}(B : Y)}=0,
$$
uniformly in ${\bf t}\in {\mathbb R}^{n}.$
The proof of the theorem is thereby complete.
\end{proof}

\begin{cor}\label{idiot-london}
Suppose that the function $F : {\mathbb R}^{n}
\rightarrow Y$ is continuous as well as that  $
\frac{\partial F(\cdot)}{\partial t_{i}} : {\mathbb R}^{n}
\rightarrow Y$ is an almost periodic function.
Suppose that at least one of the following two conditions holds:
\begin{itemize}
\item[(C1)'] The Banach space $Y$ does not contain $c_{0};$
\item[(C2)'] The range of function $F(\cdot)$ is weakly relatively compact in $Y.$
\end{itemize}
Then the function $F(\cdot)$ is almost periodic.
\end{cor}

The proof of following extension of \cite[Theorem 3.2]{integral-als}, which has already been clarified in the introductory part for the scalar-valued functions, is simple and therefore omitted:

\begin{thm}\label{london-gent}
Suppose that the function $F : {\mathbb R}^{n} \times X \rightarrow Y$ is Bohr ${\mathcal B}$-almost periodic and for each set $B\in {\mathcal B}$ 
the  Banach space $l_{\infty}(B : Y)$ does not contain $c_{0}$
or
the function $H(t_{1},t_{2},\cdot \cdot \cdot, t_{n};x):=\int^{t_{1}}_{0}F(t,t_{2},\cdot \cdot \cdot, t_{n};x)\, dt,$ $(t_{1},t_{2},\cdot \cdot \cdot, t_{n}) \in {\mathbb R}^{n},$ $x\in X$
satisfies that the range of function $H_{B}(\cdot),$ given by \eqref{sarajevo-london} with $F=G$ therein, is weakly relatively compact in $l_{\infty}(B : Y).$ If
there exist Bohr ${\mathcal B}$-almost periodic functions $G_{i} : {\mathbb R}^{n} \times X \rightarrow Y$ such that $F_{t_{i}}(t_{1},t_{2},\cdot \cdot \cdot, t_{n};x) =(G_{i})_{t_{1}}(t_{1},t_{2},\cdot \cdot \cdot, t_{n};x) $ is a continuous function on ${\mathbb R}^{n}$ for each fixed element $x\in X$ ($2\leq i\leq n$), then the function $H : {\mathbb R}^{n} \times X \rightarrow Y$
is Bohr ${\mathcal B}$-almost periodic.
\end{thm}

\begin{cor}\label{london-gentac123}
Suppose that the function $F : {\mathbb R}^{n}  \rightarrow Y$ is almost periodic and 
the  Banach space $Y$ does not contain $c_{0}$
or
the function $H(t_{1},t_{2},\cdot \cdot \cdot, t_{n}):=\int^{t_{1}}_{0}F(t,t_{2},\cdot \cdot \cdot, t_{n})\, dt,$ $(t_{1},t_{2},\cdot \cdot \cdot, t_{n}) \in {\mathbb R}^{n}$
satisfies that its range is weakly relatively compact in $Y.$ If
there exist almost periodic functions $G_{i} : {\mathbb R}^{n}  \rightarrow Y$ such that $F_{t_{i}}(t_{1},t_{2},\cdot \cdot \cdot, t_{n}) =(G_{i})_{t_{1}}(t_{1},t_{2},\cdot \cdot \cdot, t_{n}) $ is a continuous function on ${\mathbb R}^{n},$ for $2\leq i\leq n$, then the function $H : {\mathbb R}^{n}\rightarrow Y$
is almost periodic.
\end{cor}

The interested reader may try to extend the results of \cite[Theorem 4.1, Theorem 4.2]{integral-als}, regarding the  almost periodicity of function
$$
{\bf t} \mapsto \int^{t_{1}}_{0}\int^{t_{2}}_{0}\cdot \cdot \cdot \int^{t_{n}}_{0}F(s_{1},s_{2},\cdot \cdot \cdot,s_{n})\, ds_{1} \, ds_{2}\cdot \cdot \cdot ds_{n},\quad {\bf t}=(t_{1},t_{2},\cdot \cdot \cdot ,t_{n})\in {\mathbb R}^{n},
$$ 
in the above manner. The results about integration of multi-dimensional asymptotically almost periodic functions and related connections with the weak asymptotic
almost periodicity, obtained in \cite[Section 4]{RUESS-2}, will be reconsidered somewhere else.

\subsection{Composition theorems for $({\mathrm R},{\mathcal B})$-multi-almost periodic type functions}\label{marekzero}

Suppose that $F : I \times X \rightarrow Y$ and $G : I \times Y \rightarrow Z$ are given functions. The main aim of this subsection is to analyze the almost periodic properties of the multi-dimensional Nemytskii operator
$W : I  \times X \rightarrow Z$ given by
$$
W({\bf t}; x):=G\bigl({\bf t} ; F({\bf t}; x)\bigr),\quad {\bf t} \in I,\ x\in X.
$$

We will first state the following generalization of \cite[Theorem 4.16]{diagana}; the proof is similar to the proof of the above-mentioned theorem but we will present it for the sake
of completeness:

\begin{thm}\label{eovakoonakoap}
Suppose that $F : I \times X \rightarrow Y$ is $({\mathrm R},{\mathcal B})$-multi-almost periodic and $G : I \times Y \rightarrow Z$ is $({\mathrm R}',{\mathcal B}')$-multi-almost periodic,
where ${\mathrm R}'$ is a collection of all sequences $b : {\mathbb N} \rightarrow {\mathbb R}^{n}$ from ${\mathrm R}$ and all their subsequences, as well as
\begin{align}\label{ljuvenautoap}
{\mathcal B}':=\Biggl\{\bigcup_{{\bf t}\in I}F(t; B) : B\in {\mathcal B}\Biggr\}.
\end{align}
If there exists a finite constant $L>0$ such that
\begin{align}\label{zigzverinijeap}
\bigl\| G({\bf t} ;x) -G({\bf t} ;y)\bigr\|_{Z}\leq L\|x-y\|_{Y},\quad {\bf t}\in I,\ x,\ y\in Y,
\end{align}
then the function $W(\cdot; \cdot)$ is $({\mathrm R},{\mathcal B})$-multi-almost periodic.
\end{thm}

\begin{proof}
Let the set $B\in {\mathcal B}$ and the sequence $({\bf b}_{k}=(b_{k}^{1},b_{k}^{2},\cdot \cdot\cdot ,b_{k}^{n})) \in {\mathrm R}$ be given.
By definition, there exist a subsequence $({\bf b}_{k_{l}}=(b_{k_{l}}^{1},b_{k_{l}}^{2},\cdot \cdot\cdot , b_{k_{l}}^{n}))$ of $({\bf b}_{k})$ and a function
$F^{\ast} : I \times X \rightarrow Y$ such that \eqref{love12345678ap} holds. Set $B':=\bigcup_{{\bf t}\in I}F(t; B)$ and $b':=({\bf b}_{k_{l}}).$
Then there exist a subsequence $({\bf b}_{k_{l_{m}}}=(b_{k_{l_{m}}}^{1},b_{k_{l_{m}}}^{2},\cdot \cdot\cdot , b_{k_{l_{m}}}^{n}))$ of $({\bf b}_{k_{l}})$ and a function
$G^{\ast} : I \times Y \rightarrow Z$ such that
\begin{align*}
\lim_{m\rightarrow +\infty}\Bigl\| G\bigl({\bf t} +(b_{k_{l_{m}}}^{1},\cdot \cdot\cdot, b_{k_{l_{m}}}^{n});y\bigr)-G^{\ast}({\bf t};y) \Bigr\|_{Z}=0,
\end{align*}
uniformly for $y\in B'$ and ${\bf t}\in I.$
It suffices to show that
\begin{align}\label{emanuelap}
\lim_{m\rightarrow +\infty}\Bigl\| G\bigl({\bf t} +(b_{k_{l_{m}}}^{1},\cdot \cdot\cdot, b_{k_{l_{m}}}^{n});F\bigl({\bf t} +(b_{k_{l_{m}}}^{1},\cdot \cdot\cdot, b_{k_{l_{m}}}^{n});x\bigr)\bigr)-G^{\ast}({\bf t};F^{\ast}({\bf t};x)) \Bigr\|_{Z}=0,
\end{align}
uniformly for $x\in B$ and ${\bf t}\in I.$
Denote ${\bf \tau_{m}}:=(b_{k_{l_{m}}}^{1},\cdot \cdot\cdot, b_{k_{l_{m}}}^{n})$ for all $m\in {\mathbb N}$. We have (${\bf t}\in I,$ $x\in B,$ $m\in {\mathbb N}$):
\begin{align*}
&\Bigl\| G\bigl({\bf t} +{\bf \tau_{m}};F\bigl({\bf t} +{\bf \tau_{m}};x\bigr)\bigr)-G^{\ast}({\bf t};F^{\ast}({\bf t};x)) \Bigr\|_{Z}
\\& \leq  \Bigl\| G\bigl({\bf t} +{\bf \tau_{m}};F\bigl({\bf t} +{\bf \tau_{m}};x\bigr)\bigr)-G({\bf t}+{\bf \tau_{m}};F^{\ast}({\bf t};x)) \Bigr\|_{Z}
\\& + \Bigl\| G({\bf t}+{\bf \tau_{m}};F^{\ast}({\bf t};x)) -G^{\ast}({\bf t};F^{\ast}({\bf t};x)) \Bigr\|_{Z}
\\& \leq L\Bigl\|F\bigl({\bf t} +{\bf \tau_{m}};x\bigr)-F^{\ast}({\bf t};x) \Bigr\|_{Y}+\Bigl\| G({\bf t}+{\bf \tau_{m}};F^{\ast}({\bf t};x)) -G^{\ast}({\bf t};F^{\ast}({\bf t};x)) \Bigr\|_{Z}.
\end{align*}
Since $x\in B$ and $F^{\ast}({\bf t};x)\in B'$ for all ${\bf t}\in I$, the limit equality \eqref{emanuelap} holds,
uniformly for $x\in B$ and ${\bf t}\in I,$
which completes the proof of the theorem.
\end{proof}

Keeping in mind Proposition \ref{bounded-pazi}, Theorem \ref{Bochner123456}, Theorem \ref{eovakoonakoap} and the fact that a continuous function
$F : I \times X \rightarrow Y$ is $({\mathrm R},{\mathcal B})$-multi-almost periodic (Bohr ${\mathcal B}$-almost periodic) if and only if it is $({\mathrm R},\overline{{\mathcal B}})$-multi-almost periodic (Bohr $\overline{{\mathcal B}}$-almost periodic), where $\overline{{\mathcal B}}:=\{\overline{B} : B\in {\mathcal B}\},$
we can immediately clarify the following:

\begin{cor}
\label{ovakorajko}
Suppose that ${\mathcal B}$ is any collection of compact subsets of $X,$ $F : {\mathbb R}^{n} \times X \rightarrow Y$ is Bohr ${\mathcal B}$-almost periodic and $G : {\mathbb R}^{n} \times Y \rightarrow Z$ is Bohr ${\mathcal B}'$-almost periodic,
where ${\mathcal B}'$ is given by \eqref{ljuvenautoap}.
If there exists a finite constant $L>0$ such that \eqref{zigzverinijeap} holds with $I={\mathbb R}^{n},$ then 
the function $W(\cdot; \cdot)$ is Bohr ${\mathcal B}$-almost periodic.
\end{cor}

A slight modification of the proof of Theorem \ref{eovakoonakoap} (cf. also the proof of \cite[Theorem 3.31]{diagana}) shows that the following result holds true:

\begin{thm}\label{eovako12345onakoap}
Suppose that $F : I\times X \rightarrow Y$ is $({\mathrm R},{\mathcal B})$-multi-almost periodic and $G : I \times Y \rightarrow Z$ is $({\mathrm R}',{\mathcal B}')$-multi-almost periodic,
where ${\mathrm R}'$ is a collection of all sequences $b : {\mathbb N} \rightarrow {\mathbb R}^{n}$ from ${\mathrm R}$ and all their subsequences, as well as
$
{\mathcal B}'$ is given by \eqref{ljuvenautoap}.
Set
$$
{\mathcal B}^{'*}:=\bigcup_{({\bf b}_{k})\in {\mathrm R} ; B\in {\mathcal B}}\Biggl\{F^{*}(t; B) : {\bf t}\in I \Biggr\},
$$
with the meaning clear.
If 
\begin{align*}
& (\forall  B \in {\mathcal B}) \ (\forall \epsilon >0) \ (\exists \delta >0) 
\\ & \Bigl( x,\ y \in  {\mathcal B}' \cup {\mathcal B}^{'*}\mbox{ and }\ \bigl\| x-y\bigr\|_{Y} <\delta \Rightarrow \bigl\| G({\bf t};x)-G({\bf t};y)\bigr\|_{Z}<\epsilon,\ {\bf t} \in I \Bigr),
\end{align*}
then the function $W(\cdot; \cdot)$ is $({\mathrm R},{\mathcal B})$-multi-almost periodic.
\end{thm}

Now we proceed with the analysis of composition theorems for asymptotically $({\mathrm R},{\mathcal B})$-multi-almost periodic functions. Our first result is in a close connection with Theorem \ref{eovakoonakoap} and \cite[Theorem 3.49]{diagana}:

\begin{thm}\label{eovakoonakoapap}
Suppose that the set ${\mathbb D} \subseteq {\mathbb R}^{n}$ is unbounded,
$F_{0} : I \times X \rightarrow Y$ is $({\mathrm R},{\mathcal B})$-multi-almost periodic,
$Q_{0}\in C_{0,{\mathbb D},{\mathcal B}}(I \times X : Y)$ and $F({\bf t} ; x)=F_{0}({\bf t} ; x)
+Q_{0}({\bf t} ; x)$ for all ${\bf t}\in I$ and $x\in X.$ 
Suppose further that $G_{1} : I \times Y \rightarrow Z$ is $({\mathrm R}',{\mathcal B}')$-multi-almost periodic,
where ${\mathrm R}'$ is a collection of all sequences $b : {\mathbb N} \rightarrow  {\mathbb R}^{n}$ from ${\mathrm R}$ and all their subsequences as well as ${\mathcal B}'$ is defined by \eqref{ljuvenautoap} with the function $F(\cdot;\cdot)$ replaced therein by the function
$F_{0}(\cdot;\cdot),$
$Q_{1}\in C_{0,{\mathbb D},{\mathcal B}_{1}}(I \times Y : Z),$
where
\begin{align}\label{objasni}
{\mathcal B}_{1}:=\Biggl\{\bigcup_{{\bf t}\in I}F(t; B) : B\in {\mathcal B}\Biggr\},
\end{align}
and $G({\bf t} ; x)=G_{1}({\bf t} ; x)
+Q_{1}({\bf t} ; x)$ for all ${\bf t}\in I$ and $x\in Y.$
If there exists a finite constant $L>0$ such that the estimate \eqref{zigzverinijeap} holds with the function $G(\cdot;\cdot)$ replaced therein by the function
$G_{1}(\cdot;\cdot),$
then the function $W(\cdot; \cdot)$ is ${\mathbb D}$-asymptotically $({\mathrm R},{\mathcal B})$-multi-almost periodic.
\end{thm}

\begin{proof}
By Theorem \ref{eovakoonakoap}, the function $({\bf t};x)\mapsto G_{1}({\bf t}; F_{0}({\bf t};x)),$ ${\bf t}\in I,$ $x\in X$ is $({\mathrm R},{\mathcal B})$-multi-almost periodic.
Furthermore, we have the following decomposition
\begin{align*}
W({\bf t}; x)=G_{1}({\bf t}; F_{0}({\bf t};x))+\Bigl[ G_{1}({\bf t}; F({\bf t};x))-G_{1}({\bf t}; F_{0}({\bf t};x)) \Bigr]+Q_{1}({\bf t}; F({\bf t}; x)),
\end{align*}
for any ${\bf t}\in I$ and $x\in X.$
Since 
$$
\Bigl\| G_{1}({\bf t}; F({\bf t};x))-G_{1}({\bf t}; F_{0}({\bf t};x)) \Bigr\|_{Z}\leq L \bigl\| Q_{0}({\bf t} ; x)\bigr\|_{Y},\quad  {\bf t}\in I,\ x\in X,
$$
we have that the function $({\bf t};x)\mapsto G_{1}({\bf t}; F({\bf t};x))-G_{1}({\bf t}; F_{0}({\bf t};x)),$ ${\bf t}\in I,$ $x\in X$ belongs to the space $C_{0,{\mathbb D},{\mathcal B}}(I \times X : Z).$ The same holds for the function 
$({\bf t};x)\mapsto Q_{1}({\bf t}; F({\bf t}; x)),$ ${\bf t}\in I,$ $x\in X$ due to our choice of the collection ${\mathcal B}_{1}$ in \eqref{objasni}. 
\end{proof}

\begin{cor}\label{eovakoonakoapaprcv}
Suppose that ${\mathcal B}$ is any collection of compact subsets of $X,$ the set ${\mathbb D} \subseteq {\mathbb R}^{n}$ is unbounded,
$F_{0} : {\mathbb R}^{n} \times X \rightarrow Y$ is Bohr ${\mathcal B}$-almost periodic,
$Q_{0}\in C_{0,{\mathbb D},{\mathcal B}}(I \times X : Y)$ and $F({\bf t} ; x)=F_{0}({\bf t} ; x)
+Q_{0}({\bf t} ; x)$ for all ${\bf t}\in I$ and $x\in X.$ 
Suppose further that $G_{1} : I \times Y \rightarrow Z$ is Bohr ${\mathcal B}'$-almost periodic,
where ${\mathcal B}'$ is defined by \eqref{ljuvenautoap} with the function $F(\cdot;\cdot)$ replaced therein by the function
$F_{0}(\cdot;\cdot),$
$Q_{1}\in C_{0,{\mathbb D},{\mathcal B}_{1}}(I \times Y : Z),$
where ${\mathcal B}_{1}$ is given by \eqref{objasni}
and $G({\bf t} ; x)=G_{1}({\bf t} ; x)
+Q_{1}({\bf t} ; x)$ for all ${\bf t}\in I$ and $x\in Y.$
If there exists a finite constant $L>0$ such that the estimate \eqref{zigzverinijeap} holds with the function $G(\cdot;\cdot)$ replaced therein by the function
$G_{1}(\cdot;\cdot),$
then the function $W(\cdot; \cdot)$ is ${\mathbb D}$-asymptotically Bohr ${\mathcal B}$-almost periodic.
\end{cor}

It seems that we cannot remove the Lipschitz type assumptions used in Corollary \ref{ovakorajko} and Corollary \ref{eovakoonakoapaprcv} without imposing some additional conditions; but, this can be always done in the case that $F : {\mathbb R}^{n} \rightarrow Y$ is Bohr almost periodic and $G : {\mathbb R}^{n} \times Y \rightarrow Z$ is Bohr ${\mathcal B}$-almost periodic with $\overline{R(F)}=B\in {\mathcal B};$ see e.g., \cite[Theorem 2.11, p. 27]{fink} and its proof for the scalar-valued case. Keeping in mind Proposition \ref{nijenaivno}, we can state the following extension of this result:

\begin{thm}\label{nijenac}
Suppose that the set $I$ is admissible with respect to the almost periodic extensions. If $F : I \rightarrow Y$ is uniformly continuous, Bohr almost periodic and $G : I \times Y \rightarrow Z$ is Bohr ${\mathcal B}$-almost periodic with $\overline{R(F)}=B\in {\mathcal B},$
then the function $W: I \rightarrow Z$ is uniformly continuous and Bohr almost periodic, provided that the function $G(\cdot;\cdot)$ is uniformly continuous on $I\times B.$
\end{thm}

\begin{proof}
It is clear that there exists a unique almost periodic extension $\tilde{F} : {\mathbb R}^{n} \rightarrow Y$ of the function $F(\cdot)$ to the whole Euclidean space and there exists  a unique almost periodic extension $\widetilde{G_{B}} : {\mathbb R}^{n} \rightarrow l_{\infty}(B : Z)$ of the function $G_{B}(\cdot)$ to the whole Euclidean space since the function $F(\cdot)$ is uniformly continuous and the function $G(\cdot;\cdot)$ is uniformly continuous on $I\times B.$ Define
$$
\tilde{W}({\bf t}):=\Bigl[\widetilde{G_{B}}({\bf t})\Bigr]\Bigl(\tilde{F}({\bf t})\Bigr),\quad {\bf t}\in {\mathbb R}^{n}.
$$
Since $W({\bf t})=[G_{B}({\bf t})](F({\bf t}))$ for all ${\bf t}\in I,$ it is clear that the function $\tilde{W}(\cdot)$ extends the function $W(\cdot)$ to the whole Euclidean space. Furthermore, by the proof of Theorem \ref{lenny-jasson}, we have 
that $R(\tilde{F})\subseteq B$ as well as that there exists
a sequence $({\bf \tau}_{k})$ in $I$ such that $\lim_{k\rightarrow +\infty}|{\bf \tau}_{k}|=+\infty$ and
$\lim_{k\rightarrow +\infty}G_{B}({\bf t}+{\bf \tau}_{k})=\widetilde{G_{B}}({\bf t}),$ uniformly for $t\in I.$ 
In order to see that the function $\tilde{W}(\cdot)$ is
uniformly continuous on ${\mathbb R}^{n},$
we can use the following calculus
\begin{align*}
&\Bigl\| \tilde{W}({\bf t}')-\tilde{W}({\bf t}'') \Bigr\|_{Y}=\Bigl\| \Bigl[\widetilde{G_{B}}({\bf t}')\Bigr]\Bigl(\tilde{F}({\bf t}')\Bigr)-\Bigl[\widetilde{G_{B}}({\bf t}'')\Bigr]\Bigl(\tilde{F}({\bf t}'')\Bigr)\Bigr\|_{Y}
\\& \leq 
\Bigl\| \Bigl[\widetilde{G_{B}}({\bf t}')\Bigr]\Bigl(\tilde{F}({\bf t}')\Bigr)-\Bigl[\widetilde{G_{B}}({\bf t}'')\Bigr]\Bigl(\tilde{F}({\bf t}')\Bigr) \Bigr\|_{Y}
\\&+\Bigl\| \Bigl[\widetilde{G_{B}}({\bf t}'')\Bigr]\Bigl(\tilde{F}({\bf t}')\Bigr)-\Bigl[\widetilde{G_{B}}({\bf t}'')\Bigr]\Bigl(\tilde{F}({\bf t}'')\Bigr) \Bigr\|_{Y}
\\& \leq \sup_{x\in B}\Bigl\| \Bigl[\widetilde{G_{B}}({\bf t}')\Bigr](x)-\Bigl[\widetilde{G_{B}}({\bf t}'')\Bigr](x) \Bigr\|_{Y}
\\&+\limsup_{k\rightarrow +\infty}\Bigl\| \bigl[G_{B}({\bf t}''+{\bf \tau}_{k})\bigr]\bigl(\tilde{F}({\bf t}')\bigr)-\bigl[G_{B}({\bf t}''+{\bf \tau}_{k})\bigr]\bigl(\tilde{F}({\bf t}'')\bigr) \Bigr\|_{Y}
\\&=\sup_{x\in B}\Bigl\| \Bigl[\widetilde{G_{B}}({\bf t}')\Bigr](x)-\Bigl[\widetilde{G_{B}}({\bf t}'')\Bigr](x) \Bigr\|_{Y}
\\&+\limsup_{k\rightarrow +\infty}\Bigl\| \bigl[G\bigl({\bf t}''+{\bf \tau}_{k};\tilde{F}({\bf t}')\bigr)-G\bigl({\bf t}''+{\bf \tau}_{k};\tilde{F}({\bf t}'')\bigr) \Bigr\|_{Y},\quad {\bf t}',\ {\bf t}''\in {\mathbb R}^{n},
\end{align*}
the uniform continuity of $\widetilde{G_{B}}(\cdot)$ and the uniform continuity of $G(\cdot;\cdot)$ on $I\times B.$ Due to
Proposition \ref{dekartovproizvod}, for every $\epsilon>0,$ the functions $\tilde{F}(\cdot)$ and $\widetilde{G_{B}}(\cdot)$ can share the same set of $\epsilon$-almost periods which is relatively dense in ${\mathbb R}^{n}.$ Keeping in mind this fact, we can repeat almost verbatim the above calculus, with the numbers ${\bf t}'={\bf t}\in {\mathbb R}^{n}$ and ${\bf t}''={\bf t}+{\bf \tau}\in {\mathbb R}^{n}$ so as to conclude that the function $\tilde{W}(\cdot)$ is Bohr almost periodic on ${\mathbb R}^{n},$ finishing the proof.
\end{proof}

We can also prove the following result which corresponds to Theorem \ref{eovako12345onakoap} and \cite[Theorem 3.50]{diagana}:

\begin{thm}\label{eovakoonakoap1}
Suppose that the set ${\mathbb D} \subseteq {\mathbb R}^{n}$ is unbounded,
$F_{0} : I \times X \rightarrow Y$ is $({\mathrm R},{\mathcal B})$-multi-almost periodic,
$Q_{0}\in C_{0,{\mathbb D},{\mathcal B}}(I \times X : Y)$ and $F({\bf t} ; x)=F_{0}({\bf t} ; x)
+Q_{0}({\bf t} ; x)$ for all ${\bf t}\in I$ and $x\in X.$ 
Suppose further that $G_{1} : I \times Y \rightarrow Z$ is $({\mathrm R}',{\mathcal B}')$-multi-almost periodic,
where ${\mathrm R}'$ is a collection of all sequences $b : {\mathbb N} \rightarrow  {\mathbb R}^{n}$ from ${\mathrm R}$ and all their subsequences as well as ${\mathcal B}'$ is defined by \eqref{ljuvenautoap} with the function $F(\cdot;\cdot)$ replaced therein by the function
$F_{0}(\cdot;\cdot),$
$Q_{1}\in C_{0,{\mathbb D},{\mathcal B}_{1}}(I \times Y : Z),$
where
${\mathcal B}_{1}$ is given through \eqref{objasni},
and $G({\bf t} ; x)=G_{1}({\bf t} ; x)
+Q_{1}({\bf t} ; x)$ for all ${\bf t}\in I$ and $x\in Y.$
Set
$$
{\mathcal B}_{2}:=\Biggl\{\bigcup_{{\bf t}\in I}F_{0}(t; B) : B\in {\mathcal B}\Biggr\} \cup \bigcup_{({\bf b}_{k})\in {\mathrm R} ; B\in {\mathcal B}}\Biggl\{F^{*}_{0}(t; B) : {\bf t}\in I \Biggr\}.
$$
If 
\begin{align*}
& (\forall B \in {\mathcal B}) \ (\forall \epsilon >0) \ (\exists \delta >0) 
\\ & \Bigl( x,\ y \in  {\mathcal B}_{1} \cup {\mathcal B}_{2}\mbox{ and }\ \bigl\| x-y\bigr\|_{Y} <\delta \Rightarrow \bigl\| G_{1}({\bf t};x)-G_{1}({\bf t};y)\bigr\|_{Z}<\epsilon,\ {\bf t} \in I\Bigr),
\end{align*}
then the function $W(\cdot; \cdot)$ is ${\mathbb D}$-asymptotically $({\mathrm R},{\mathcal B})$-multi-almost periodic.
\end{thm}

It is clear that Theorem \ref{eovako12345onakoap} and Theorem \ref{eovakoonakoap1} can be reformulated for Bohr ${\mathcal B}$-almost periodic functions with small terminological difficulties concerning the use of  limit functions. Similar results can be established for the class of ${\mathcal B}$-uniformly recurrent functions (\cite{nova-man}). 

\subsection{Invariance of $({\mathrm R},{\mathcal B})$-multi-almost periodicity under the actions of convolution products}\label{sabseksn}
This subsection investigates the invariance of $({\mathrm R},{\mathcal B})$-multi-almost periodicity under the actions of convolution products.
We will use the following notation: 
if any component of tuple ${\bf t}=(t_{1},t_{2},\cdot \cdot \cdot, t_{n})$ is strictly positive, then we simply write ${\bf t}> {\bf 0}.$ 

We start by stating the following result, which is very similar to \cite[Proposition 2.6.11]{nova-mono} (the main details of proof for Stepanov generalizations will be given in our forthcoming paper \cite{genralized-stepamultiap}): 

\begin{thm}\label{krucija}
Let $(R({\bf t}))_{{\bf t}> {\bf 0}}\subseteq L(X,Y)$ be a strongly continuous operator family such that
$\int_{(0,\infty)^{n}}\|R({\bf t} )\|\, d{\bf t}<\infty .$ If $f : {\mathbb R}^{n} \rightarrow X$ is almost periodic, then the function $F: {\mathbb R}^{n} \rightarrow Y,$ given by
\begin{align}\label{wer}
F({\bf t}):=\int^{t_{1}}_{-\infty}\int^{t_{2}}_{-\infty}\cdot \cdot \cdot \int^{t_{n}}_{-\infty} R({\bf t}-{\bf s})f({\bf s})\, d{\bf s},\quad {\bf t}\in {\mathbb R}^{n},
\end{align}
is well-defined and almost periodic.
\end{thm}

For completeness, we will include the proof of following result:

\begin{thm}\label{krucija-rrrr}
Let $(R({\bf t}))_{{\bf t}> {\bf 0}}\subseteq L(X,Y)$ be a strongly continuous operator family such that
$\int_{(0,\infty)^{n}}\|R({\bf t} )\|\, d{\bf t}<\infty .$ If $f : {\mathbb R}^{n} \rightarrow X$ is a bounded ${\mathrm R}$-almost periodic function, then the function $F: {\mathbb R}^{n} \rightarrow Y,$ given by \eqref{wer},
is well-defined and ${\mathrm R}$-almost periodic.
\end{thm}

\begin{proof}
Let $({\bf b}_{k}=(b_{k}^{1},b_{k}^{2},\cdot \cdot\cdot ,b_{k}^{n})) \in {\mathrm R}$ be given. Then there exist a subsequence $({\bf b}_{k_{l}}=(b_{k_{l}}^{1},b_{k_{l}}^{2},\cdot \cdot\cdot , b_{k_{l}}^{n}))$ of $({\bf b}_{k})$ and a function
$f^{\ast} : {\mathbb R}^{n} \rightarrow X$ such that
$
\lim_{l\rightarrow +\infty}f({\bf t} +(b_{k_{l}}^{1},\cdot \cdot\cdot, b_{k_{l}}^{n}))=f^{\ast}({\bf t}) 
$
uniformly for ${\bf t}\in {\mathbb R}^{n}.$ Hence, the function $f^{\ast} : {\mathbb R}^{n} \rightarrow X$ is bounded and measurable. Clearly, 
\begin{align}\label{swishtime}
F({\bf t})=\int_{[0,\infty)^{n}}R({\bf s}) f({\bf t}-{\bf s})\, d{\bf s}\mbox{ for all }{\bf t}\in {\mathbb R}^{n}
\end{align}
and the integral $ \int_{[0,\infty)^{n}}R({\bf s}) f^{\ast}({\bf t}-{\bf s})\, d{\bf s}$
is well defined for all ${\bf t}\in {\mathbb R}^{n}.$ 
Furthermore,  
$$
\lim_{l\rightarrow \infty} \int_{[0,\infty)^{n}}R({\bf s}) f({\bf t}+{\bf b}_{k_{l}}-{\bf s})\, d{\bf s}= \int_{[0,\infty)^{n}}R({\bf s}) f^{\ast}({\bf t}-{\bf s})\, d{\bf s}
$$
uniformly for ${\bf t}\in {\mathbb R}^{n},$ because 
\begin{align*}
\Biggl\|\int_{[0,\infty)^{n}}R({\bf s}) & f({\bf t}+{\bf b}_{k_{l}}-{\bf s})\, d{\bf s}- \int_{[0,\infty)^{n}}R({\bf s}) f^{\ast}({\bf t}-{\bf s})\, d{\bf s}\Biggr\|_{Y}
\\& \leq \int_{[0,\infty)^{n}}\|R({\bf s})\| \cdot \bigl\| f({\bf t}+{\bf b}_{k_{l}}-{\bf s})- f^{\ast}({\bf t}-{\bf s}) \bigr\|\, d{\bf s},\quad {\bf t}\in {\mathbb R}^{n},\ l\in {\mathbb N},
\end{align*}
which simply yields the required conclusion. 
\end{proof}

Under certain extra conditions, we can also reformulate the above results for uniformly recurrent functions defined on ${\mathbb R}^{n}$.
On the other hand, it seems that we must slightly strengthen the notion introduced in Definition \ref{braindamage12345} in order to investigate the invariance of ${\mathbb D}$-asymptotical multi-almost periodicity under the actions of ``finite'' convolution products:

\begin{defn}\label{braindamage12345678}
Suppose that the set ${\mathbb D} \subseteq {\mathbb R}^{n}$ is unbounded, and
$F : I \times X \rightarrow Y$ is a continuous function. Then we say that $F(\cdot ;\cdot)$ is 
strongly ${\mathbb D}$-asymptotically $({\mathrm R},{\mathcal B})$-multi-almost periodic, resp. strongly ${\mathbb D}$-asymptotically $({\mathrm R}_{\mathrm X},{\mathcal B})$-multi-almost periodic,
if and only if there exist an $({\mathrm R},{\mathcal B})$-multi-almost periodic function $G : {\mathbb R}^{n} \times X \rightarrow Y$, resp. an $({\mathrm R}_{\mathrm X},{\mathcal B})$-multi-almost periodic function $G : {\mathbb R}^{n} \times X \rightarrow Y$, and a function
$Q\in C_{0,{\mathbb D},{\mathcal B}}(I\times X :Y)$ such that
$F({\bf t} ; x)=G({\bf t} ; x)+Q({\bf t} ; x)$ for all ${\bf t}\in I$ and $x\in X.$

Let $I = {\mathbb R}^{n}.$ Then it is said that $F(\cdot ;\cdot)$ is strongly
asymptotically $({\mathrm R},{\mathcal B})$-multi-almost periodic, resp. strongly asymptotically $({\mathrm R}_{\mathrm X},{\mathcal B})$-multi-almost periodic, if and only if $F(\cdot ;\cdot)$ is 
strongly ${\mathbb R}^{n}$-asymptotically $({\mathrm R},{\mathcal B})$-multi-almost periodic, resp. strongly ${\mathbb R}^{n}$-asymptotically $({\mathrm R}_{\mathrm X},{\mathcal B})$-multi-almost periodic. Finally, if $X=\{0\},$ then we also say that the function 
$F(\cdot )$ is strongly
asymptotically ${\mathrm R}$-multi-almost periodic, and so on and so forth.
\index{function!strongly ${\mathbb D}$-asymptotically $({\mathrm R},{\mathcal B})$-multi-almost periodic}
\index{function!strongly ${\mathbb D}$-asymptotically $({\mathrm R}_{\mathrm X},{\mathcal B})$-multi-almost periodic}
\end{defn}
\index{function!strongly ${\mathbb D}$-asymptotically Bohr ${\mathcal B}$-almost periodic} \index{function!strongly ${\mathbb D}$-asymptotically uniformly recurrent}

Set, for brevity, $I_{{\bf t}}:=(-\infty,t_{1}] \times (-\infty,t_{2}]\times \cdot \cdot \cdot \times (-\infty,t_{n}]$ and 
${\mathbb D}_{{\bf t}}:=I_{{\bf t}} \cap {\mathbb D}$
for any ${\bf t}=(t_{1},t_{2},\cdot \cdot \cdot, t_{n})\in {\mathbb R}^{n}.$  Now we are ready to formulate the following result:

\begin{prop}\label{idio-mult}
Suppose that $(R({\bf t}))_{{\bf t}> {\bf 0}}\subseteq L(X,Y)$ is a strongly continuous operator family such that
$\int_{(0,\infty)^{n}}\|R({\bf t} )\|\, d{\bf t}<\infty .$ If $f : I \rightarrow X$ is strongly ${\mathbb D}$-asymptotically almost periodic (bounded strongly  ${\mathbb D}$-asymptotically ${\mathrm R}$-multi-almost periodic),
\begin{align}\label{bprefi}
\lim_{|{\bf t}|\rightarrow \infty,  {\bf t} \in {\mathbb D}}\int_{I_{{\bf t}}\cap {\mathbb D}^{c}}\| R({\bf t}-{\bf s})\|\, d{\bf s}=0
\end{align}
and for each $r>0$ we have
\begin{align}\label{bprefi1}
\lim_{|{\bf t}|\rightarrow \infty, {\bf t} \in {\mathbb D}}\int_{{\mathbb D}_{{\bf t}}\cap B(0,r)}\| R({\bf t}-{\bf s})\|\, d{\bf s}=0,
\end{align}
then the function 
\begin{align*}
F({\bf t}):=\int_{{\mathbb D}_{{\bf t}}}R({\bf t}-{\bf s})f({\bf s})\, ds,\quad {\bf t}\in I
\end{align*}
is strongly ${\mathbb D}$-asymptotically almost periodic (bounded strongly  ${\mathbb D}$-asymptotically ${\mathrm R}$-multi-almost periodic).
\end{prop}

\begin{proof}
We will consider only strong ${\mathbb D}$-asymptotical almost periodicity.
By definition, we have the existence of  an almost periodic function $G : {\mathbb R}^{n} \rightarrow X$ and a function
$Q\in C_{0,{\mathbb D}}(I :X)$ such that
$f({\bf t})=g({\bf t})+q({\bf t} )$ for all ${\bf t}\in I$ and $x\in X.$
Clearly, we have the decomposition
$$
F({\bf t})= \int_{I_{{\bf t}}}R({\bf t}-{\bf s})g({\bf s})\, d{\bf s} +\Biggl[ \int_{{\mathbb D}_{{\bf t}}}R({\bf t}-{\bf s})q({\bf s})\, d{\bf s}-\int_{I_{{\bf t}}\cap {\mathbb D}^{c}}R({\bf t}-{\bf s})g({\bf s})\, d{\bf s} \Biggr],\quad {\bf t}\in I.
$$
Keeping in mind Theorem \ref{krucija}, it suffices to show that the function 
$$
{\bf t} \mapsto \int_{{\mathbb D}_{{\bf t}}}R({\bf t}-{\bf s})q({\bf s})\, d{\bf s}-\int_{I_{{\bf t}}\cap {\mathbb D}^{c}}R({\bf t}-{\bf s})g({\bf s})\, d{\bf s},\quad  {\bf t}\in I
$$
belongs to the class $ C_{0,{\mathbb D}}(I :X).$ For the second addend, this immediately follows from the boundedness of function $g(\cdot)$ and condition \eqref{bprefi}. In order to show this for the first addend, fix a number $\epsilon>0$. Then there exists $r>0$ such that, for every ${\bf t}\in {\mathbb D}$ with $|{\bf t}|>r,$
we have $\|q({\bf t})\|<\epsilon.$ Furthermore, we have
$$
\int_{{\mathbb D}_{{\bf t}}}R({\bf t}-{\bf s})q({\bf s})\, d{\bf s}=\int_{{\mathbb D}_{{\bf t}}\cap B(0,r)}R({\bf t}-{\bf s})q({\bf s})\, d{\bf s}+\int_{{\mathbb D}_{{\bf t}}\cap B(0,r)^{c}}R({\bf t}-{\bf s})q({\bf s})\, d{\bf s},\quad {\bf t}\in I.
$$
Clearly, $M:=\sup_{{\bf t}\in {\mathbb D}}\|q({\bf t})\|<\infty$ and
$$
\Biggl\|\int_{{\mathbb D}_{{\bf t}}\cap B(0,r)}R({\bf t}-{\bf s})q({\bf s})\, d{\bf s}\Biggr\|_{Y} \leq M\int_{{\mathbb D}_{{\bf t}}\cap B(0,r)}\| R({\bf t}-{\bf s})\|\, d{\bf s},\quad {\bf t}\in I,
$$
so that the first addend in the above sum belongs to the class $ C_{0,{\mathbb D}}(I :X)$ due to condition \eqref{bprefi1}. This is also clear for the second addend since
$$
\Biggl\|\int_{{\mathbb D}_{{\bf t}}\cap B(0,r)}R({\bf t}-{\bf s})q({\bf s})\, d{\bf s}\Biggr\|_{Y} \leq \epsilon \int_{(0,\infty)^{n}}\|R({\bf s} )\|\, d{\bf s},\quad {\bf t}\in I.
$$
\end{proof}

If ${\mathbb D}=[\alpha_{1},\infty) \times [\alpha_{2},\infty) \times \cdot \cdot \cdot \times [\alpha_{n},\infty)$ for some real numbers $\alpha_{1},\ \alpha_{2},\cdot \cdot \cdot,\ \alpha_{n},$ then ${\mathbb D}_{{\bf t}}=[\alpha_{1},t_{1}]\times [\alpha_{2},t_{2}] \times \cdot \cdot \cdot \times [\alpha_{n},t_{n}]$ and conditions \eqref{bprefi}-\eqref{bprefi1} hold, as easily shown, which implies that
the function
$
F({\bf t})=\int^{{\bf \alpha}}_{{\bf t}}R({\bf t}-{\bf s})f({\bf s})\, ds,$ $ {\bf t}\in I
$ is strongly ${\mathbb D}$-asymptotically almost periodic, where we accept the notation
$$
\int^{{\bf \alpha}}_{{\bf t}}\cdot =\int_{\alpha_{1}}^{t_{1}}\int_{\alpha_{2}}^{t_{2}}\cdot \cdot \cdot \int_{\alpha_{n}}^{t_{n}}.
$$

\section{Examples and applications to the abstract Volterra integro-differential equations}\label{some12345}

In this section, we apply our results established so far in the analysis of existence and uniqueness of the multi-almost periodic type solutions for various classes of abstract Volterra integro-differential equations. 

We start with some illustrative examples and applications:\vspace{0.1cm}

1. Let $Y$ be one of the spaces $L^{p}({\mathbb R}^{n}),$ $C_{0}({\mathbb R}^{n})$ or $BUC({\mathbb R}^{n}),$ where $1\leq p<\infty.$ It is well known that the Gaussian semigroup\index{Gaussian semigroup}
$$
(G(t)F)(x):=\bigl( 4\pi t \bigr)^{-(n/2)}\int_{{\mathbb R}^{n}}F(x-y)e^{-\frac{|y|^{2}}{4t}}\, dy,\quad t>0,\ f\in Y,\ x\in {\mathbb R}^{n},
$$
can be extended to a bounded analytic $C_{0}$-semigroup of angle $\pi/2,$ generated by the Laplacian $\Delta_{Y}$ acting with its maximal distributional domain in $Y;$ see \cite[Example 3.7.6]{a43} for more details (recall that the semigroup $(G(t))_{t>0}$ is not strongly continuous at zero on $L^{\infty}({\mathbb R}^{n})$). Suppose now
that 
$\emptyset  \neq I'\subseteq I= {\mathbb R}^{n}$ and
$F(\cdot)$ is bounded Bohr $({\mathcal B},I')$-almost periodic, resp. bounded $({\mathcal B},I')$-uniformly recurrent. Then for each $t_{0}>0$ the function ${\mathbb R}^{n}\ni x\mapsto u(x,t_{0})\equiv (G(t_{0})F)(x) \in {\mathbb C}$
is likewise bounded Bohr $({\mathcal B},I')$-almost periodic, resp. bounded $({\mathcal B},I')$-uniformly recurrent. Towards see this, it suffices to recall the corresponding definitions and observe that, for every $x,\,\ \tau \in {\mathbb R}^{n},$ we have:
$$
\Bigl|u\bigl(x+\tau,t_{0}\bigr)-u\bigl(x,t_{0}\bigr)\Bigr| \leq \bigl( 4\pi t_{0} \bigr)^{-(n/2)}\int_{{\mathbb R}^{n}}|F(x-y+\tau)-F(x-y)|e^{-\frac{|y|^{2}}{4t_{0}}}\, dy;
$$
see also Proposition \ref{convdiaggas} which shows that for each $t_{0}>0$ the function ${\mathbb R}^{n}\ni x\mapsto u(x,t_{0})\equiv (G(t_{0})F)(x) \in {\mathbb C}$
is bounded, $({\mathrm R},{\mathcal B})$-multi-almost periodic provided that ${\mathrm R}$ is a certain collection of subsets in ${\mathbb R}^{n}$ and the function
$F(\cdot)$ is bounded, $({\mathrm R},{\mathcal B})$-multi-almost periodic
(in such a way, we have extended the conclusions obtained by S. Zaidman  \cite[Example 4, p. 32]{30} to the multi-dimensional case). Concerning this example, it should be recalled that F. Yang and C. Zhang 
have analyzed, in \cite[Proposition 2.4-Proposition 2.6]{fyang1}, the existence and uniqueness of remotely almost periodic solutions of multi-dimensional heat equations following a similar approach; we will further consider the class of multi-dimensional remotely 
almost periodic functions somewhere else.

We can  similarly clarify the corresponding results for the Poisson semigroup, which is given by
$$
(T(t)F)(x):=\frac{\Gamma((n + 1)/2)}{\pi^{(n+1)/2}}
\int_{{\mathbb R}^{n}}F(x-y)\frac{t \cdot dy}{(t^{2}+|y|^{2})^{(n+1)/2}},\quad t>0,\ f\in Y,\ x\in {\mathbb R}^{n}.
$$
Let us recall that the Fourier transform of the function
$$
x\mapsto \frac{\Gamma((n + 1)/2)}{\pi^{(n+1)/2}}
\frac{t }{(t^{2}+|x|^{2})^{(n+1)/2}},\quad x\in {\mathbb R}^{n},
$$
is given by $e^{-t|\cdot|}$ for all $t>0$ (see
\cite[Example 3.7.9]{a43} for more details).

2. Set
$$
E_{1}(x,t):=\bigl( \pi t\bigr)^{-1/2}\int^{x}_{0}e^{-y^{2}/4t}\, dy,\quad x\in {\mathbb R},\ t>0.
$$
Concerning the homogeneous solutions of
the heat equation on domain $I:=\{(x,t): x>0,\ t>0\}$, we would like to recall that F. Tr\`eves \cite[p. 433]{treves} has proposed the following formula:
\begin{align}\label{out-zxc}
u(x,t)=\frac{1}{2}\int^{x}_{-x}\frac{\partial E_{1}}{\partial y}(y,t)u_{0}(x-y)\, dy-\int^{t}_{0}\frac{\partial E_{1}}{\partial t}(x,t-s)g(s)\, ds,\quad x>0,\ t>0,
\end{align}
for the solution of the following mixed initial value problem:
\begin{align}\label{heat-prvi}
\begin{split}
& u_{t}(x,t)=u_{xx}(x,t),\ \ x>0,\ t>0; \\
& u(x,0)=u_{0}(x), \ x>0,\ \  u(0,t)=g(t), \  t>0
\end{split}
\end{align}
(for simplicity, we will not consider here the evolution analogues of \eqref{out-zxc} and the generation of various classes of operator semigroups with the help of this formula). 
Concerning the existence and uniqueness of multi-dimensional almost periodic type solution of \eqref{heat-prvi}, we will present only one result which exploits the formula \eqref{out-zxc} with $g(t)\equiv 0.$ Suppose that $0<T<\infty$ and the function $u_{0} : [0,\infty) \rightarrow {\mathbb C}$ is bounded Bohr $I_{0}$-almost periodic, resp. bounded $I_{0}$-uniformly recurrent, for a certain non-empty subset $I_{0}$ of $[0,\infty).$
Set $I':=I_{0} \times (0,T).$ If ${\mathbb D}$ is any unbounded subset of $I$ which has the property that 
$$
\lim_{|(x,t)| \rightarrow +\infty, (x,t)\in {\mathbb D}}\min\Biggl(\frac{x^{2}}{4(t+T)},t\Biggr)=+\infty,  
$$
then the solution $u(x,t)$ of \eqref{heat-prvi} is ${\mathbb D}$-asymptotically $I'$-almost periodic of type $1$, resp. ${\mathbb D}$-asymptotically $I'$-uniformly recurrent of type $1$ (see Definition \ref{nafaks123456789012345123}). In order to see that,
observe that the formula \eqref{out-zxc}, in our concrete situation, reads as follows
$$
u(x,t)=\frac{1}{2}\int^{x}_{-x}\bigl( \pi t\bigr)^{-1/2}e^{-y^{2}/4t}u_{0}(x-y)\, dy,\quad x>0,\ t>0
$$
as well as that for any $(x,t) \in I$ and $(\tau_{1},\tau_{2})\in I$ we have:
\begin{align}\label{heat-prviqw}
\begin{split}
\bigl| & u\bigl(x+\tau_{1},t+\tau_{2}\bigr)-u(x,t) \bigr|\leq \frac{\|u_{0}\|_{\infty}}{2}\int^{x+\tau_{1}}_{x}\bigl( \pi (t+\tau_{2})\bigr)^{-1/2}e^{-y^{2}/4( t+\tau_{2})}\, dy
\\&+\frac{\|u_{0}\|_{\infty}}{2}\int_{-(x+\tau_{1})}^{-x}\bigl( \pi (t+\tau_{2})\bigr)^{-1/2}e^{-y^{2}/4( t+\tau_{2})}\, dy
\\&+\frac{1}{2}\int^{x}_{-x}\Biggl|\bigl( \pi (t+\tau_{2})\bigr)^{-1/2}e^{-y^{2}/4(t+\tau_{2})}u_{0}\bigl(x+\tau_{1}-y\bigr)-\bigl( \pi t\bigr)^{-1/2}e^{-y^{2}/4t}u_{0}(x-y)\Biggr|\, dy.
\end{split}
\end{align}
The consideration for both classes is similar and we will analyze the class of ${\mathbb D}$-asymptotically $I'$-almost periodic functions of type $1$ below, only.
Let $\epsilon>0$ be given. Then we know that 
there exists $l>0$ such that for each $x_{0} \in I_{0}$ there exists $\tau_{1} \in (x_{0}-l,x_{0}+l) \cap I_{0}$ such that
\begin{align}\label{emojmarko1454321}
\bigl|u_{0}(x+\tau_{1})-u_{0}(x)\bigr|\leq \epsilon, \quad x\geq 0.
\end{align}
Furthermore, there exists a finite real number $M_{0}>0$ such that $\int_{v}^{+\infty}e^{-x^{2}}\, dx <\epsilon$ for all $v \geq M_{0}.$
Let $M>0$ be such that 
\begin{align}\label{biojefaul}
\min\Biggl(\frac{x^{2}}{4(t+T)},t\Biggr)>M_{0}^{2}+\frac{1}{\epsilon},\mbox{ provided }(x,t)\in {\mathbb D}\mbox{ and }|(x,t)|>M.
\end{align}
So, let $(x,t)\in {\mathbb D}$ and $|(x,t)|>M.$ For the first addend in \eqref{heat-prviqw}, we can use the estimates
\begin{align*}
\frac{\|u_{0}\|_{\infty}}{2}&\int^{x+\tau_{1}}_{x}\bigl( \pi (t+\tau_{2})\bigr)^{-1/2}e^{-y^{2}/4( t+\tau_{2})}\, dy 
\\ & =\pi^{-1/2}\|u_{0}\|_{\infty}\int^{(x+\tau_{1})/2\sqrt{t+\tau_{2}}}_{x/2\sqrt{t+\tau_{2}}}e^{-v^{2}}\, dv
\\& \leq \pi^{-1/2}\|u_{0}\|_{\infty}\int^{+\infty}_{x/2\sqrt{t+\tau_{2}}}e^{-v^{2}}\, dv
\\& \leq \pi^{-1/2}\|u_{0}\|_{\infty}\int^{+\infty}_{x/2\sqrt{t+T}}e^{-v^{2}}\, dv\leq \epsilon \pi^{-1/2}\|u_{0}\|_{\infty};
\end{align*}
the same estimate can be used for the second addend in \eqref{heat-prviqw}. For the third addend in \eqref{heat-prviqw}, we can use the decomposition (see \eqref{emojmarko1454321})
\begin{align*}
&\frac{1}{2}\int^{x}_{-x}\Biggl|\bigl( \pi (t+\tau_{2})\bigr)^{-1/2}e^{-y^{2}/4(t+\tau_{2})}u_{0}\bigl(x+\tau_{1}-y\bigr)-\bigl( \pi t\bigr)^{-1/2}e^{-y^{2}/4t}u_{0}(x-y)\Biggr|\, dy
\\& \leq \frac{1}{2}\int^{x}_{-x}\bigl( \pi (t+\tau_{2})\bigr)^{-1/2}e^{-y^{2}/4(t+\tau_{2})}\Bigl| u_{0}\bigl(x+\tau_{1}-y\bigr)-
u_{0}(x-y)\Bigr|\, dy,
\\& + \frac{1}{2}\int^{x}_{-x}\Biggl|\bigl( \pi (t+\tau_{2})\bigr)^{-1/2}e^{-y^{2}/4(t+\tau_{2})}u_{0}\bigl(x-y\bigr)-\bigl( \pi t\bigr)^{-1/2}e^{-y^{2}/4t}u_{0}(x-y)\Biggr|\, dy,
\end{align*}
which enables one to further continue the computation as follows:
\begin{align*}
& \leq \frac{\epsilon}{2}\int^{x}_{-x}\bigl( \pi (t+\tau_{2})\bigr)^{-1/2}e^{-y^{2}/4(t+\tau_{2})}\, dy
\\& +\frac{\|u_{0}\|_{\infty}}{2}\int^{x}_{-x}\Biggl|\bigl( \pi (t+\tau_{2})\bigr)^{-1/2}e^{-y^{2}/4(t+\tau_{2})}-\bigl( \pi t\bigr)^{-1/2}e^{-y^{2}/4t}\Biggr|\, dy
\\& \leq \epsilon \pi^{-1/2}\int_{-\infty}^{+\infty}e^{-v^{2}}\, dv
\\& +\frac{\|u_{0}\|_{\infty}}{2}\int^{x}_{-x}\Biggl|\bigl( \pi (t+\tau_{2})\bigr)^{-1/2}e^{-y^{2}/4(t+\tau_{2})}-\bigl( \pi t\bigr)^{-1/2}e^{-y^{2}/4t}\Biggr|\, dy.
\end{align*}
Applying the substitution $v^{2}=y^{2}/4t,$ we get that
\begin{align*}
&\frac{\|u_{0}\|_{\infty}}{2}\int^{x}_{-x}\Biggl|\bigl( \pi (t+\tau_{2})\bigr)^{-1/2}e^{-y^{2}/4(t+\tau_{2})}-\bigl( \pi t\bigr)^{-1/2}e^{-y^{2}/4t}\Biggr|\, dy
\\& \leq \pi^{-1/2}\|u_{0}\|_{\infty}\int^{+\infty}_{-\infty}\Biggl| \sqrt{\frac{t}{t+\tau_{2}}}e^{-v^{2}\cdot \frac{t}{t+\tau_{2}}} -e^{-v^{2}}\Biggr|\, dv.
\end{align*}
Applying the Lagrange mean value theorem for the function $x\mapsto xe^{-v^{2}x^{2}}$, $x\in [\sqrt{\frac{t}{t+\tau_{2}}},1]$ ($v\in {\mathbb R}$ is fixed), we obtain
\begin{align*}
& \pi^{-1/2}\|u_{0}\|_{\infty}\int^{+\infty}_{-\infty}\Biggl| \sqrt{\frac{t}{t+\tau_{2}}}e^{-v^{2}\cdot \frac{t}{t+\tau_{2}}} -e^{-v^{2}}\Biggr|\, dv
\\& \leq \pi^{-1/2}\|u_{0}\|_{\infty}\int^{+\infty}_{-\infty}\Biggl| \sqrt{\frac{t}{t+\tau_{2}}}-1\Biggr| \max_{\zeta \in [\sqrt{\frac{t}{t+\tau_{2}}},1]}e^{-v^{2}\zeta^{2}}\bigl(1+2\zeta^{2}v^{2}\bigr)\, dv
\\& \leq \pi^{-1/2}\|u_{0}\|_{\infty}\int^{+\infty}_{-\infty}\Biggl| \sqrt{\frac{t}{t+\tau_{2}}}-1\Biggr|e^{-\frac{t}{t+\tau_{2}}v^{2}}\bigl( 1+2v^{2}\bigr)\, dv
\\& \leq \pi^{-1/2}\|u_{0}\|_{\infty}\Biggl| \sqrt{\frac{t}{t+\tau_{2}}}-1\Biggr| \int^{+\infty}_{-\infty}e^{-\frac{M_{0}^{2}}{M_{0}^{2}+T}v^{2}}\bigl( 1+2v^{2}\bigr)\, dv.
\end{align*}
The final conclusion now follows from the estimate \eqref{biojefaul}, by observing that
\begin{align*}
\Biggl| \sqrt{\frac{t}{t+\tau_{2}}}-1\Biggr|=\frac{\tau_{2}}{t+\tau_{2}+\sqrt{t^{2}+t \tau_{2}}}\leq \frac{T}{t}.
\end{align*}

The following observation should be also made: If $u_{0} : [0,\infty) \rightarrow {\mathbb C}$ is an essentially bounded function,
then it can be easily shown that for each $x>0$ the function $t\mapsto u(x,t),$ $t\geq 0$ is bounded and continuous. Furthermore,
the calculus established above enables one to see that for each $x>0$ the function $t\mapsto u(x,t),$ $t\geq 0$ is S-asymptotically $\omega$-periodic
for any positive real number $\omega>0$ (see 
H. R. Henr\'iquez, M. Pierri, P. T\' aboas \cite{pierro} for the notion).

3. Let $\Omega=(0,\infty) \times {\mathbb R}^{n}.$ Consider the following Hamilton-Jacobi equation 
\begin{align}\label{hamilton-jacobi}
\begin{split}
& u_{t}+H(Du)=0\ \ \text{ in }\Omega, \\
& u(0,\cdot)=u_{0}(\cdot) \ \ \text{ in }{\mathbb R}^{n},
\end{split}
\end{align}
where $D$ is the gradient operator in space variable and $H$ is the Hamiltonian. If we
assume that $H\in C(\Omega)$ and $u_{0}\in BUC({\mathbb R}^{n}),$ then the Hamilton-Jacobi equation (\ref{hamilton-jacobi}) has a unique viscosity solution. This result has been proved by M. G. Crandall and P.-L. Lions in \cite[Theorem VI.2]{crandall-lions}:

\begin{thm}\label{crandalko}
Suppose that $H\in C(\Omega)$ and $u_{0}\in BUC({\mathbb R}^{n})$. Then for each finite real number $T>0$ there exists a unique function $u\in
C(\overline{\Omega}) \cap C_{b}([0,T]\times {\mathbb R}^{n} ) $ which is a viscosity solution of (\ref{hamilton-jacobi}) and satisfies
$$
\lim_{t\downarrow 0+}\bigl\| u(\cdot,t)-u_{0}(t)\bigr\|_{L^{\infty}({\mathbb R}^{n})}=0.
$$
Moreover,
\begin{align}\label{lionsinjo}
|u(t,x)-u(t,y)|\leq \sup_{\xi \in {\mathbb R}^{n}}\bigl| u_{0}(\xi)-u_{0}(\xi +y-x) \bigr|,\quad x,\ y\in {\mathbb R}^{n},\ t\geq 0.
\end{align}
\end{thm}
As a direct consequence of this result (cf. the estimate \eqref{lionsinjo}), we have that the Bohr $I'$-almost periodicity ($I'$-uniform recurrence) of function $u_{0}(\cdot)$ implies the Bohr $I'$-almost periodicity ($I'$-uniform recurrence) of viscosity solution $x\mapsto u(t,x),$ $x\in {\mathbb R}^{n}$ for every fixed real number $t\geq 0$ ($\emptyset \neq I'\subseteq {\mathbb R}^{n}$).

4. Consider
the following Hammerstein integral equation of convolution type on ${\mathbb R}^{n}$ (see e.g., \cite[Section 4.3, pp. 170-180]{cord-int} and references cited therein for more details on the subject):
\begin{align}\label{multiintegral}
y({\bf t})=g({\bf t})+ \int_{{\mathbb R}^{n}}k({\bf t}-{\bf s})F({\bf s},y({\bf s}))\, d{\bf s},\quad {\bf t}\in {\mathbb R}^{n},
\end{align}
where $g  : {\mathbb R}^{n}  \rightarrow X$ is almost periodic, $F : {\mathbb R}^{n} \times X  \rightarrow X$ is Bohr ${\mathcal B}$-almost periodic with ${\mathcal B}$ being the collection of all compact subsets of $X$, \eqref{zigzverinijeap} holds with $F=G$ and $X=Y=Z,$ 
$ k \in L^{1}({\mathbb R}^{n})$
and $ L\| k\|_{L^{1}({\mathbb R}^{n})}<1.$
Then \eqref{multiintegral} has a unique Bohr almost periodic solution. In actual fact, it suffices to apply the Banach contraction principle 
since the mapping
$$
AP\bigl({\mathbb R}^{n} :X\bigr) \ni y \mapsto g(\cdot)+ \int_{{\mathbb R}^{n}}k(\cdot-{\bf s})F({\bf s},y({\bf s}))\, d{\bf s} \in AP\bigl({\mathbb R}^{n} :X\bigr) 
$$ 
is a well defined $(L\| k\|_{L^{1}({\mathbb R}^{n}}))$-contraction, as can be easily shown by using Proposition \ref{convdiaggas}, Proposition \ref{kontinuitetap}(v), Corollary \ref{ovakorajko}
and a simple calculation.

Suppose now that ${\mathrm R}$ is a certain collection of sequences in ${\mathbb R}^{n}$ 
which satisfies that, for every sequence from ${\mathrm R},$ any its subsequence also belongs to ${\mathrm R}.$ Let ${\mathcal B}'$ be the collection of all bounded subsets of $X,$
let $F : {\mathbb R}^{n} \times X  \rightarrow X$ be $({\mathrm R},{\mathcal B}')$-multi-almost periodic, \eqref{zigzverinijeap} holds with $F=G$ and $X=Y=Z,$ 
$ k \in L^{1}({\mathbb R}^{n})$
and $ L\| k\|_{L^{1}({\mathbb R}^{n})}<1.$ 
Consider the integral equation \eqref{multiintegral}, where $g  : {\mathbb R}^{n}  \rightarrow X$ is a bounded ${\mathrm R}$-multi-almost periodic function. Denote by ${\mathrm R}_{b}({\mathbb R}^{n} : X)$ the vector space consisting of all such functions; applying Proposition \ref{2.1.10ap1}, we get that ${\mathrm R}_{b}({\mathbb R}^{n} : X)$ is a Banach space equipped with the sup-norm. Taking into account Proposition \ref{convdiaggas} and Theorem \ref{eovakoonakoap} (with ${\mathrm R}'={\mathrm R}$), the use of Banach contraction principle
enables one to see that the integral equation
\eqref{multiintegral} has a unique bounded ${\mathrm R}$-multi-almost periodic solution since
the mapping
$$
{\mathrm R}_{b}\bigl({\mathbb R}^{n} :X\bigr) \ni y \mapsto g(\cdot)+ \int_{{\mathbb R}^{n}}k(\cdot-{\bf s})F({\bf s},y({\bf s}))\, d{\bf s} \in {\mathrm R}_{b}\bigl({\mathbb R}^{n} :X\bigr) 
$$ 
is a well defined $(L\| k\|_{L^{1}({\mathbb R}^{n}}))$-contraction.\index{space!${\mathrm R}_{b}({\mathbb R}^{n} : X)$}

We can similarly analyze the existence and uniqueness of Bohr almost periodic solutions (bounded ${\mathrm R}$-multi-almost periodic solutions) of the following integral equation
\begin{align*}
y({\bf t})=g({\bf t})+ \int_{{\mathbb R}^{n}}F({\bf t},{\bf s},y({\bf s}))\, d{\bf s},\quad {\bf t}\in {\mathbb R}^{n},
\end{align*}
provided that $F : {\mathbb R}^{2n} \times X  \rightarrow X$ is Bohr ${\mathcal B}$-almost periodic with ${\mathcal B}$ being the collection of all compact subsets of $X$ (${\mathrm R}$ is a certain collection of sequences in ${\mathbb R}^{2n}$ 
which satisfies that, for every sequence from ${\mathrm R},$ any its subsequence also belongs to ${\mathrm R}$)
and there exists a costant $L\in (0,1)$ such that
$$
\| F({\bf t},{\bf s},x)-F({\bf t},{\bf s},y)\| \leq L\|x-y\|,\quad {\bf t}, \ {\bf s}\in {\mathbb R}^{n}; \ x, \ y\in X.
$$
Details can be left to the interested readers.

5. It is clear that Theorem \ref{krucija} and Theorem \ref{krucija-rrrr} can be applied in the analysis of existence of almost periodic solutions for an essentially large class of abstract
partial differential equations in Banach spaces, constructed in a little bit artificial way (even with fractional derivatives and multivalued linear operators; see \cite{nova-mono} for more details).
For example, let $A_{i}$ be the infinitesimal generator of a uniformly integrable, strongly continuous semigroup $(T_{i}(t))_{t\geq 0}$ on $X$ ($i=1,2$), and let $F : {\mathbb R}^{2}\rightarrow X$ be an almost periodic function. Define $T(t_{1},t_{2}):=T_{1}(t_{1})T_{2}(t_{2}),$ $t_{1},\ t_{2} \geq 0$ and
$$
u\bigl( t_{1},t_{2}  \bigr):=\int_{[0,\infty)^{2}}T_{1}\bigl( s_{1} \bigr)T_{2}\bigl( s_{2} \bigr)F\bigl( t_{1}-s_{1},t_{2}-s_{2} \bigr)\, ds_{1}\, ds_{2},\quad t_{1},\ t_{2}\in {\mathbb R}.
$$
Due to Theorem \ref{krucija} (see also the equation \eqref{swishtime}), we have that $u : {\mathbb R}^{2} \rightarrow X$ is almost periodic; furthermore, under certain conditions, we can use the Fubini theorem, interchange the order of integration and partial derivation, and use a well known result from the one-dimensional case, in order to see that:
\begin{align*}
u_{t_{2}}&\bigl( t_{1},t_{2}  \bigr)=\frac{\partial}{\partial t_{2}}\int_{[0,\infty)}T_{1}\bigl( s_{1} \bigr)\Biggl( \int^{\infty}_{0}T_{2} \bigl( s_{2}\bigr)F\bigl( t_{1}-s_{1},t_{2}-s_{2} \bigr)\, ds_{2}\Biggr)\, ds_{1} 
\\& =\int_{[0,\infty)}T_{1}\bigl( s_{1} \bigr)\frac{\partial}{\partial t_{2}}\Biggl( \int^{\infty}_{0}T_{2} \bigl( s_{2}\bigr)F\bigl( t_{1}-s_{1},t_{2}-s_{2} \bigr)\, ds_{2}\Biggr)\, ds_{1} 
\\&=\int_{[0,\infty)}T_{1}\bigl( s_{1} \bigr)\Biggl( A_{2}\int^{\infty}_{0}T_{2} \bigl( s_{2}\bigr)F\bigl( t_{1}-s_{1},t_{2}-s_{2} \bigr)\, ds_{2}+F\bigl( t_{1}-s_{1},t_{2} \bigr) \Biggr)\, ds_{1} 
\end{align*}
and 
\begin{align*}
u_{t_{2}t_{1}}&\bigl( t_{1},t_{2}  \bigr)=\frac{\partial}{\partial t_{1}}\Biggl(A_{2}u\bigl( t_{1},t_{2}  \bigr)+\int_{[0,\infty)}T_{1}\bigl( s_{1} \bigr)F\bigl( t_{1}-s_{1},t_{2}\bigr)\, ds_{1} \Biggr)
\\&=A_{2}u_{t_{1}}\bigl( t_{1},t_{2}  \bigr)+A_{1}\int^{\infty}_{0}T_{1}\bigl( s_{1} \bigr)F\bigl( t_{1}-s_{1},t_{2}\bigr)\, ds_{1}+F\bigl( t_{1},t_{2}  \bigr),\quad t_{1},\ t_{2} \in {\mathbb R}. 
\end{align*}

6. Consider the system of abstract partial differential equations \eqref{bounded-almost} for $(s,t)\in [0,\infty)^{2},$ accompanied with the initial condition $u(0,0)=x$ (since there is no risk for confusion, we will also refer to this problem as \eqref{bounded-almost}). In this part, we would like to note that some partial results on the existence  and uniqueness of ${\mathbb D}$-asymptotically almost periodic type solutions of this problem can be obtained by using the results from \cite[Section 2.1]{integral-dis} and some additional analyses. For simplicity, let us assume that $A$ and $B$ are two complex matrices of format $n\times n,$ $AB=BA,$ and $A,$ resp. $B,$ generate an exponentially decaying, strongly continuous semigroup 
$(T_{1}(s))_{s\geq 0},$ resp.
$(T_{2}(t))_{t\geq 0}$.
Let the functions $f_{1}(s,t)$ and 
$f_{2}(s,t)$ be continuously differentiable, let the compatibility condition $(f_{2})_{s}-Af_{2}=(f_{1})_{t}-Bf_{1}$ hold ($s,\ t\geq 0$), ${\mathbb D}:=\{(s,t)\in [0,\infty)^{2} : c_{1}s\leq t\leq c_{2}s\mbox{ for some positive real numbers }c_{1}\mbox{ and }c_{2}\},$ and let the following conditions hold true:
\begin{itemize}
\item[(i)] There exists a finite real constant $M>0$ such that $|f_{1}(v,0)|+|f_{2}(0,\omega)|\leq M,$ provided that $v,\ \omega \geq 0$ (here and hereafter, $|(z_{1},\cdot \cdot \cdot,z_{n})|:=(|z_{1}|^{2}+\cdot \cdot \cdot +|z_{n}|^{2})^{1/2}$ if $z_{i}\in {\mathbb C}$ for all $i\in {\mathbb N}_{n}$); 
\item[(ii)] The mappings $g_{i} :{\mathbb R}^{2}\rightarrow {\mathbb C}^{n}$ are continuous, bounded ($i=1,2$) and satisfy that, for every $\epsilon>0,$ there exists $l>0$ such that any subinterval $I$ of ${\mathbb R}$ of length $l>0$ contains a number $\tau \in I$ such that, for every $s,\ t\geq 0,$ we have $|g_{1}(s+\tau,t)-g_{1}(s,t)|\leq \epsilon$ and $|g_{2}(s,t+\tau)-g_{2}(s,t)|\leq \epsilon ;$
\item[(iii)]  We have that the function $q_{i}: [0,\infty)^{2} \rightarrow {\mathbb C}^{n}$ is bounded, $q_{i}\in C_{0,{\mathbb D}}([0,\infty)^{2} :{\mathbb C}^{n})$
and $f_{i}(s,t)=g_{i}(s,t)+q_{i}(s,t)$ for $(s,t)\in [0,\infty)^{2}$ and $i=1,2.$
\end{itemize}
Then there exists a unique classical solution $u(s,t)$ of \eqref{bounded-almost} (see \cite[Definition 2.13]{integral-dis}), and moreover, there exist a continuous function
$u_{ap}(s,t)$ on $[0,\infty)^{2}$ and a function 
$u_{0}\in C_{0,{\mathbb D}}([0,\infty)^{2} :{\mathbb C}^{n})$ such that $u(s,t)=u_{ap}(s,t)+u_{0}(s,t)$ for all $(s,t)\in [0,\infty)^{2}$, as well as
for every $\epsilon>0,$ there exists $l>0$ such that any subinterval $I$ of $[0,\infty)$ of length $l>0$ contains a number $\tau \in I$ such that, for every $s,\ t\geq 0,$ we have $|u_{ap}(s+\tau,t)-u_{ap}(s,t)|\leq \epsilon$
and $|u_{ap}(s,t+\tau)-u_{ap}(s,t)|\leq \epsilon.$
Keeping in mind \cite[Theorem 2.6, Theorem 2.16]{integral-dis}, all that we need is to prove that the above conclusion holds for the function
\begin{align*}
u(s,t)&=T_{1}(s)T_{2}(t)x+
T_{1}(s)\int^{t}_{0}T_{2}(t-\omega)f_{2}(0,\omega)\, d\omega
\\& +
\int^{s}_{0}T_{1}(s-v)f_{1}(v,t)\, dv
\\&=T_{1}(s)T_{2}(t)x+
T_{2}(t)\int^{s}_{0}T_{1}(s-v)f_{1}(v,0)\, dv
\\& +
\int^{t}_{0}T_{2}(t-\omega)f_{2}(s,\omega)\, d\omega
,\quad s,\ t\geq 0.
\end{align*}
Since the quantities $s,\ t$ and $|(s,t)|$ are equivalent on ${\mathbb D},$ with the meaning clear, our assumption (i) and the exponential decaying of $(T_{1}(s))_{s\geq 0}$
($(T_{2}(t))_{t\geq 0}$) together imply that:
\begin{align*}
&\lim_{(s,t)\in {\mathbb D},|(s,t)|\rightarrow \infty}\Biggl[T_{1}(s)T_{2}(t)x+
T_{1}(s)\int^{t}_{0}T_{2}(t-\omega)f_{2}(0,\omega)\, d\omega\Biggr]
\\& =\lim_{(s,t)\in {\mathbb D},|(s,t)|\rightarrow \infty}\Biggl[T_{1}(s)T_{2}(t)x+
T_{2}(t)\int^{s}_{0}T_{1}(s-v)f_{1}(v,0)\, dv\Biggr]=0.
\end{align*}
Using the decomposition ($s,\ t\geq 0$)
\begin{align*}
&\int^{s}_{0}T_{1}(s-v)f_{1}(v,t)\, dv
\\&=\int_{-\infty}^{s}T_{1}(s-v)g_{1}(v,t)\, dv
+\Biggl[\int^{s}_{0}T_{1}(s-v)q(v,t)\, dv-\int_{-\infty}^{0}T_{1}(s-v)g_{1}(v,t)\, dv\Biggr]
,
\end{align*}
the corresponding decomposition for the term $t\mapsto \int^{t}_{0}T_{2}(t-\omega)f_{2}(s,\omega)\, d\omega,$ $t\geq 0,$
our assumptions (ii)-(iii) and the argumentation contained in the proofs of \cite[Proposition 2.6.11, Proposition 2.6.13; Remark 2.6.14]{nova-mono}, the 
required conclusion simply follows. Let us note, finally, that there exists a great number of concrete situations where the above assumptions are really satisfied. Suppose, for example, that $n=1,$ $A=B=[-1],$ 
$$
f_{1}(s,t)=\sin s +\cos s+\int^{t}_{0}\frac{e^{\xi-t}}{1+\xi^{2}}d\, \xi,\quad s,\ t\geq 0
$$
and
$$
f_{2}(s,t)=\sin s +\frac{1}{1+t^{2}},\quad s,\ t\geq 0;
$$
see also \cite[Proposition 1.3.5(d)]{a43}. Then the above requirements hold.

7. Concerning the big quantity of applications and technics in the existing literature which are devoted to the study of bi-almost periodic functions and  bi-almost automorphic functions (see the references cited below and the references cited in \cite{nova-man}), we would like to note first that Z. Hu and Z. Jin \cite{hujin} have analyzed
almost automorphic mild solutions to the following nonautonomous evolution equation
\begin{align}\label{profesor}
\frac{d}{dt}\bigl[u(t)+f(t,u(t))\bigr]=A(t)\bigl[u(t)+f(t,u(t))\bigr]+g(t,u(t)),\quad t\in {\mathbb R}
\end{align}
and its generalization
\begin{align}\label{profesor1}
\frac{d}{dt}\bigl[u(t)+f(t,Bu(t))\bigr]=A(t)\bigl[u(t)+f(t,Bu(t))\bigr]+g(t,Cu(t)),\quad t\in {\mathbb R},
\end{align}
where the domains of operators $A(t)$ are not necessarily densely defined and satisfy the well known Acquistapace-Terreni conditions, the functions $f,\ g : {\mathbb R} \times X \rightarrow X$ are almost automorphic in the first argument and Lipschitzian in the second argument as well as $B$ and $C$ are bounded linear operators on $X.$ We would like to note that the statements of \cite[Lemma 17, Theorem 18]{hujin}, concerning the existence and uniqueness of almost automorphic solutions of the problem \eqref{profesor}, can be straightforwardly reformulated for almost periodicity by replacing the assumptions (H4) and (H5) with the corresponding almost periodicity assumptions as well as by assuming that the function $\Gamma(t,s)$ from the condition (H3) of this paper is $({\mathrm R},{\mathcal B})$-almost periodic with ${\mathrm R}$ being the collection of all sequuences in ${\mathbb A}:=\{(a,a) : a\in {\mathbb R}\}$ and $X\in {\mathcal B}.$ Similarly, the statements of \cite[Lemma 20, Theorem 21]{hujin}, concerning the existence and uniqueness of almost automorphic solutions of the problem \eqref{profesor1}, can be straightforwardly reformulated for almost periodicity; see also \cite[Theorem 26, Theorem 27]{zxia}, where the same comment can be given and the recent result of J. Cao, Z. Huang and G. M. N'Gu\' er\' ekata \cite[Theorem 3.6]{jcao}, where a similar modification of condition (H3) for bi-almost periodicity on bounded subsets can be made. 

We also stimulate the interested reader to reformulate the recent results of A. Ch\'avez, M. Pinto and U. Zavaleta established in the third section and the fourth section of the paper \cite{chavez3}, as well as the recent results of Y.-K. Chang, S. Zheng \cite[Theorem 4.4]{jkchang} and Z. Xia, D. Wang \cite[Theorem 3.1, Theorem 3.2]{zxia1} for almost periodicity. It seems very plausible that all these results can be reformulated for almost periodicity by replacing the notion of bi-almost automorphy (on bounded subsets) in their formulations and proofs with the notion of bi-almost periodicity (on bounded subsets). 

\subsection{Application to nonautonomous retarded functional evolution equations}\label{kamlebravo}

In this subsection, we study the asymptotic behavior of bounded solutions to the following classes of time-delay function evolution equations:
\begin{equation}
     u'(t)= A(t)u(t) + f(t, u(t-r))\quad \text{for}\; t \in \mathbb{R},  \label{Eq_1}    
\end{equation}
where $ r>0 $ is the constant time delay, $ (A(t),D(A(t))) $, $t\in \mathbb{R}$ is a family of linear closed operators defined on a Banach space $ X $. The nonlinear terms $ f: \mathbb{R}\times X \rightarrow X $ is assumed to be bounded and continuous with respect to $t$ and satisfying suitable conditions with respect to the second variable. Our aim here is to prove the existence and uniqueness of almost periodic solutions to the equation \eqref{Eq_1}. 

Let $(A(t),D(A(t))), \; t\in \mathbb{R}$ be the generators of a strongly continuous evolution family, i.e.,  $(U(t,s))_{t\geq s}  \subseteq L(X) $ such that for $ t\geq s $ the map $ (t,s)\mapsto U(t,s)$ is strongly continuous, $ U(t,s)U(s,r)=U(t,r) \quad \mbox{and}\quad U(t,t)=I \quad \mbox{for} \quad t\geq s\geq r$ such that the following linear Cauchy problem
\begin{align*}
 \left\{
    \begin{array}{ll}
        u'(t)=A(t)u(t), & t\geq s,\  t,\ s \in \mathbb{R}, \\
       u(s)=x \in X, & 
    \end{array} 
\right.  
\end{align*}
has a unique solution (at least in the mild sense) given by $ u(t):=U(t,s)x $. For more details, we refer to \cite{AquTer1,Lun} and references therein. 

Let $ {\mathrm R} $ be the collection of sequences defined in [L1].  
Let us define the mapping $ F: \mathbb{R}^{2}\times X\longrightarrow X $ by 
$$   F(t,s;x):=U(t,s)f(s,x), \quad  t,\ s \in \mathbb{R}, \ x\in X. $$

{\bf Hypotheses.} Here, we list our main hypotheses:\\
(H1) There exists $x_0 \in X$ such that $$  \sup_{t\in \mathbb{R}} \int_{-\infty}^{t} \| F(t,s;x_0) \| \, ds <\infty .$$
(H2)  There exists a function $L: \mathbb{R}^{2} \rightarrow (0,\infty)$ satisfying $ \sup_{t\in \mathbb{R}} \int_{\mathbb{R}} L(t,s)\, ds <\infty $ and
$$ \| F(t,s;x)-F(t,s;y)\| \leq L(t,s)\|x-y \|, \quad x,\ y\in X, \ t,\ s\in \mathbb{R} .$$
(H3) The mapping $(t,s;x) \in \mathbb{R}\times\mathbb{R}\times X \mapsto F(t,s;x)$ is $({\mathrm R},\mathcal{B})$-almost periodic (in the sense of [L1] above).\\

Hence, a mild solution of the equation \eqref{Eq_1} is a continuous function $ u:\mathbb{R}\longrightarrow X $ such that
\begin{equation}
u(t)=\int_{-\infty}^{t}F (t,s,u(s-r))\,  ds, \quad  t\in\mathbb{R}; \label{Mild solution form}
\end{equation}
see \cite{chavez3} for more details. Notice that, in view of (H1)-(H2), the integral formula \eqref{Mild solution form} is well-defined.  

\begin{prop}\label{Invariance result}
Assume that \emph{(H1)}-\emph{(H3)} are satisfied. Then the mapping $\Gamma: AP(\mathbb{R} : X) \rightarrow C_{b}(\mathbb{R} : X),$ given by 
\begin{equation}\label{obracun}
(\Gamma u)(t):=\int_{-\infty}^{t}F (t,s,u(s-r)) \, ds, \quad  t\in\mathbb{R},
\end{equation}
is well defined and maps $AP(\mathbb{R} : X) $ into itself.
\end{prop}

\begin{proof}
Let $ u\in AP(\mathbb{R} : X) $. Firstly, we check that the mapping $\Gamma (\cdot)$ is well-defined. In fact, from (H1)  and (H2), we have 
\begin{align*}
\| (\Gamma u)(t)\|&\leq \int_{-\infty}^{t} \| F (t,s,u(s-r)) \| \, ds \\
&\leq  \int_{-\infty}^{t} \| F (t,s,x_0) \| \, ds +  \int_{-\infty}^{t} L (t,s) \| u(s-r)-x_0 \| \, ds \\
&\leq \sup_{t\in \mathbb{R}} \int_{-\infty}^{t} \| F (t,s,x_0) \| \, ds +\left( \|u \|_{\infty}+\|x_0 \|\right)  \sup_{t\in \mathbb{R}}  \int_{-\infty}^{t} L (t,s) \, ds.
\end{align*}
 Let $ (b_n,b_n)_n \subseteq R $ be defined as in [L1], where $(b_n)_n \subseteq \mathbb{R}$ is any scalar sequence. Since $ u\in AP(\mathbb{R} : X) $, there exist a subsequence $(a_n)_n \subseteq (b_n)_n $ and a function $ u^{*}(\cdot) $ such that
 $$\lim_{n\rightarrow \infty }  u(t+a_n)=u^* (t)  \text{ uniformly in } t\in \mathbb{R}.$$
Moreover, by (H3), for all bounded subset $ B $, there exists a function $ F^* (\cdot,\cdot;\cdot)$ such that
 $$\lim_{n\rightarrow \infty }  F\bigl(t+a_n,t+a_n;x\bigr)=F^* (t,s;x) \text{ uniformly in } t,\ s\in \mathbb{R},\ x\ \in B. $$
Define the mapping $$ (\Gamma u)^{*}(t):=\int_{-\infty}^{t}F^{*} (t,s,u^{*}(s-r)) \, ds, \; t\in \mathbb{R}.$$
By Theorem \ref{eovakoonakoap},  the function $ (t,s) \mapsto F(t,s;u(s-r)) ,$ $(t,s)\in {\mathbb R}^{2}$  is $(R,\mathcal{B})$-almost periodic. That is 
 $$
\lim_{n\rightarrow \infty }  F\bigl(t+a_n,t+a_n;u(s+a_n -r)\bigr)=F^* \bigl(t,s;u^{*}(s-r)\bigr) \text{ uniformly in } t,\ s\in \mathbb{R}. 
$$
A straightforward calculation yields
  \begin{align*}
&  \|u(t+a_n)-u^* (t) \| \\ 
&=\Biggl\| \int_{-\infty}^{t+a_n}F (t+a_n ,s,u(s-r))\, ds-  \int_{-\infty}^{t}F^{*} (t,s,u^{*}(s-r)) \, ds \Biggr\| \\
& \leq \int_{0}^{+\infty} \Bigl\|  F \bigl(t+a_n ,t-s+a_n,u(t-s+a_n-r)\bigr)-  F^{*} \bigl(t,t-s,u^{*}(t-s-r)\bigr)  \Bigr\| \, ds.
  \end{align*}
  Moreover, we have
  $$  \Bigl\|  F \bigl(t+a_n ,t-s+a_n,u(t-s+a_n-r)\bigr)-  F^{*} \bigl(t,t-s,u^{*}(t-s-r)\bigr)  \Bigr\| \leq 2L(t,s),  $$ 
provided $t \geq s$ and $t,\ s \in \mathbb{R}.$
Using the dominated convergence theorem, we obtain that 
  $$  \bigl\|u(t+a_n)-u^* (t) \bigr\| \rightarrow 0 \text{ as } n\rightarrow \infty, \text{ uniformly in } t \in \mathbb{R}.$$
\end{proof}

\begin{thm} \label{Main Theorem}
Suppose that \emph{(H1)}-\emph{(H3)} hold. Then, the equation \eqref{Eq_1} has a unique mild almost periodic solution $u(\cdot),$ given by the integral formula  \eqref{Mild solution form}, provided that $   \sup_{t\in \mathbb{R}} \int_{-\infty}^{t} L(t,s)\, ds <1 $.
\end{thm}

\begin{proof}
Consider the mapping $\Gamma :AP(\mathbb{R} : X) \rightarrow C_{b}(\mathbb{R} : X)$ defined by \eqref{obracun}.
By Proposition \ref{Invariance result}, we have that  $\Gamma(AP(\mathbb{R} : X))\subseteq AP(\mathbb{R}: X)$. Moreover, for $p>1$, we have 
\begin{align*}
\Vert (\Gamma u)(t)-(\Gamma v)(t)\Vert &\leq \int^{t}_{-\infty}\Vert F(t,s;u(s-r))-F(t,s;v(s-r))\Vert \, ds
\\ &\leq \int^{t}_{-\infty}L(t,s)\Vert u(s-r)-v(s-r)\Vert \, ds
\\ & \leq  \sup_{t\in \mathbb{R}} \int_{-\infty}^{t} L(t,s)\, ds \, \cdot  \| u-v\|_{\infty}, \quad t\in \mathbb{R}.
\end{align*}
Therefore, by the Banach contraction principle, the mapping $\Gamma(\cdot)$ has a unique fixed point $u\in AP(\mathbb{R} : X)$. This proves the result.
\end{proof}

Now we will provide an illustrative application:

\begin{example}
Consider the following reaction-diffusion model with time-dependent diffusion and finite delay coefficients given by:
 \begin{align} \label{Equ App}
 \dfrac{\partial u(t,x)}{\partial t}= \delta(t)\dfrac{\partial^{2} u(t,x)}{\partial x^{2}}+ \alpha(t) u(t,x)+h(t,u(t-r,x)), \quad  t\in\mathbb{R},\;\ x\in \mathbb{R} ,
 \end{align}
where $\delta ,\ \alpha : \mathbb{R}\longrightarrow \mathbb{R}$ are almost periodic functions such that $ \alpha(t) \leq  - \tilde{\omega}   < 0$ and there exists $\delta _{0}>0$ such that $\inf_{t\in \mathbb{R}} \delta (t)\geq \delta _{0}$. The nonlinear term $h:\mathbb{R}\times \mathbb{R}\longrightarrow \mathbb{R}$  is assumed to be almost periodic with respect to $t$ and $L$-Lipschitzian with respect to the second variable with $h(t,0)\neq 0$ for all $t\in \mathbb{R}$. \\

Let  $X:=L^{2}(\mathbb{R}) $ and
$A:=\Delta $ with its maximal distributional domain.
It is well known that $ (A,D(A))  $ generates a contraction strongly continuous analytic semigroup $ (T(t))_{t\geq 0} $ on $X,$ i.e., $\|T(t) \| \leq 1$ for all $t\geq 0.$ Clearly, the operators
$$  A(t):=\delta(t)A+\alpha(t) \quad \text{with} \quad  D(A(t)):=D(A ), \; t\in \mathbb{R} $$
generate a strongly continuous evolution family given by
$$   U(t,s) :=e^{\displaystyle \int^{t}_{s}\alpha(\tau)\, d\tau} T\left( \int^{t}_{s }\delta(\tau)\, d\tau \right)  , \quad t\geq s .$$
Notice that the formula $ T\left( \int^{t}_{s }\delta(\tau)\, d\tau \right)$ for $ t\geq s  $, corresponds to the mild solution for equation \eqref{Equ App} with $\alpha,f=0$. This follows by applying the Fourier transform and the diffusion semigroup explicite formula; see e.g., \cite{a43}. Set $\omega:= \tilde{\omega}+\lambda \delta_0 >0.$

It is well known that $\sigma(A)=(-\infty,0] $. Therefore, using the spectral mapping theorem $\sigma(T(t))\setminus \{0 \}=e^{t \sigma (A)} $, $t\geq 0,$ we get that
$$   \Biggl\| T\left( \int^{t}_{s }\delta(\tau)\, d\tau \right) \varphi \Biggr\| \leq e^{-\lambda \displaystyle  \int^{t}_{s }\delta(\tau)\, d\tau} \| \varphi \|, \quad \varphi \in X   \text{ for some } \lambda \geq 0 .$$
Hence, 
\begin{align*}
  \|U(t,s) \varphi \| & \leq  e^{\displaystyle \int^{t}_{s}\alpha(\tau)\, d\tau -\lambda \displaystyle  \int^{t}_{s }\delta(\tau)\, d\tau} \|\phi \|  \\
   & \leq e^{-\left( \tilde{\omega}+\lambda \delta_0 \right)  (t-s)} \| \varphi \|=e^{-\omega  (t-s)} \| \varphi \|, \quad t\geq s, \; \varphi \in X.
\end{align*}
Furthermore, we define $f:\mathbb{R}\times X\longrightarrow X $ through
$$  f(t,\varphi)(x):=h(t,\varphi(x)), \; t, \ x \in \mathbb{R},\ \varphi \in X.  $$
It is clear that $f(\cdot,\cdot)$ is $\mathcal{B}$-almost periodic. We also have:

\begin{lem}
Hypotheses \emph{(H1)} and \emph{(H2)} are satisfied with $$L(t,s):=Le^{-\omega (t-s)}, \quad t \geq s .$$
\end{lem}

\begin{proof}
Define $F(t,s;\varphi):=U(t,s)f(s,\varphi)$ for all $ \varphi \in X, $ $ t \geq s $. Then, 
\begin{align*}
\int_{-\infty}^{t}  \| F(t,s;0) \| \, ds & \leq \int_{-\infty}^{t} \| U(t,s)f(s,0) \| \, ds  \\
& \leq  \int_{-\infty}^{t} e^{-\omega (t-s)}\| f(s,0) \| \, ds \\
&\leq \dfrac{1}{\omega} \|f(\cdot,0) \|_{\infty}, \quad t\in \mathbb{R}.
\end{align*}
Let $\varphi,\ \psi \in X$. Then the above calculation yields 
\begin{align*}
\| F(t,s;\varphi )-F(t,s;\psi) \| & \leq  L e^{-\omega (t-s)} \|\varphi -\psi\|, \quad t\in \mathbb{R}.
\end{align*}
This proves the result.
\end{proof}
To show the hypothesis (H3), it suffices to prove that $ U (\cdot,\cdot)$ is $({\mathrm R},\mathcal{B})$-almost periodic. 

\begin{prop}\label{Proposition 1 App DEK19}
The mapping $(t,s)\mapsto U(t,s)$ is $({\mathrm R},\mathcal{B})$-almost periodic. Moreover, $ F(\cdot,\cdot;\cdot) $ is $({\mathrm R},\mathcal{B})$-almost periodic.
\end{prop}

\begin{proof}
Let $ B \subseteq X $ be bounded and $(b_{k},b_{k})_{k\geq 0} \subseteq R$ be any sequence, where $(b_{k})_{k\geq 0}$ is any scalar sequence. Since $\delta \in AP(\mathbb{R})$ and $ \alpha\in AP(\mathbb{R}),$ it follows that there exists a subsequence $(a_{k})_{k\geq 0}\subseteq (b_{k})_{k\geq 0}$ and measurable functions $ \tilde{\delta} $  and $\tilde{\alpha}$ such that
\[\lim_{k} \delta(t +a_{k})=\tilde{\delta}(t)  \mbox{ uniformly in } t\in\mathbb{R},\]
and
\[\lim_{k} \alpha(t+a_{k})=\tilde{\alpha}(t)  \mbox{ uniformly in } t\in\mathbb{R}.\]
Let  $\varphi\in B.$ Define 
$$\tilde{U}(t,s)\varphi :=e^{\displaystyle \int^{t}_{s}\tilde{\alpha}(\tau)\, d\tau} T\left(\displaystyle \int^{t}_{s }\tilde{\delta}(\tau)\, d\tau \right) \varphi \mbox{ for all } t\geq s. 
$$ Thus, by the semigroup property of $  (T(t))_{t\geq 0}  $, we have 
\begin{align*}
& \Vert U\bigl(t+a_{k},s+a_{k}\bigr)\varphi-\tilde{U}(t,s)\varphi \Vert  \\
&= \Biggl\| e^{\displaystyle \int^{t+a_{k}}_{s+a_{k}}\alpha(\tau)\,  d\tau} T\left( \int^{t+a_{k}}_{s+a_{k}}\delta(\tau)\,  d\tau \right) \varphi-e^{\displaystyle \int^{t}_{s}\tilde{\alpha}(\tau)\,  d\tau} T\left(\displaystyle \int^{t}_{s }\tilde{\delta}(\tau)\, d\tau \right) \varphi \Biggr\|\\
&= \Biggl\| e^{\displaystyle \int^{t}_{s}\alpha(\tau+a_{k})\,  d\tau} T\left( \int^{t}_{s}\delta(\tau +a_{k})\, d\tau \right) \varphi-e^{\displaystyle \int^{t}_{s}\tilde{\alpha}(\tau)\,  d\tau} T\left(\displaystyle \int^{t}_{s }\tilde{\delta}(\tau)\, d\tau \right) \varphi \Biggr\| \\ 
& \leq  e^{\displaystyle \int^{t}_{s}\alpha(\tau+a_{k})\,  d\tau} \Biggl\|T\left( \int^{t}_{s} \delta(\tau+a_{k})\,  d\tau \right) \varphi- T\left(\displaystyle \int^{t}_{s }\tilde{\delta}(\tau)\,  d\tau \right) \varphi \Biggr\| \\
&+ \left( e^{\displaystyle \int^{t}_{s}\alpha(\tau+a_{k})\,  d\tau} - e^{\displaystyle\int^{t}_{s}\tilde{\alpha}(\tau)\,  d\tau}\right)  \Biggl\| T\left(\displaystyle \int^{t}_{s }\tilde{\delta}(\tau)\,  d\tau \right) \varphi \Biggr\|.
\end{align*}
Therefore, by the strong continuity of the semigroup, we obtain that 
$$ e^{\displaystyle \int^{t}_{s}\alpha(\tau+a_{k}) \,  d\tau} \Biggl\| T\left( \int^{t}_{s} \delta(\tau+a_{k})\,  d\tau \right) \varphi- T\left(\displaystyle \int^{t}_{s }\tilde{\delta}(\tau)\,  d\tau \right) \varphi \Biggr\| \rightarrow 0  \; \text{ as } k\rightarrow \infty, 
$$ 
uniformly in $t,\ s. $
Furthermore,  we have that 
\begin{align*}
& \left( e^{\displaystyle \int^{t}_{s}\alpha(\tau+a_{k})\,  d\tau} - e^{\displaystyle\int^{t}_{s}\tilde{\alpha}(\tau)\, d\tau}\right)  \Biggl\| T\left(\displaystyle \int^{t}_{s }\tilde{\delta}(\tau)\, d\tau \right) \varphi \Biggr\| \\ 
& \leq  \left( e^{\displaystyle \int^{t}_{s}\alpha(\tau+a_{k})\,  d\tau} - e^{\displaystyle\int^{t}_{s}\tilde{\alpha}(\tau)\,  d\tau}\right) \| \varphi \|  \\  &\rightarrow 0  \; \text{ as } k\rightarrow \infty, \; \text{ uniformly in } t,\ s.
\end{align*} 
So,  from Theorem 2.45, we obtain that $F(\cdot,\cdot;\cdot) $ satisfies (H3).
\end{proof}

Hence, the following result can be deduced by applying Theorem \ref{Main Theorem}:

\begin{thm}
Assume that $ L<\omega $. Then the model \eqref{Equ App} admits a unique almost periodic solution.
\end{thm}
\end{example}

\section{Appendix}\label{append}

In the appendix section, we will separately consider two intriguing topics: $n$-parameter strongly continuous semigroups and 
applications of multivariate trigonometric polynomials in approximations of periodic functions of several real variables.

\subsection{$n$-Parameter strongly continuous semigroups}  
The notion of a semigroup over topological monoid naturally generalizes the notion of 
usually considered 
one-parameter strongly continuous semigroup 
of bounded linear operators. 
This broad class of semigroups includs the semigroups defined on the set $[0,\infty)^{n},$ which are oftenly called 
multiparameter semigroups (this class of semigroups was introduced by E. Hille in 1944; see \cite{hill} and \cite{bberens}).\index{semigroup over topological monoid}\index{multiparameter semigroup}\index{multiparameter semigroup!generator}

So, let $(M,+)$ be a topological monoid with the neutral element $0$. By a semigroup over a Banach space $X$ defined over a monoid $M$ we mean any
operator-valued function $T : M \rightarrow L(X)$ such that $T(0)=I$ and $T(t+s)=T(t)T(s)$ for all $t,\ s\in M.$ A semigroup  $T : M \rightarrow L(X),$ which we also denote by $(T(t))_{t\in M},$ is called strongly continuous if and only if the mapping $t\mapsto T(t)x,$ $t\in M$ is strongly continuous at $t=0.$ In \cite{dahya}, R. Dahya has extended a well known result saying that every weakly continuous semigroup $(T(t))_{t\geq 0}$ is strongly continuous to the semigroups over topological monoids.

Further on,
if $M=[0,\infty)^{n}$ and $(T({\bf t}))_{{\bf t}\in M}$ is strongly continuous, then we denote by $T_{i}(t):=T(te_{i}),$ $t\geq 0$ the corresponding one-parameter strongly continuous semigroup ($i\in {\mathbb N}_{n}$).
Let $T_{i}(\cdot)$ be generated by $A_{i}$ ($i\in {\mathbb N}_{n}$). Then 
the tuple
$(A_{1},A_{2},\cdot \cdot \cdot,A_{n})$ is said to be the infinitesimal generator of $(T({\bf t}))_{{\bf t}\in [0,\infty)^{n}}.$
We can simply prove that $(T({\bf t}))_{{\bf t}\in [0,\infty)^{n}}$ is 
strongly continuous (uniformly
continuous) if and only if for each $i\in {\mathbb N}_{n}$ the one-parameter semigroup $(T_{i}(t))_{t\geq 0}$  is strongly continuous (uniformly continuous). A strongly continuous semigroup $(T({\bf t}))_{{\bf t}\in [0,\infty)^{n}}$ is 
always exponentially bounded in the sense that there exist two finite real constants $M\geq 1$ and $\omega>0$ such that $\|T({\bf t})\|\leq Me^{\omega |{\bf t}|}$ for all ${\bf t}\in [0,\infty)^{n};$ see e.g., Theorem 1 in the paper \cite{babalola} by V. A. Babalola, where the author has considered  generalizations of the Hille-Yosida-Phillips theorem
for abstract-parameter semigroups. Furthermore, the
following holds (\cite{bberens}, \cite{hill}):
\begin{itemize}
\item[(i)] If $i\in {\mathbb N}_{n}$ and $x\in D(A_{i}),$ then $T({\bf t})x\in D(A_{i})$ for all ${\bf t}\in [0,\infty)^{n}$ and $T({\bf t})A_{i}x=A_{i}T({\bf t})x,$ ${\bf t}\in [0,\infty)^{n};$
\item[(ii)] $\bigcap_{i\in {\mathbb N}_{n}}D(A_{i})$ is dense in $X;$
\item[(iii)] If $i,\ j\in {\mathbb N}_{n},$ then $D(A_{i}) \cap D(A_{i}A_{j}) \subseteq D(A_{j}A_{i})$ and for each $x\in D(A_{i}) \cap D(A_{i}A_{j})$ we have
$A_{i}A_{j}x=A_{j}A_{i}x.$
\end{itemize}

Set $I:=[0,T_{1}]   \times [0,T_{2}] \times \cdot \cdot \cdot \times [0,T_{n}]$ for some $(T_{1},T_{2},\cdot \cdot \cdot, T_{n})\in (0,\infty)^{n}.$ The well-posedness of the following homogeneous $n$-parameter abstract Cauchy problem\index{abstract Cauchy problem!$(ACP)$}
\[(ACP):\left\{
\begin{array}{l}
u\in C(I : X)\cap C^1(I^{\circ} :X),\\
u_{t_{i}}({\bf t})=A_{i}u({\bf t)}+F_{i}({\bf t}),\;{\bf t}\in I,\ 1\leq i\leq n,\\
u(0)=x,\quad x\in \bigcap_{i\in {\mathbb N}_{n}}D(A_{i}),
\end{array}
\right.
\]
has been analyzed by M. Janfada and A. Niknam in \cite[Theorem 2.1]{janfada}, who proved that, if $(A_{1},A_{2},\cdot \cdot \cdot,A_{n})$ is the infinitesimal generator of a strongly continuous semigroup $(T({\bf t}))_{{\bf t}\in [0,\infty)^{n}},$
then (ACP) has a unique solution $u({\bf t})=T({\bf t})x,$ $t\in I$ for all initial values $x\in \bigcap_{i\in {\mathbb N}_{n}}D(A_{i});$ a converse of this statement has been analyzed in \cite[Theorem 2.2]{janfada} (see also Theorem 2.5 in this paper, where the authors have shown a negative result about the uniqueness of solutions of the abstract Cauchy problem (ACP) as well as the paper \cite{Khanehgir0} where the authors have considered the special case $n=2$ by using the notion of a two-parameter integrated semigroup; the special case $n=2$ has been also analyzed in \cite{janfada00} and \cite{Khanehgir123}, where the authors have introduced the notion of a two-parameter $C$-regularized semigroup and the notion of a two-parameter $N$-times integrated semigroup, respectively, where the operator $C\in L(X)$ is injective and $N\in {\mathbb N}$). A Hille-Yosida type theorem for multiparameter semigroups has been analyzed by Yu. S. Mishura and A. S. Lavr\'entev in \cite{mishura}, while 
the different notions of generators of two parameter semigroups have been analyzed by Sh. Al-Sharif, R. Khalil \cite{al-sherif} and
S. Arora, S. Sharda \cite{arora}. See also an interesting generalization of R. Datko's result \cite[Theorem 1.3]{ichikawa} to nonlinear two parameter semigroups established by A. Ichikawa in \cite[Theorem 2.1]{ichikawa}. 
\index{multiparameter integrated semigroup}\index{multiparameter $C$-regularized semigroup}

To the best knowledge of the authors, the notion of an $n$-parameter $\alpha$-times integrated semigroup,  the notion of an $n$-parameter $C$-regularized semigroup and the notion of an $n$-parameter $\alpha$-times integrated $C$-regularized semigroup have not been introduced so far. The degenerate case also remains still very unexplored, even for degenerate two-parameter strongly continuous semigroups (in the existing literature, we have not been able to locate any research paper regarding this problematic).

Concerning some applications of multiparameter semigroups in the analysis of multi-dimensional almost periodic type solutions of the abstract partial differential equations and their systems, the situation is similar: there is only a few research papers devoted to the study of 
variation of parameters formulae for multiparameter semigroups (fractional multiparameter resolvent families have not been analyzed elsewhere, as well) but the existence and uniqueness of almost periodic type solutions of the abstract partial differential equations and their systems have not still been analyzed with the help of the theory of multiparameter semigroups. With regards to this intriguing topic, we want to mention only the investigation of M. Khanehgir, M. Janfada and A. Niknam
\cite{Khanehgir}, where the authors have examined the well-posedness of the following inhomogeneous abstract Cauchy problem
\[(ACP)_{2}:\left\{
\begin{array}{l}
u\in C(I : X)\cap C^1(I^{\cdot} :X),\\
u_{t_{i}}({\bf t})=A_{i}u({\bf t)}+F({\bf t}),\;{\bf t}\in I,\ i=1,2,\\
u(0)=x,\quad x\in \bigcap_{i\in {\mathbb N}_{2}}D(A_{i}),
\end{array}
\right.
\]\index{abstract Cauchy problem!$(ACP)_{2}$}
where the pair $(A_{1},A_{2})$ generates a strongly continuous semigroup\\ $(T(t_{1},t_{2}))_{t_{1}\geq 0,t_{2}\geq 0}$
on $X$.
In their analysis, the same inhomogeneity has appeared for $i=1$ and $i=2$, which forces a very unpleasant condition
$$
F_{t_{1}}\bigl( t_{1},t_{2} \bigr)-F_{t_{2}}\bigl( t_{1},t_{2} \bigr)=\bigl(  A_{1}-A_{2}\bigr)F\bigl( t_{1},t_{2} \bigr),\quad t_{1},\ t_{2}>0.
$$
Despite of this, the following formula for a solution of $(ACP)_{2}$ has been proposed:
$$
u\bigl( t_{1},t_{2} \bigr)=T\bigl( t_{1},t_{2} \bigr)x+\int^{t_{1}}_{0}T\bigl(t_{1}-t,t_{2} \bigr)F(t,0)\, dt+\int^{t_{2}}_{0}T\bigl(0, t- t_{2} \bigr)F\bigl(t_{1},0\bigr)\, dt,
$$
for any $t_{1}, \ t_{2}>0.$
The interested readers may try to formulate some results about the asymptotically almost periodic solutions of this solution provided that 
the semigroup 
$(T(t_{1},t_{2}))_{t_{1}\geq 0,t_{2}\geq 0}$
is exponentially decaying and the function
$F(\cdot;\cdot)$ is asymptotically almost periodic in a certain sense.

The multiparameter semigroups play an important role in the study of the approximations of periodic functions of several real variables (A. P. Terehin \cite{terehin}) and the study of diffusion equations in space-time dynamics (S. V. Zelik \cite{zelik}).

\subsection{Multivariate trigonometric polynomials and approximations of periodic functions of several real variables} Without any doubt, trigonometric polynomials of several real variables, sometimes also called multivariate trigonometric polynomials, presents the best explored class of almost periodic functions 
of several real variables. Multivariate trigonometric polynomials have an invaluable importance in the theory of approximations of periodic functions of several real variables, especially in the two-dimensional case. For the basic source of information about this subject, the reader may consult the research monographs \cite{dumitrescu} by B. Dumitrescu, \cite{ddung} by D.
Dung, V. Temlyakov, T. Ullrich, \cite{temlyakov} and \cite{temlyakov1} by V. Temlyakov (see also the list of references given in \cite{nova-man}; the appendix to the third chapter).

Here we will briefly explain the main resuls and ideas of papers \cite{babayev} by A. M.-B. Babayev, \cite{pfister} by L. Pfister, Y. Bresler and \cite{potts1} by L. K\"ammerer, D. Potts, T. Volkmer. If $f : {\mathbb R} \rightarrow {\mathbb R}$ belongs to the space $C_{2\pi}$ of all real continuous functions of period $2\pi,$ then it is well known that the Vallee-Poussin singular integral $V_{k}(\cdot),$ defined by\index{Vallee-Poussin singular integral}
$$
V_{k}(x):=\frac{1}{2\pi}\frac{(2k)!!}{(2k-1)!!}\int^{\pi}_{-\pi}f(t)\cos^{2k}\frac{t-x}{2}\, dt,\quad x\in {\mathbb R} \ \ (k\in {\mathbb N}),
$$
has the property that $\lim_{k\rightarrow +\infty}V_{k}(x)=f(x),$ uniformly for $x\in {\mathbb R}.$ This result of Vallee-Poussin improves the classical Weierstrass second theorem on the density of trigonometric polynomials in the spaces of continuous functions.
Two-dimensional Vallee-Poussin singular integral $V_{k,m}(\cdot),$ defined for each $x\in {\mathbb R}$ by ($k,\ m\in {\mathbb N}$):
$$
V_{k,m}(x,y):=\frac{1}{(2\pi)^{2}}\frac{(2k)!!}{(2k-1)!!}\frac{(2m)!!}{(2m-1)!!}\int^{\pi}_{-\pi}f(t,\tau)\cos^{2k}\frac{t-x}{2}\cos^{2k}\frac{\tau-y}{2}\, d\tau,
$$\index{Vallee-Poussin singular integral!two-dimensional}
has been introduced in \cite[Definition 2]{babayev}. In the same paper, the author has shown that $\lim_{k\rightarrow +\infty}\lim_{m\rightarrow +\infty}V_{k,m}(x,y)=f(x,y),$ uniformly for $(x, y)\in {\mathbb R}^{2}$ as well as that 
$V_{k,m}(x,y)$ is a 
trigonometric polynomial in variables $x$ and $y,$ for all $k,\ m\in {\mathbb N}$ (see \cite[Theorem 2]{babayev}). For proving the last fact, the author has used a lemma clarifying that the product of two trigonometric polynomials of two variables is
also the trigonometric polynomial of two variables whose order equals the sum of
order of cofactors as well as that any even
trigonometric polynomial $T (x, y),$ i.e,. a trigonometric polynomial $T (x, y)$ which satisfies that
$T (-x,-y) = T (x, y),$ $ T (-x,y) = T (x, y)$ and $T (x,-y) = T (x; y) $ identically for  $(x, y)\in {\mathbb R}^{2}$,
may be represented in the form
$$
T (x, y) = A +
\sum^{m}_{k=1}\sum^{n}_{l=1}
\bigl(a_{kl} \cos kx \cos ly + b_{kl}\cos kx + c_{kl }\cos ly\bigr),\quad (x, y)\in {\mathbb R}^{2},
$$
which does not contain the sines of multiple arcs
(see \cite[Lemma 3, Lemma 4]{babayev}). We would like to note that the obtained results continue to hold in the vector-valued case.

In \cite{pfister}, L. Pfister and Y. Bresler have investigated bounding multivariate trigonometric polynomials
and given some applications to the problems of filter bank design. Denote
$$
T_{l}^{n}:=span\Bigl\{ e^{i \langle {\bf k}, \lambda\rangle}  :  \lambda \in [0,2\pi]^{n},\ {\bf k} \in {\mathbb Z}^{n}, \|{\bf k}\|:=\sup_{1\leq i\leq n}\bigl|k_{i}\bigr| \leq l \Bigr\}\ \ (l\in {\mathbb N})
$$
and
$$
\Theta_{N}:=\Bigl \{  2\pi k/N : k=0,1,\cdot \cdot \cdot, N-1\Bigr\} \ \ (N\in {\mathbb N}).
$$

For any $N\in {\mathbb N}$ and for any real-valued trigonometric polynomial 
$$
P(   \lambda ):=\sum_{k_{1}=-l}^{l}\sum_{k_{2}=-l}^{l}\cdot \cdot \cdot \sum_{k_{n}=-l}^{l}c_{k_{1}k_{2}\cdot \cdot \cdot k_{n}}e^{i\langle k,\lambda \rangle} \in T_{l}^{n},
$$
i.e., the multivariate trigonometric polynomial $P(\cdot)$ for which\\ $c_{k_{1},k_{2},\cdot \cdot \cdot, k_{n}}=c^{\ast}_{-k_{1},-k_{2},\cdot \cdot \cdot ,-k_{n}}$ ($\|{\bf k}\|\leq l;$ the star denotes complex conjugation), we define
$$
\|P\|_{\infty}:=\max_{\lambda \in [0,2\pi]^{n}}|P(   \lambda )| \mbox{  and  }\|P\|_{N^{n},\infty}:=\max_{\lambda \in \Theta_{N}^{n}}|P(   \lambda )|.
$$
Then two well known results of  the approximation theory state that
$$
\|P\|_{\infty}\leq \|P\|_{(2l+1)^{n},\infty}\Bigl(  1+4\pi^{-1}+2\pi^{-1}\ln (2l+1)\Bigr)^{n} 
$$
and, in the one-dimensional case, 
$$
\|P\|_{\infty}\leq \frac{\|P\|_{N,\infty}}{\sqrt{1-(2l/N)}}.
$$
In \cite[Theorem 1]{pfister}, the authors have shown that the assumptions $N\geq 2l+1$ and $\alpha=2l/N$ yield the existence of a positive real constant $C_{N,l}^{n}\in [0,(1-\alpha)^{-(n/2)}]$ such that 
$$
\|P\|_{\infty}\leq C_{N,l}^{n}\|P\|_{N^{n},\infty},\quad P\in T_{l}^{n}
$$
and $C_{N,l}^{n}\|P\|_{N^{n},\infty}-\|P\|_{\infty} =O(ln/N),$ $ P\in T_{l}^{n}.$ In order to acheive their aims, the authors have used the de la Vall\'ee-Poussin kernels and the tensor products of one-dimensional Dirichlet kernels.

In \cite{potts1}, the authors have investigated certain algorithms for the approximation of multivariate periodic
functions by trigonometric polynomials, which are based on the use of a single one-dimensional
fast Fourier transform and the so-called method of
sampling of multivariate functions on rank-$1$ lattices. In their analysis, the authors have used
periodic Sobolev spaces of generalized mixed smoothness and presented some advantages of their method compared to 
the method based on the trigonometric interpolations on generalized sparse
grids. Some numerical results and tests are presented up to dimension $n = 10,$ as well.

\end{document}